\documentclass[10pt,twoside, a4paper, english, reqno]{amsart}
\usepackage{a4wide}
\usepackage{amscd}
\usepackage{amssymb}
\usepackage{amsthm}
\usepackage{graphicx}
\usepackage{amsmath,calrsfs}
\usepackage{latexsym}
\usepackage{enumitem,dsfont}

\usepackage[colorlinks]{hyperref}
\usepackage{color,graphics}
\usepackage[usenames,dvipsnames]{xcolor}
\usepackage{pdfsync}
\usepackage{ulem,cancel,pgf}
\usepackage{cancel}
\usepackage{comment}

\def\niklas#1{\textcolor{purple!100!black}{#1} }

\def\yas#1{\textcolor{brown!98!black}{#1} }

\usepackage{frcursive}

\theoremstyle{plain}
\newtheorem{theorem}{Theorem}[section]
\theoremstyle{plain}
\newtheorem{lemma}[theorem]{Lemma}
\newtheorem{prop}[theorem]{Proposition}

\newtheorem{rem}[theorem]{Remark}

\theoremstyle{definition}
\newtheorem{definition}{Definition}[section]
\newtheorem{defi}{Definition}[section]
\newtheorem{remark}{Remark}[section]

\newtheorem*{maintheorem*}{Main Theorem}
\newtheorem*{maincorollary*}{Main Corollary}

\newcommand{\dint}{\ensuremath{\displaystyle\int}}

\newcommand{\R}{\ensuremath{\mathbb{R}}}
\newcommand{\N}{\ensuremath{\mathbb{N}}}
\newcommand{\E}{\ensuremath{\mathbb{E}}}

\newcommand{\Div}{\operatorname{div}\,}

\newcommand{\eps}{\ensuremath{\varepsilon}}

\def\V{V}

\def\Vprime{V^\prime}

\def\penalisation#1{(#1)^-}

\newcommand{\erww}[1]{\mathbb{E}\left[{#1}\right]}
\newcommand{\erws}[1]{\overline{\mathbb{E}}\left[{#1}\right]}
\newcommand{\ueps}{u_{\varepsilon}}

\numberwithin{equation}{section} \allowdisplaybreaks

\title[ Stochastic  pseudomonotone parabolic    obstacle problem]
{Stochastic  pseudomonotone parabolic    obstacle problem: well-posedness $\&$ Lewy-Stampacchia's inequalities }
\date{\today}


\author[Niklas Sapountzoglou]{Niklas Sapountzoglou}
\address[Niklas Sapountzoglou]{\newline
Institute of Mathematics, Clausthal University of Technology}
\email[Niklas Sapountzoglou]{niklas.sapountzoglou@tu-clausthal.de}

\author[Yassine Tahraoui]{Yassine Tahraoui}
\address[Yassine Tahraoui]{\newline
Scuola Normale Superiore, Piazza dei Cavalieri, 7,  56126 Pisa, Italy}
\email[Yassine Tahraoui]{yassine.tahraoui@sns.it}

\author[Guy Vallet]{Guy Vallet}
\address[Guy Vallet]{\newline
Laboratory of Mathematics and 
Applications of Pau (LMAP)
UMR CNRS 5142}
\email[Guy Vallet]{guy.vallet@univ-pau.fr}

\author[Aleksandra Zimmermann]{Aleksandra Zimmermann}
\address[Aleksandra Zimmermann]{\newline
Institute of Mathematics, Clausthal University of Technology}
\email[Aleksandra Zimmermann]{aleksandra.zimmermann@tu-clausthal.de}

\allowdisplaybreaks[1]

\begin{document}

\begin{abstract}
We consider obstacle problems for nonlinear stochastic evolution equations. More precisely, the leading operator in our equation is a nonlinear, second order pseudomonotone operator of Leray-Lions type. The multiplicative noise term is given by a stochastic integral with respect to a Q-Wiener process. We show well-posedness of the associated initial value problem for random initial data on a bounded domain with a homogeneous Dirichlet boundary condition. First, we consider a singular perturbation of our problem by a higher order operator. Through the \textit{a priori} estimates for the approximate solutions of the singular perturbation, only weak convergence is obtained. This convergence is not compatible with the nonlinearities in the equation. Therefore we use the theorems of Prokhorov and Skorokhod to establish existence of martingale solutions. Then, path-wise uniqueness follows from a L1-contraction principle and we may apply the method of Gy{\"o}ngy-Krylov to obtain stochastically strong solutions. These well-posedness results serve as a basis for the study of variational inequalities and Lewy-Stampacchia's inequalities for our problem.
\end{abstract}

\maketitle
\tableofcontents


\section{Introduction}

\subsection{Functional spaces and stochastic framework}
Let us denote by $D \subset \R^d$ a Lipschitz bounded domain, $T>0$ and consider $\max(1,\frac{2d}{d+2}) <p <+\infty$.  
As usual, $p^\prime=\frac{p}{p-1}$ denotes the conjugate exponent of $p$, $V=W^{1,p}_0(D)$, the sub-space of elements of $W^{1,p}(D)$ with null trace, endowed with Poincar\'e's norm,  $H=L^2(D)$ is identified with its dual space so that, the corresponding dual spaces to $V$, $V^\prime$ is $W^{-1,p^\prime}(D)$ and the Lions-Guelfand triple $\V \hookrightarrow_d H=L^2(D)   \hookrightarrow_d \Vprime$ holds. The duality bracket for $T \in V^\prime$ and $v \in V$ is denoted $\langle T,v\rangle$.
Denote by $p^*=\frac{pd}{d-p}$ if $p<d$ the Sobolev embedding exponent and remind that 
\begin{align*}
\text{ if }p<d,\quad &V \hookrightarrow L^a(D),\ \forall a \in [1,p^*]\text{ and compact if }a \in [1,p^*),
\\
\text{ if }p=d,\quad & V \hookrightarrow L^a(D),\ \forall a <+\infty\text{ and compactly}, 
\\
\text{ if }p>d,\quad &V \hookrightarrow C(\overline D)\text{ and compactly }.
\end{align*}
Since $\max(1,\frac{2d}{d+2}) <p <+\infty$, the compactness of the embeddings hold in Lions-Guelfand triple.\\

Concerning the stochastic framework, let $(\Omega,\mathcal{F},(\mathcal{F}_t)_{t\geq 0},\mathds{P})$ be a stochastic basis with the usual assumptions on the filtration, \textit{i.e.}, $(\mathcal{F}_t)_{t\geq 0}$ is right continuous and $\mathcal{F}_0$ contains all negligible sets of $\mathcal{F}$. 
Let $(\beta^k)_k=((\beta^k_t)_{t\geq 0})_k$ a sequence of independent, real-valued Wiener processes with respect to $(\mathcal{F}_t)_{t\geq 0}$. Denote by $\Omega_T=(0,T)\times \Omega$ and $\mathcal{P}_T$ the predictable $\sigma$-field on $\Omega_T$\footnote{$\mathcal{P}_{T}:=\sigma(\{ (s,t]\times F_s \ | \ 0\leq s < t \leq T, \ F_s\in \mathcal{F}_s \} \cup \{\{0\}\times F_0 \ | \ F_0\in \mathcal{F}_0 \})$ (see \cite[p. 33]{Liu-Rock}). Then, a mapping defined on $\Omega_T$ with values in a separable Banach space $E$ is predictable if it is $\mathcal{P}_{T}$-measurable.} and by $Q_T$ the product $(0,T)\times D$.  Unless otherwise stated, for any separable Banach space $X$ and $1\leq r<\infty$,
\[L^r(\Omega_T;X):=\left\{u:\Omega\times (0,T)\rightarrow X \ \text{predictable and} \ u\in L^r(\Omega;L^r(0,T;X))\right\}. \] 
In particular,  $u$ is a random variable with values in $L^r(0,T;X)$ and the measurability of $u$ as an element of the Bochner space $L^r(\Omega\times (0,T);X)$ is understood as the predictable measurability.
We fix a separable Hilbert space $U$ such that $H\subset U$ and a non-negative symmetric trace class operator $Q:U\rightarrow U$ with $Q^{1/2}(U)=H$. Let $(e_k)_k$ be an orthonormal basis of $U$ made of eigenvectors of $Q$ with corresponding eigenvalues $(\lambda_k)_k\subset [0,\infty)$. Then, 
\[W(t):=\sum_{k=1}^{\infty}\sqrt{\lambda_k} e_k\beta^k_t,\ t\geq 0\] 
is a $(\mathcal{F}_t)$-adapted $Q$-Wiener process with values in $U$ according to \cite[Section 4.1.1]{DPZ14}. Denoting the separable Hilbert space of Hilbert-Schmidt operators from $H$ to $H$ by $HS(H)$, the stochastic It\^{o} integral of a predictable mapping $g \in L^2(\Omega_T; HS(H))$ with respect to $(W(t))_{t\geq 0}$ may be defined as in \cite[Section 4.2]{DPZ14} and will be denoted by 
\[\int_0^t g(s)\,dW(s), \ t\in [0,T].\]

We recall that an element $\xi\in L^{p'}(\Omega_T;V')$ is called \textit{non-negative} iff 
\[\erww{\int_0^T\langle \xi,\varphi\rangle_{V',V}\,dt}\geq 0\]
holds for all $\varphi\in L^p(\Omega_T;V)$ such that $\varphi\geq 0$. In this case, with a slight abuse of notation, we will often write $\xi\geq 0$.
Denote by $L^{p}(\Omega_T;V)^* =L^{p^\prime}(\Omega_T;\Vprime)^+ - L^{p^\prime}(\Omega_T;\Vprime)^+ \subset L^{p^\prime}(\Omega_T;\Vprime)$ the order dual as being the difference of two non-negative elements of $L^{p^\prime}(\Omega_T;\Vprime)$, more precisely $h\in L^{p}(\Omega_T;V)^*$ iff
$h=h^+-h^-$ with $h^+,h^- \in L^{p^\prime}(\Omega_T;\Vprime)^+$, each element being predictable \cite[H$_{5}$ and Rmk. 5]{YV}.\\ 
It is worth  recalling that $L^{p}(\Omega_T;V)^* \subsetneq L^{p^\prime}(\Omega_T;\Vprime)$, see \textit{e.g.} \cite[Example I.1]{Han-Joly}.
When there is no ambiguity, for a function $f:X\times Y \to Z$ of two variables $x\in X$ and $y\in Y$, one will sometimes denote by $f$ the function $(x,y) \mapsto f(x,y)$, $f(x)$ the function $f(x,\cdot):y\mapsto f(x,y)$ and $f(y)$ the function $f(\cdot,y):x\mapsto f(x,y)$, \textit{i.e.} $a(u,\nabla u)$ can be a notation for $x \mapsto a(x,u(x),\nabla u(x))$.
\subsection{Obstacle problem and Lewy-Stampacchia's inequalities}
Our aim is to show existence and uniqueness of a solution $(u,\rho)$ to the following obstacle problem for a nonlinear, pseudomonotone stochastic evolution equation
\begin{align}\label{I}
\begin{cases}
du-\Div a(u,\nabla u)\,ds+ \rho\,ds= f\,ds+G(u)\,dW(s) \quad & \text{in} \quad D\times  \Omega_T, \\ 
u(t,0)=u_0 & \text{in} \quad H,\  \text{a.s. in }\Omega,\\ 
u \geq \psi & \text{in} \quad  D\times  \Omega_T, \\  
u=  0   &\text{on} \quad\partial D\times \Omega_T,  \\
\langle \rho,u-\psi\rangle_{L^{p'}(\Omega_T;V'),L^p(\Omega_T;V)}  = 0 \hspace*{0.2cm} \text{and} \hspace*{0.2cm} -\rho \geq 0 \  \text{in} \ L^{p'}(\Omega_T;V')  & 
\end{cases}
\end{align}
Then we prove that the solution $(u,\rho)$ to \eqref{I} satisfies the following Lewy-Stampacchia's inequalities, in the sense of the order dual:
\begin{align}\label{LSfinal}
0 \leq & \partial_t \left( u-\displaystyle\int_0^\cdot G(u)\,dW\right) - \Div a(u,\nabla u)-f
\leq \left(f - \partial_t \left( \psi-\displaystyle\int_0^\cdot G(\psi)\,dW\right) + \Div a(\psi,\nabla \psi)\right)^- 
\end{align}
The precise notion of solution will be given in Definition \ref{def}. The precise notion of \eqref{LSfinal} will be given in Definition \ref{230310_def1}. All assumptions on the pseudomonotone operator  $A:u \mapsto -\Div a(u,\nabla u)$ and on the given data $u_0$, $\psi$ and $f$ will be given in Section \ref{Assumptions and notations}.
\subsection{State of the art}
Obstacle problems appear in the mathematical modeling of many phenomena in physics, finance, biology, \textit{etc}. For example,  the questions of  the evolution of damage in a continuous medium and American option pricing, we refer \textit{e.g.} to \cite{Bar1,Bauz,Kara,Yas} and references therein. Several studies exist in the literature concerning  obstacle problems, where in the deterministic setting, it can be formulated by using variational inequalities, see \textit{e.g.} \cite{J.L.L,GMYV,Mok-Val,J.F} and their references.
\\

Concerning  stochastic obstacle problems,
many authors have been interested in the topic, without seeking to be exhaustive, let us  mention  the work in  \cite{Hau-Par}, where the authors  studied the well-posedness of a reflected parabolic problem governed by a bounded linear operator.
The question of the semi-linear case was studied in \cite{DonatiMartinPardoux,Ras2}, and the quasilinear case have been proposed in \cite{DenisMatoussiZhang,DenisMatoussiZhangNonHom}. We mention \cite{Ben-Ras} for a differential inclusion approach on a Hilbert space and \cite{Bauz} for an Allen-Cahn type equation. Strong solutions to some  stochastic variational inequalities with monotone operators in a non-Hilbertian case have been addressed in \cite{Rascanu}. Recently, a stochastic non-linear T-monotone  obstacle  problem, in the presence of multiplicative noise, has  been studied in  \cite{YV}, where  stochastic Lewy-Stampacchia's inequalities were introduced. 	The proof is based on  \textit{ad hoc} perturbation of the stochastic term and a penalization of the constraint in the frame of Sobolev spaces. Concerning some qualitative properties and the behavior of solutions to stochastic obstacle problems, let us mention  \cite{MatoussiLDP,YassLDP} about  large deviations and \cite{YassInvarinat} on ergodicity and invariant measures. It is worth mentioning that the additive noise case requires a different approach based on parabolic potential theory to show the existence of solution, see \cite{STVZ25} for more details.\\

After the first results of H. Lewy and G. Stampacchia \cite{L-S} concerning inequalities in the context of super-harmonic problems, deterministic Lewy-Stampacchia's  inequalities have been largely studied, without trying to be exhaustive, let us cite the monograph \cite{J.F} and the papers  \cite{Mok-Mur,Mok-Val} for elliptic problems, and \cite{GigliMosconi} concerning an abstract presentation. 
The literature on Lewy-Stampacchia's inequalities is mainly aimed at elliptic problems, or close to elliptic problems and fewer papers are concerned with other type of problems. Let us cite  \cite{Rodrigues} for hyperbolic problems,  \cite{FLAVIODONATI} for parabolic problems with a monotone operator and recently \cite{GMYV} for parabolic Leray-Lions pseudomonotone problems. Concerning the stochastic Lewy-Stampacchia's inequalities, 
they have been proved first in \cite{YV} for parabolic T-monotone obstacle problems, then in \cite{IYG} for stochastic scalar conservation laws with constraint. It is worth mentioning that Lewy-Stampacchia's inequalities are irrelevant in the additive noise case.

\subsection{Our aim} In this work we provide a general result of well-posedness for the obstacle problem, associated with the corresponding Lewy-Stampacchia's inequalities, for nonlinear operators.   The operator is pseudomonotone and not monotone; this requires the use of suitable compactness methods, combined with an \textit{ad hoc} perturbation of the operator and the noise. It is worth mentioning that one should use the compactness method carefully, see \cite{Ondrejat}. Let us note that even in the deterministic case, the pseudomonotone case is more delicate, and  has been solved recently (\cite{GMYV}) while the monotone one was known for a long time (\cite{FLAVIODONATI}). It should also be noted that the addition of noise makes the study more delicate, and \cite{GMYV} proposes a more general situation. 

\par Lewy-Stampacchia's inequalities play an important part in this method. In the proof, they are the estimates of the penalty, then of the transition from a regular case to the general case. They are both a result and a tool for obtaining it. Their formulation, the sense given to the left hand side of \eqref{I}, is closely linked to the dual order assumption H$_5$. 
This makes it possible to formulate the obstacle problem as an equation with a Lagrange multiplier having a natural regularity: the dual space in which the obstacle-free problem is posed. This multiplier can also be interpreted as a control ensuring the obstacle constraint. Finally,  these inequalities give a control on the Lagrange multiplier, and  the particular case where the right-hand side in \eqref{LSfinal} vanishes leads to   a comparison (maximum) principle. \\

To the best of the author's knowledge, there doesn't exist in the literature any result of existence and uniqueness associated with corresponding Lewy-Stampacchia's inequalities for stochastic pseudomonotone obstacle problem and our aim is to propose such a result.

\subsection{Assumptions and notations}\label{Assumptions and notations}
We will consider in the sequel the following assumptions:
\begin{enumerate}
 \item[H$_1$] : $u_0\in L^2(\Omega;H)$ is $\mathcal{F}_0$-measurable with values in $H$, product measurable on $\Omega\times D$.
 \item[H$_2$] : $A:V\rightarrow V'$ is a Leray-Lions pseudomonotone operator of the form 
 \begin{align*}
 v \mapsto A(v) = -\Div a(\cdot,v,\nabla v)
 \end{align*}
satisfying the following conditions:\\
 \begin{enumerate}
  \item[H$_{2,1}$] $a:D\times\R\times\R^d\rightarrow \mathbb{R}^d$  is a Carath\'eodory function on $D \times \R^{d+1}$,
  \item[H$_{2,2}$] $a$ is monotone with respect to its last argument, \textit{i.e.}, for almost all $x\in D$, for all $\lambda\in\mathbb{R}$, for all $\xi,\eta\in\mathbb{R}^d$
  \begin{align*} 
  \left(a(x,\lambda,\xi)-a(x,\lambda,\eta)\right)\cdot (\xi - \eta)\geq 0.
  \end{align*}  
  \item[H$_{2,3}$] $a$ is coercive and bounded: there exist  constants $\bar\alpha > 0$ and $C_1^{a},C_2^{a},\bar\gamma \geq 0$, a function  $\bar h$  in $L^1(D)$  and a function  $\bar k$  in $L^{p^\prime}(D)$ and an exponent $q<p$ such that, 
  for a.e. $x\in D$, for all $\lambda\in \R$ and for all $\xi\in \R^d$, 
  \begin{align}
a(x,\lambda,\xi)\cdot \xi &\geq \bar\alpha |\xi|^{p} - \bar\gamma 
  |\lambda|^{q} + \bar h(x), \quad  \label{coercif1} \\
  |a(x,\lambda,\xi)| &\leq \bar k(x) + C_1^{a}|\lambda|^{p-1} + C_2^{a}|\xi|^{p-1}. \label{croissance1}
\end{align}
  
 \item[H$_{2,4}$] There exists a constant $C_3^{a} \geq 0$ and a non-negative function $l\in L^{p^\prime}(D)$ such that
 \begin{align*}
|a(x,\lambda_1,\xi)-a(x,\lambda_2,\xi)|\leq \left(C_3^{a}|\xi|^{p-1}+l(x)\right)|\lambda_1-\lambda_2| 
 \end{align*}
 for all $\lambda_1$, $\lambda_2\in \mathbb{R}$, for all $\xi \in \mathbb{R}^d$ and almost all $x\in D$.

 \end{enumerate}

\item[H$_3$] : 
Let $G: \Omega_T \times H \to HS(H)$ be defined by \[G(\omega,t,u)=g(\omega,t)+\sigma(\omega,t,u)\] 
with $g \in L^2(\Omega_T;HS(H))$ and $\sigma \in L^2(\Omega_T; \operatorname{Lip}_0(H; HS(H)))$ where $\operatorname{Lip}_0$ denotes the space of Lipschitz-continuous functions vanishing at $0$.
In a more concrete way, there exist constants $C_g\geq 0$, $C_{\sigma}\geq 0$ such that, if $(\mathfrak{e}_k)_{k\in\mathbb{N}}$ is an orthonormal basis of $H$, for any $k\in\mathbb{N}$, 
\begin{enumerate}
\item[H$_{3,1}$:] $g(\omega,t)(\mathfrak{e}_k) = \{ x \mapsto g_k(\omega,t,x)\}$ where $g_k \in L^2(\Omega_T;L^2(D))$  is such that
\[\sum_{k=1}^{\infty} \|g_k(\omega,t)\|_{L^2(D)}^2=\sum_{k=1}^{\infty}\Vert g(\omega,t)(\mathfrak{e}_k)\Vert_{L^2(D)}^2=\Vert g(\omega,s)\Vert^2_{HS(H)}\leq C_g\]

\item[H$_{3,2}$:] For any $u\in H$, $\sigma(\omega,t,u) (\mathfrak{e}_k)=\{x \mapsto h_k(\omega,t,x,u(x))\}$ where $h_k:\Omega_T\times D \times\mathbb{R}\rightarrow \mathbb{R}$ is a Carath\'eodory function, measurable with respect to $\mathcal{P}_T\otimes\mathcal{B}(D)$ on $\Omega_T\times D$ and continuous with respect to its last variable, satisfying $h_k(\omega,t,x,0)=0$ $d\mathds{P}\otimes dt\otimes dx$-a.e. in $\Omega_T\times D$ and
\begin{align*}
\sum_{k=1}^\infty |h_k(\omega,t,x,\lambda)-h_k(\omega,t,x,\mu)|^2\leq C_{\sigma}|\lambda-\mu|^2 
\end{align*}
for all $\mu$, $\lambda\in\mathbb{R}$, $d\mathds{P}\otimes dt\otimes dx$-a.e. in $\Omega_T\times D$. Thus, $d\mathds{P}\otimes dt$-a.e., for all $ u,v \in H$, $\sigma(\omega,t,u) \in HS(H)$ and 
\[\|\sigma(\omega,t,u)-\sigma(\omega,t,v)\|^2_{HS(H)}\leq C_{\sigma}\|u-v\|^2_{H}.\]
Moreover, for fixed $u\in H$ and any $v\in HS(H)$
\[(\omega,t) \mapsto (\sigma(\omega,t,u),v)_{HS(H)}=\sum_{k=1}^{\infty} \int_D h_k(\omega,t,x,u(x))v(\mathfrak{e}_k)\,dx\]
is predictable by Fubini's theorem, hence $(\omega,t) \mapsto \sigma(\omega,t,u)$ is weakly measurable in $HS(H)$, thus measurable by Pettis' theorem on weak/strong measurability. Note that 
\[\Omega_T\times H\ni(\omega,t,u)\mapsto \sigma(\omega,t,u)\in HS(H)\] 
is then a Carath\'eodory function and, for any predictable process $u:\Omega\times[0,T]\rightarrow H$, $\sigma(.,u)$ is also predictable  \cite[Lem. 1.2.3]{Castaing}.

\item[H$_{3,3}$:] With a slight abuse of notation, we have $\sigma \in L^2(\Omega_T;\operatorname{Lip}_0(H; HS(H)))$ in the following sense:
$\sigma(\omega,t,\cdot)\in \operatorname{Lip}_0(H; HS(H)))$, a Banach space for the norm given by the Lipschitz coefficient. This Banach space is not \textit{a priori} separable, but  Bochner measurability of $\sigma$ yields that $\sigma$ is the limit of a sequence of simple functions with values in $\operatorname{Lip}_0(H; HS(H))$ and hence $\sigma(\Omega_T)$ is contained in the closure w.r.t. the norm $\| \cdot \|_{\operatorname{Lip}_0(H; HS(H))}$ of the linear hull of the values of those simple functions and we call this space $\Sigma$. By construction, $\Sigma$ is a separable sub-Banach space of $\operatorname{Lip}_0(H; HS(H)))$ with $\sigma \in L^2(\Omega_T;\Sigma)$, where $\Sigma$ is equipped with the same norm as $\operatorname{Lip}_0(H; HS(H))$. 
\end{enumerate}
\noindent \\
Recall that the properties in H$_3$ do not depend on the choice of the ONB of $H$ and that it is possible to choose $\mathfrak{e}_k:= Q^{1/2}(e_k)$.

\item[H$_4$] :  Let $\psi \in  L^p(\Omega_T;V)\cap L^2(\Omega_T;H)$ such that 
\[Y:=\partial_t \left( \psi-\displaystyle\int_0^\cdot G(\psi)\,dW\right)\in L^{p^\prime}(\Omega_T;\Vprime)\]
and $\psi(0,\cdot)=\psi_0\in H$. 
\begin{remark}\label{conti-obst}
From H$_4$ it follows that $\psi\in C([0,T];H)$ is an adapted stochastic process and 
\[\erww{\sup_{t\in [0,T]} \|\psi\|^2_{L^2(D)}}<+\infty,\] 
see \cite[Thm 4.2.5]{Liu-Rock}.
\end{remark}

\item[H$_5$] :  There exist predictable $f\in L^{p^\prime}(\Omega_T;V^\prime)$ and $h^+,h^- \in L^{p^\prime}(\Omega_T;\Vprime)^+\subset L^{p^\prime}(\Omega_T;\Vprime)$ such that
 \begin{align}\label{230308_02}
h := f - \partial_t \left(\psi - \int_0^{\cdot} G(\psi)\,dW\right)  +\Div a(\psi,\nabla \psi)=h^+-h^- \in L^{p}(\Omega_T;V)^*.
 \end{align}

\item[H$_6$] : The initial value $u_0$ satisfies the constraint, \textit{i.e.} $u_0 \geq \psi(0)$ $d\mathds{P}\otimes dx$-a.s. in $\Omega\times D$.
\end{enumerate}

\begin{remark}\label{remark}
Taking into account Assumptions H$_4$ and H$_5$, it's worth noticing that $\psi$ solves the following stochastic problem
\begin{align}\label{PDEforpsi}
d\psi-\operatorname{div}\,a(\psi,\nabla \psi)\,dt =(f-h)\,dt + G(\psi)\,dW
\end{align}
and the obstacle can be understood as a constraint in the coupling of stochastic PDEs.
\end{remark}

\subsection{Notion of solution} \label{section2}
\begin{definition}\label{Def_K}
Denote by $K$ the convex set of admissible functions
\[K=\left\{v \in L^p(\Omega_T;V), \ v(\omega,t,x)\geq \psi(\omega,t,x) \ d\mathds{P}\otimes dt\otimes dx\text{-a.e. in }  \Omega_T\times D\right\}.\]
\end{definition}
Let us introduce the concept of a solution for Problem \eqref{I}.
\begin{definition}\label{def}
For fixed $\psi\in L^p(\Omega_T;V)\cap L^2(\Omega_T;H)$, the pair $(u,\rho)$ is a solution to Problem \eqref{I} iff:
 \begin{itemize}
\item $u: \Omega_T \to L^2(D)$ is a predictable process with
$u\in L^2(\Omega;C([0,T];L^2(D)))$.
  \item $u\in L^p(\Omega_T;V)$, $u(0,\cdot)=u_0$ and $u\geq \psi$, \textit{i.e.}, $u\in K$.
  \item  $\rho \in L^{p^\prime}(\Omega_T; V^\prime)$ with 
  \begin{align}\label{230530_01}
  - \rho\in L^{p^\prime}(\Omega_T;V^\prime)^+ \ \text{and} \ \erww{\int_0^T\langle \rho,u-\psi\rangle_{V',V}\,dt}=0
  \end{align}
 \item It holds
\begin{align}\label{221013_01}
u(t)+\int_0^t \rho\,ds-\int_0^t \Div a(u,\nabla u)\,ds=u_0+\int_0^t G(u)\, dW(s)+\int_0^t f\,ds
\end{align}
for all $t\in[0,T]$, $\mathds{P}$-a.s. in $\Omega$.
\end{itemize}
\end{definition}
\begin{remark}\label{Prog-meas}
Since the embedding $V \hookrightarrow H$ is continuous, $u$ is  equally a predictable process  with values in $H$ or in $V$ (thanks to Kuratowski's theorem  \cite[Th. 1.1 p. 5]{Vakh}). Moreover, the first point of Definition \ref{def} could be given in a more general way by replacing $C([0,T],L^2(D))$ by $C([0,T];W^{-1,p'}(D))\cap L^\infty(0,T;L^2(D))$. This assumption yields $u(\omega)\in C_w([0,T];L^2(D))$\footnote{$C_w([0,T];L^2(D))$	denotes	the	Bochner space of weakly continuous functions with values in $L^2(D).$}, then, since $u$ satisfies \eqref{221013_01}, we have back $u(\omega)\in C([0,T];L^2(D))$ according to \cite[Thm 4.2.5]{Liu-Rock}.
\end{remark}
\begin{remark}
Condition \eqref{230530_01} can be understood as a minimality condition on $\rho$ in the sense that $\rho$ vanishes on the set $\{u>\psi\}$. Moreover, \eqref{230530_01} implies that, for all $v \in K$,  
\[\erww{\int_0^T\langle \rho,u-v\rangle_{V',V}\,dt}  \geq 0.\]
\end{remark}
\begin{remark}
 Note that  Problem \eqref{I} can be written in the equivalent form (see \cite[p.$7-8$]{BARBU}):
\begin{align*} 
\partial I_{K}(u)\ni f-\partial_t \left(u-\int_0^\cdot G(u)\,dW\right) + \Div a(u,\nabla u)
 \end{align*}
 where $\partial I_{K}(u)$ represents the sub-differential of $I_{K}: L^p(\Omega_T;V)\rightarrow \bar{\mathbb{R}} $ defined as 
 \begin{align*}
 I_{K}(u)=\left\{
 \begin{array}{l}
 0, \quad\quad  u \in K,\\ [1.2ex]
 +\infty, \quad u \notin K,
 \end{array}
 \right.
 \end{align*}
 and 
\begin{align*} 
\partial I_{K}(u)=N_{K}(u)=
\left\{y \in L^{p^\prime}(\Omega_T;V^\prime); \ \E\int_0^T\langle y,u-v\rangle_{V',V}dt \geq 0, \ \text{for all} \ v \in K \right\}.
\end{align*} 
\end{remark}

\begin{definition}\label{230310_def1}
Let H$_1$-H$_6$ be satisfied. The solution $(u,\rho)$ of Problem \eqref{I} satisfies Lewy-Stampacchia's inequalities \eqref{LSfinal}, iff $0\leq -\rho\leq h^{-}$ in the sense that for all $\varphi\in L^p(\Omega_T;V)$ such that $\varphi\geq 0$, 
\[\erww{\int_0^T\langle h^{-}+\rho,\varphi\rangle_{V',V}\,dt}\geq 0  \text{\quad  and \quad } \erww{\int_0^T\langle -\rho,\varphi\rangle_{V',V}\,dt}\geq 0.\]
\end{definition}

\subsection{Main results and outline}
Our aim is to prove the following Theorem
\begin{theorem}\label{MTh}
Under assumptions H$_1$-H$_6$, there exists a unique solution $(u,\rho)$ to \eqref{I} in the sense of Definition \ref{def}. Moreover, $(u,\rho)$ satisfies the Lewy-Stampacchia's inequalities \eqref{LSfinal} in the sense of Definition \ref{230310_def1}.
\end{theorem}
In order to prove this result, we will first consider a singular perturbation of our problem by a higher order operator in Section \ref{S1}. This well-posedness result serves as a basis for the study of solutions to \eqref{I} when $h^{-}\in L^{\alpha}(\Omega_T;L^{\alpha}(D))$ for $\alpha=\max(2,p')$ in Section \ref{sec-obstacle-pb}. In Section \ref{sec-LS-regular} we prove Lewy-Stampacchia's inequalities in the regular case under the additional assumption $\partial_t h^{-}\in L^2(\Omega_T;H)$. Then, in Section \ref{GC} we prove the result in the general case.

\section{Penalization $\&$ Existence of approximate solution}\label{S1}

\subsection{Penalization}
Let $\varepsilon >0$ and consider the following approximation problem:
\begin{align}\label{penalization}
\begin{aligned}
u_\eps(t)-\dfrac{1}{\varepsilon}\int_0^t (u_\eps-\psi)^-ds-\int_0^t \Div\tilde{a}(u_\eps,\nabla u_\eps) \,ds &=u_0+\int_0^t \tilde{G}(u_\eps)\,dW(s)+\int_0^t f\,ds,\\[2ex]
u_\eps(0)&=u_0, 
\end{aligned}
\end{align}
where $\tilde{G}(u)=G(\max(u,\psi))$, $\tilde{a}(\omega,t,x,u_\eps,\nabla u_\eps)=a(x,\max(u_\eps,\psi(\omega,t)),\nabla u_\eps) $. The idea of the perturbation of 
$G$  (resp. $a$ ) is to have formally an additive stochastic source  
 with a monotone operator 
on the free-set where the constraint is violated.

\subsection{A singular perturbation}\label{section-2.2}

\subsubsection{A higher order problem}\label{higher order problem}

For $\delta>0$, denote by $\mathcal{V}=W^{m,\nu}_0(D)$ with $\nu>\max\{p,2,2p(p-1)\}$ and let $m\in\mathbb{N}$ be chosen such that $H^m_0(D)$ has a continuous injection in $W_0^{1,2p}(D)$ and $L^\infty(D)$. Then, there exists a constant $C_e\geq 0$ such that for all $u\in W^{m,\nu}_0(D)$,
\[\left(\Vert u\Vert_{H^m_0}+\Vert u\Vert_{W^{1,2p}_0}+\Vert u\Vert_{W^{1,p}_0}+\Vert u\Vert_{\infty}+\Vert u\Vert_{2}\right)\leq C_e\Vert u\Vert_{W^{m,\nu}_0}.\]
For $\delta>0$ we define
\begin{align*}
&A_\delta : \Omega_T\times \mathcal{V} \to \mathcal{V}^\prime,\ (\omega,t,u) \mapsto A_\delta(\omega,t,u), 
\\
& 
A_\delta(\omega,t,u): v \mapsto \int_D \left( a(\max(\psi(\omega,t),u),\nabla u)\cdot\nabla v  - \frac1\varepsilon  \penalisation{u-\psi(\omega,t)}v\right)\,dx 
+\delta b(u,v) -\langle  f(\omega,t),v \rangle_{V',V}
\end{align*}
where, for $u,v \in\mathcal{V}$,
\begin{align}\label{230104_01}
b(u,v)=\langle \partial J(u),v\rangle =\sum_{|\alpha|\leq m}\int_{D}(1 + |D^\alpha u|^{\nu-2})D^\alpha u D^\alpha v \,dx
\end{align}
denotes the variational formulation of the maximal monotone operator on $\mathcal{V}$ associated with the G\^ateaux derivative $\partial J$ of $J=v \mapsto \frac{1}{\nu}\|v\|^{\nu}_{W^{m,\nu}_0(D)} + \frac12\|v\|^2_{H^m_0(D)}$.
Thanks to the assumptions on $a$ and $\psi$, $A_\delta$ is, on $\mathcal{V}$,  progressively measurable and hemicontinuous. 
In the following, we will often use the notation 
\[\langle A_\delta(\omega,t,u),v \rangle_{\mathcal{V}',\mathcal{V}} = \langle A_0(\omega,t,u),v \rangle_{\mathcal{V}',\mathcal{V}}+\delta b(u,v)\] 
where 
\begin{align}\label{230104_02}
\langle A_{0}(\omega,t,u),v \rangle_{\mathcal{V}',\mathcal{V}}:=\int_D \left( a(\max(\psi(\omega,t) ,u),\nabla u)\cdot\nabla v  - \frac1\varepsilon  \penalisation{u-\psi(\omega,t) }v \right) \,dx  -\langle  f(\omega,t),v \rangle_{V',V}
\end{align}
for all $u,v\in \mathcal{V}$. Concerning the coercivity property, by \eqref{coercif1} we have
\begin{align*}
\int_D a(\max(\psi(\omega,t) ,u),\nabla u)\cdot\nabla u \,dx\geq \int_D \left( \bar{\alpha}|\nabla u|^p-\bar{\gamma}|\max(\psi(\omega,t) ,u)|^q-\bar{h}(x) \right) \,dx.
\end{align*}
Using Young's inequality with $\rho:=\frac{p}{q}>1$, for any constant $C_1>0$ we get
\begin{align*}
-\bar{\gamma}|\max(\psi(\omega,t) ,u)|^q\geq -\frac{1}{C_1^{\rho'/\rho}\rho'}\bar{\gamma}^{\rho'}-\frac{C_1}{\rho}|\max(\psi(\omega,t) ,u)|^p.
\end{align*}
Hence, by Poincar{\'e}'s inequality,
\begin{align}\label{221020_01}
\begin{aligned}
&\int_D a(\max(\psi(\omega,t) ,u),\nabla u)\cdot\nabla u \,dx\\
\geq& \bar{\alpha}\Vert\nabla u\Vert^p_{L^p} -\Vert \bar{h}\Vert_{L^1}-\frac{|D|}{C_1^{\rho'/\rho}\rho'}\bar{\gamma}^{\rho'}-\frac{C_1}{\rho}2^p\Vert\psi(\omega,t)\Vert_{L^p}^p-\frac{C_1}{\rho}2^pC_{\text{Poincar\'e}}\Vert\nabla u\Vert^p_{L^p}\\
&=\left(\bar{\alpha}-\frac{C_1}{\rho}2^pC_{\text{Poincar\'e}}\right)\Vert\nabla u\Vert^p_{L^p}-\Vert\bar{h}\Vert_{L^1}-\frac{|D|}{C_1^{\rho'/\rho}\rho'}\bar{\gamma}^{\rho'}-\frac{C_1}{\rho}2^p\Vert\psi(\omega,t)\Vert_{L^p}^p.
\end{aligned}
\end{align}
Using Young's inequality, for any constant $C_2>0$ we obtain
\begin{align*}
\langle f(\omega,t),u\rangle_{V',V}\leq \Vert f(\omega,t)\Vert_{V'}\Vert u\Vert_V
\leq\frac{1}{C_2p'}\Vert f(\omega,t)\Vert_{V'}^{p'}+\frac{C_2}{p}\Vert\nabla u\Vert_{L^p}^p.
\end{align*}
Together with the inequality
\begin{align}\label{221019_01}
-\penalisation{u-\psi }u \geq \frac1{2} \left(|(u-\psi )^-|^{2}- |\psi |^{2}\right),
\end{align}
we arrive at
\begin{align*}
\langle A_{0}(\omega,t,u),u \rangle_{\mathcal{V}',\mathcal{V}}\geq& \left(\bar{\alpha}-\frac{C_1}{\rho}2^pC_{\text{Poincar\'e}}-\frac{C_2}{p}\right)\Vert\nabla u\Vert_{L^p}^p+\frac{\|(u-\psi)^-\|_{L^{2}(D)}^{2}- \|\psi \|_{L^{2}(D)}^{2}}{2\varepsilon}\\
&-K_1\Vert\psi(\omega,t)\Vert_{L^p}^p-K_2\Vert f(\omega,t)\Vert_{V'}^{p'}-K_3
\end{align*}
where
\begin{align*}
K_1:=\frac{C_1}{\rho}2^p, \quad K_2:=\frac{1}{C_2p'}, \quad K_3:=\Vert h\Vert_{L^1}+\frac{|D|}{C_1^{\rho'/\rho}\rho'}\bar{\gamma}^{\rho'}. 
\end{align*}
In particular, choosing $C_1$, $C_2$ such that
\[K_4:=\left(\bar{\alpha}-\frac{C_1}{\rho}2^pC_{\text{Poincar\'e}}-\frac{C_2}{p}\right)>0,\]
for any $u\in \mathcal{V}$
\begin{align*}
&\langle A_\delta(\omega,t,u), u \rangle_{\mathcal{V}',\mathcal{V}} \\ 
\geq&\delta \|u\|^{\nu}_{W^{m,\nu}_0}  + \delta \|u\|^2_{H^m_0}  + K_4 \|\nabla u\|_{L^p}^p - K_1 \|\psi(\omega,t) \|^p_{L^p} - \frac{1}{2\eps} \|\psi(\omega,t)\|^2_{L^2} - K_2\|f(\omega,t) \|^{p'}_{V'}-K_3\\
\geq& \delta \|u\|^{\nu}_{W^{m,\nu}_0}+\tilde{h}(\omega,t)
\end{align*}
where $\tilde h \in L^1(\Omega_T)$ is given by
\[\tilde{h}(\omega,t)=-K_1 \|\psi(\omega,t) \|^p_{L^p} - \frac{1}{2\eps} \|\psi(\omega,t)\|^2_{L^2} - K_2\|f(\omega,t) \|^{p'}_{V'}-K_3\]
is predictable.
For the boundedness, by \eqref{croissance1}, recalling that $W^{1,p}_0(D)\hookrightarrow L^2(D)$  and using Poincar\'{e}'s inequality, there exist constants $C\geq 0$, $C^{\varepsilon}_1\geq 0$ such that for any $u,v\in V$, 
\begin{align}\label{221220_01}
{\small \begin{aligned}
&|\langle A_{0}(\omega,t,u),v \rangle_{\mathcal{V}',\mathcal{V}}|=\left |\int_D \left(a(\max(\psi(\omega,t),u),\nabla u)\cdot\nabla v  - \frac1\varepsilon  \penalisation{u-\psi(\omega,t)}v \right) \,dx  -\langle  f(\omega,t),v \rangle_{V',V}\right|
\\ \leq &
C_1^{\varepsilon}\left(\|\bar k\|_{L^{p'}\!(D)} + \|u\|_{L^p\!(D)}^{p-1}+\|\psi(\omega,t)\|_{L^p\!(D)}^{p-1}+ \|\nabla u \|_{L^p\!(D)}^{p-1}+\|f(\omega,t)\|_{V'}+\|u\|_{L^{2}\!(D)}+\|\psi(\omega,t) \|_{L^{2}\!(D)}\right)
\|\nabla v\|_{L^p\!(D)}
\\ \leq& C_1^\varepsilon \left((C+1)\Vert u\Vert_{W^{1,p}_0}^{p-1}+C\Vert u\Vert_{W^{1,p}_0}+\bar{h}(\omega,t)\right) \Vert v\Vert_{W^{1,p}_0}. 
\end{aligned}
}\end{align}
Where $\bar{h}:\Omega_T\rightarrow\mathbb{R}$ defined by 
\[\bar{h}(\omega,t)=\Vert \psi(\omega,t)\Vert_{L^p}^{p-1}+\Vert \psi(\omega,t)\Vert_{L^2}+\Vert f(\omega,t)\Vert_{V'}+\|\bar k\|_{L^{p'}(D)}\] 
is an integrable predictable function. Then, using the continuous embeddings $W^{1,p}_0(D)\hookrightarrow L^2(D)$ and $W^{m,\nu}_0(D)\hookrightarrow W^{1,p}_0(D)$ we find a constant $C_2^{\varepsilon}\geq 0$ such that 
\begin{align*}
\langle A_0(\omega,t,u), v \rangle_{\mathcal{V}',\mathcal{V}}\leq C_2^{\varepsilon}\left((\Vert u\Vert_{W^{m,\nu}_0}+1)^{\nu-1}+\Vert u\Vert_{W^{m,\nu}_0}+\bar{h}(\omega,t)\right)\Vert v\Vert_{W^{m,\nu}_0} 
\end{align*}
consequently, 
\begin{align*}
\|A_\delta(\omega,t,u)\|_{\mathcal{V}'} \leq C_2^{\varepsilon}\left((\Vert u\Vert_{W^{m,\nu}_0}+1)^{\nu-1}+\Vert u\Vert_{W^{m,\nu}_0}+\bar{h}(\omega,t)\right)+\delta\left(\Vert u\Vert_{H^m_0}+\Vert u\Vert_{W^{m,\nu}_0}^{\nu-1}\right).
\end{align*}
Now, recalling the embedding $H^m_0(D)\hookrightarrow W^{m,\nu}_0(D)$ and that $\nu>\max\{p,2\}$ it follows that
\begin{align*}
\|A_\delta(\omega,t,u)\|_{\mathcal{V}'} \leq& C_2^{\varepsilon}\left(2(\Vert u\Vert_{W^{m,\nu}_0} +1)^{\nu-1}+\bar{h}(\omega,t)\right)+\delta (C_e+1)\left(\Vert u\Vert_{W^{m,\nu}_0}+1\right)^{\nu-1}\\
\leq& K_5 (\Vert u\Vert_{W^{m,\nu}_0}^{\nu-1}+1)+C_2^{\varepsilon}\bar{h}(\omega,t)
\end{align*}
where $K_5=2^{\nu-1}(2C_2^{\varepsilon}+\delta (C_e+1))$. Now, it follows that 
\begin{align*}
\|A_\delta(\omega,t,u)\|^{\nu'}_{\mathcal{V}'} 
\leq & 2^{\nu'} K_5^{\nu'}\left(\Vert u\Vert_{W^{m,\nu}_0}^{\nu}+1\right)+(C_2^{\varepsilon})^{\nu'}\bar{h}(\omega,t)^{\nu'}
\end{align*}
and, since $\nu'<\min(p',2)$, it follows that $\bar{h}^{\nu'}\in L^1(\Omega_T)$.
In order to apply the well-posedness result of \cite[Section 5.1]{Liu-Rock}, let us verify that $A_\delta$ is a local monotone operator.  Set, $u,v \in \mathcal{V}$, then, thanks to the monotony of the penalization procedure, H$_{2,2}$ and then $H_{2,4}$, there exists $c=c_{\nu}>0$ such that
\begin{align*}
&\langle A_\delta u - A_\delta v,u-v \rangle_{\mathcal{V}',\mathcal{V}} \\ 
\geq& c\delta \|u-v\|^{\nu}_{W^{m,\nu}_0}+\delta\|u-v\|^2_{H^m_0(D)} + \int_D \left(a(\max(\psi,u),\nabla u)-a(\max(\psi,v),\nabla v)\right)\cdot\nabla (u-v) \,dx\\ 
&- \frac1\varepsilon \int_D \left(\penalisation{u-\psi}-\penalisation{v-\psi}\right)(u-v) \,dx\\ 
\geq& c\delta \|u-v\|^{\nu}_{W^{m,\nu}_0}+ \delta \|u-v\|^2_{H^m_0} + \int_D \left(a(\max(\psi,u),\nabla u)-a(\max(\psi,v),\nabla u)\right)\cdot\nabla (u-v)  \,dx\\ 
\geq& c\delta \|u-v\|^{\nu}_{W^{m,\nu}_0}+ \delta \|u-v\|^2_{H^m_0} 
- \int_D \left(C_3^{a}|\nabla u|^{p-1}+l\right)|\max(\psi,u)-\max(\psi,v)| 
|\nabla (u-v)| \,dx.
\end{align*}
Using again that $H^m_0(D)$  has a continuous injection in $V$, $W_0^{1,2p}(D)$ and $L^\infty(D)$, one gets that
\begin{align*}
&\left|\int_D \left(C_3^{a}|\nabla u|^{p-1}+l\right)|\max(\psi,u)-\max(\psi,v)| 
|\nabla (u-v)|  \,dx\right| 
\\ 
\leq& 
\Vert C_3^{a}|\nabla u|^{p-1}+l\|_{L^{p'}}\|u-v\|_{L^{2p}} 
\|\nabla (u-v)\|_{L^{2p}}
\\  
\leq& C(\|\nabla u\|_{L^p}^{p-1}+\|l\|_{L^{p'}})\|u-v\|^{\frac1p}_{L^2} \|u-v\|^{\frac1{p'}}_{L^\infty}
\|\nabla (u-v)\|_{L^{2p}}
\\ 
\leq& 
\frac\delta4 \|u-v\|^{(1+\frac1{p'})\frac{2p}{2p-1}}_{H^m_0}  + C_\delta\left(\|\nabla u\|_{L^p}^{2p(p-1)}+1\right)\|u-v\|^{2}_{L^2}
\\ 
\leq& \frac\delta4 \|u-v\|^2_{H^m_0}  + C_\delta\left(\|\nabla u\|_{L^p}^{2p(p-1)}+1\right)\|u-v\|^{2}_{L^2}.
\end{align*}
In conclusion, 
\begin{align*}
\langle A_\delta u - A_\delta v,u-v \rangle 
\geq C\delta \|u-v\|^{\nu}_{W^{m,\nu}_0}+\frac{\delta}{4} \|u-v\|^2_{H^m_0} 
- C_\delta \left(\pi(u)+1\right)\|u-v\|^{2}_{L^2}
\end{align*}
where $\pi :u\mapsto\pi(u)=\|\nabla u\|_{L^p}^{2p(p-1)}$ is a measurable hemicontinuous function and locally bounded in $\mathcal{V}$.\\
Thus, taking into account H$_{3,1}$ and H$_{3,2}$, we have shown that $A_{\delta}$ satisfies the conditions $(H1)$, $(H2')$, $(H3)$ and $(H4')$ of \cite[Section 5.1]{Liu-Rock} with $\beta=0$ and $\alpha=\nu$. Recalling that $\nu>2p(p-1)$ we get
\begin{align*}
\pi(u)\leq C(\Vert u\Vert_{\mathcal{V}}^{\nu}+1)
\end{align*}
for all $u\in \mathcal{V}$.
Consequently, from \cite[Theorem 5.1.3]{Liu-Rock}, for any fixed $\delta>0$ and for any initial value $u_0$ satisfying H$_1$ there exists a unique predictable solution $u_{\delta}^{\varepsilon} = u_{\delta}$ in $L^p(\Omega\times(0,T);\mathcal{V}) \cap L^2(\Omega;C([0,T],H))$ to
\begin{align}\label{Pblme_delta}
du_\delta + \delta b(u_\delta,\cdot)\,dt - \Div a(\max(\psi,u_\delta),\nabla u_\delta) \,dt -\frac1\varepsilon\penalisation{u_\delta-\psi}\,dt = f \,dt + G(\max(\psi,u_\delta))\,dW.
\end{align}
\subsubsection{Estimates uniformly in $\delta$ for fixed $\varepsilon$} \label{estimates_delta}
Let us fix $\varepsilon >0$. In order to not overload the notation in this sub-subsection, we omit all dependencies on $\varepsilon$ here, especially the dependency of the approximate solutions and the upcoming constants on $\varepsilon$. Applying It\^o's energy formula, for any $t\in [0,T]$ one gets that 
\begin{align*}
&\frac12 \|u_\delta(t)\|^2_{L^2} + \delta \int_0^t (\|u_\delta\|^{\nu}_{W^{m,\nu}_0}+ \|u_\delta\|^2_{H^m_0})\,ds + \int_0^t \int_D  a(\max(\psi,u_\delta),\nabla u_\delta)\cdot\nabla u_\delta \,dx\,ds \\
&-\frac1\varepsilon\int_0^t\int_D\penalisation{u_\delta-\psi}u_\delta \,dx\,ds\\ 
=& \int_0^t \langle f,u_\delta \rangle_{V',V}\,ds + \frac12 \int_0^t \|G(\max(\psi,u_\delta))\|^2_{HS}\,ds + \int_0^t (G(\max(\psi,u_\delta))(\cdot),u_\delta)_H \,dW
+\frac12 \|u_0\|^2_{L^2}.
\end{align*}
Then, using \eqref{coercif1}, \eqref{221019_01}, Young's and Poincar\'{e}'s inequalities and similar arguments as in \eqref{221020_01} we get for constants $C_1$, $C_{\bar{\alpha}}>0$
\begin{align*}
&\frac12 \|u_\delta(t)\|^2_{L^2} + \delta \int_0^t (\|u_\delta\|^{\nu}_{W^{m,\nu}_0}+ \|u_\delta\|^2_{H^m_0}) \,ds + \bar\alpha \int_0^t\|\nabla u_\delta \|_{L^p}^p \,ds + \frac{1}{2\varepsilon}\int_0^t\|(u_\delta-\psi)^-\|^{2}_{L^{2}}\,ds
\\ 
\leq& \int_0^t \frac{C_1}{\kappa}2^p C_{\text{Poincar\'e}}\Vert\nabla u_{\delta}\Vert_{L^p}^p\,ds+C_{\overline \alpha}\int_0^t   \| f\|^{p'}_{V'} \,ds  + \frac{\bar \alpha}{2} \int_0^t \|u_\delta\|^p_{V} \,ds \\ 
&+ \frac12 \int_0^t \|G(\max(\psi,u_\delta))\|^2_{HS}\,ds + \int_0^t (G(\max(\psi,u_\delta))(\cdot),u_\delta)_H \,dW+\frac12 \|u_0\|^2_{L^2}+C(1+\|\psi\|_{L^p(Q_T)}^p),
\end{align*}
with $\kappa=\frac{p}{q}>1$ and a constant $C\geq 0$ that may depend on $\varepsilon$. Using H$_{3,1}$ and H$_{3,2}$ and choosing $C_1>0$ such that $K_5:=\frac{\bar{\alpha}}{2}-\frac{C_1}{\kappa}2^p C_{\text{Poincar\'e}}>0$
\begin{align*}
&\frac12 \|u_\delta(t)\|^2_{L^2} + \delta \int_0^t (\|u_\delta\|^{\nu}_{W^{m,\nu}_0}+ \|u_\delta\|^2_{H^m_0}) \,ds + K_5 \int_0^t\|\nabla u_\delta \|_{L^p}^p \,ds\\ 
\leq& C_{\bar{\alpha}}\| f\|^{p'}_{L^{p'}(0,T,V')}+\int_0^t\Vert G(0)\Vert_{HS}^2+2C_{\sigma}\Vert \psi\Vert_{L^2}^2\,ds+K_6\int_0^t  \Vert u_{\delta}\Vert_{L^2}^2\,ds\\
&+\int_0^t (G(\max(\psi,u_\delta))(\cdot),u_\delta)_H \,dW +\frac12 \|u_0\|^2_{L^2}+C(1+\|\psi\|_{L^p(Q_T)}^p).
\end{align*}
Thus, by Remark \ref{conti-obst} and assumptions on $G$, $\sigma$, and H$_1$ on the initial value there exists $\Theta \in L^1(\Omega)$ that may depend on $\eps>0$ such that 
\begin{align}\label{DefTheta}
\Theta = C\left(1+\| f\|^{p'}_{L^{p'}(0,T,V')}+\Vert G(0)\Vert_{L^2(0,T,HS)}^2+\Vert \psi\Vert_{L^2(Q_T)}^2+\|\psi\|_{L^p(Q_T)}^p+\|u_0\|^2_{L^2}\right)
\end{align}
and with this notation
\begin{align}\label{221020_03}
\begin{aligned}
&\frac12 \|u_\delta(t)\|^2_{L^2} + \delta \int_0^t (\|u_\delta\|^{\nu}_{W^{m,\nu}_0}+ \|u_\delta\|^2_{H^m_0})\,ds + K_5 \int_0^t\|\nabla u_\delta \|_{L^p}^p\,ds
\\ 
\leq& K_6\int_0^t \|u_\delta(t)\|^2_{L^2}\,ds +   \int_0^t (G(\max(\psi,u_\delta))(\cdot),u_\delta)_H \,dW+\Theta
\end{aligned}
\end{align}
holds for all $t\in [0,T]$. Taking the supremum over $t\in [0,T]$, then the expectation in \eqref{221020_03} and discarding non-negative terms, it follows that
\begin{align}\label{221020_04}
\begin{aligned}
&\frac{1}{2}\erww{\sup_{t\in [0,T]}\Vert u_{\delta}(t)\Vert_{L^2}^2}\\
\leq& K_6\erww{\int_0^T\sup_{\tau\in [0,s]}\Vert u_{\delta}(\tau)\Vert_{L^2}^2\,ds}+\erww{\sup_{t\in [0,T]}\left|\int_0^t (G(\max(\psi,u_\delta))(\cdot),u_\delta)_H \,dW\right|}+\erww{\Theta}.
\end{aligned}
\end{align}
Using Burkholder-Davis-Gundy, Cauchy-Schwarz and Young inequalities, for 
any $\eta>0$ we get
\begin{align}\label{221020_02}
\begin{aligned}
& \erww{\sup_{t\in [0,T]} \left|\int_0^t (G(\max(\psi,u_\delta))(\cdot),u_\delta)_H \,dW \right|} 
\\ \leq& 
C_B \erww{\left(\int_0^T \|G(\max(\psi,u_\delta))\|_{HS}^2 \,\|u_\delta\|_{L^2}^2 \,ds\right)^{1/2}}
\\ \leq& 
C_B \erww{\sup_{t\in [0,T]}\|u_\delta(t)\|_{L^2}\left(\int_0^T \|G(\max(\psi,u_\delta))\|_{HS}^2\,ds\right)^{1/2}}
\\ \leq& 
\frac{C_B\eta}{2}\erww{\sup_{t\in [0,T]}\|u_\delta(t)\|^{2}_{L^2}} 
+ 
\frac{C_B}{2\eta} \erww{\int_0^T \|G(\max(\psi,u_\delta))\|_{HS}^{2}  \,ds}
\\ \leq& 
\frac{C_B\eta}{2}\erww{\sup_{t\in [0,T]}\|u_\delta(t)\|^{2}_{L^2}}
+ 
C_\eta^1\erww{\int_0^T \!\!\!\|g\|_{HS}^{2}\,ds}
+
C_\eta^1 \erww{\int_0^T\!\!\! \|\sigma(\cdot,\max(\psi,u_\delta))\|_{HS}^{2} \,ds\!}
\end{aligned}
\end{align}
for a constant $C_\eta^1 >0$ depending on $\eta>0$. Using H$_{3,2}$ we get from \eqref{221020_02}
\begin{align}\label{221020_05}
\begin{aligned}
& \erww{\sup_{t\in [0,T]} \left|\int_0^t (G(\max(\psi,u_\delta))(\cdot),u_\delta)_H \,dW \right| }\\
\leq& \frac{C_B\eta}{2}\erww{\sup_{t\in [0,T]}\|u_\delta(t)\|^{2}_{L^2}}
+ 
C_\eta^1\erww{\int_0^T \|g\|_{HS}^{2}\,ds}\\
&+
C_\eta^1C_{\sigma}2^{2}\erww{\int_0^T  \|\psi\|_{L^2}^{2}\,ds}
+ 
C_\eta^1 C_{\sigma}2^{2}\erww{\int_0^T \|u_\delta\|_{L^2}^{2} \,ds}\\ 
\leq& 
\frac{C_B\eta}{2}\erww{\sup_{t\in [0,T]}\|u_\delta(t)\|^{2}_{L^2}}
+ 
C_\eta^1\erww{\int_0^T \|g\|_{HS}^{2}\,ds}\\
&+
C_\eta^2\erww{\int_0^T  \|\psi\|_{L^2}^{2}\,ds}
+ 
C_\eta^2\erww{\int_0^T \sup_{\tau\in [0,s]}\|u_\delta(\tau)\|_{L^2}^{2} \,ds}
\end{aligned}
\end{align}
where $C_\eta^2>0$ is a constant depending on $\eta>0$.
Combining \eqref{221020_04} and \eqref{221020_05},
we arrive at 
{\small \begin{align*}
&\frac{1-C_B\eta}{2}\erww{\sup_{t\in [0,T]}\Vert u_{\delta}(t)\Vert_{L^2}^2}
\leq K_7\erww{\int_0^T \sup_{\tau\in [0,s]}\|u_\delta(\tau)\|_{L^2}^{2} \,ds}+K_8\left(\Vert g\Vert_{L^2(\Omega_T;HS(H))}^2+\Vert \psi\Vert_{L^2(\Omega_T;L^2(D))}^2\right)+K_9
\end{align*}
}where $K_7$, $K_8$, $K_9>0$ only depend on $\eta>0$ and $\eps>0$. Choosing $\eta>0$ such that $1-C_B\eta>0$, since the above inequality holds also replacing $T$ by any $t\in [0,T]$, from Gronwall's lemma it follows that there exists $K_{10}\geq 0$ not depending on $\delta>0$ such that
\begin{align}\label{221020_06}
\erww{\sup_{t\in [0,T]}\Vert u_{\delta}(t)\Vert_{L^2}^2}\leq K_{10}.
\end{align}
Taking expectation in \eqref{221020_03}, from \eqref{221020_06} it follows that there exists a constant $K_{11}\geq 0$, not depending on $\delta>0$ such that
\begin{align}\label{Estim1}
\erww{\sup_{t\in [0,T]}\|u_\delta(t)\|^2_{L^2}} + \delta \erww{\int_0^T (\|u_\delta\|^{\nu}_{W^{m,\nu}_0}+ \|u_\delta\|^2_{H^m_0})\,dt} + \erww{\int_0^T\|\nabla u_\delta \|_{L^p}^p \,dt}
\leq  K_{11}.
\end{align}
Note that $u_\delta$ is the solution in $\mathcal{V}'$ to the problem 
\begin{align*}
\partial_t \left(u_\delta - \int_0^{\cdot} G(\max(\psi,u_\delta))\,dW\right) + A_\delta(u_\delta)= 0.
\end{align*}
We recall that $A_{\delta}=A_0+\delta b$ where $b$ is defined in \eqref{230104_01} and $A_0$ is defined in \eqref{230104_02}. Keeping the $L^2$-Norm in the second line of  \eqref{221220_01} it follows that
\begin{align}\label{221222_04}
\begin{aligned}
&\erww{\int_0^T\left\Vert\partial_t \left(u_\delta - \int_0^{\cdot} G(\cdot,\max(\psi,u_\delta))\,dW\right)\right\Vert_{\mathcal{V}'}^{\nu'}\,dt}=\erww{\int_0^T \|A_\delta(u_{\delta})\|^{\nu'}_{\mathcal{V}'}\,dt}\\
\leq& 2^{\nu'-1} C_1^\varepsilon \erww{\int_0^T\left((C+1)\Vert u_{\delta}\Vert_{W^{1,p}_0}^{p-1}+C\Vert u_{\delta}\Vert_{L^2}+\bar{h}(\omega,t)\right)^{\nu'}\,dt}\\
&+2^{\nu'-1}\delta^{\nu'}\erww{\int_0^T \left(\Vert u_{\delta}\Vert_{H_0^m}+\Vert u_{\delta}\Vert^{\nu-1}_{W^{m,\nu}_0}\right)^{\nu'}\,dt}.
\end{aligned}
\end{align}
Since $1<\nu'<2$, by \eqref{Estim1}, for all $0<\delta\leq 1$ we have
\begin{align}\label{221222_05}
\begin{aligned}
&\delta^{\nu'}\erww{\int_0^T \left(\Vert u_{\delta}\Vert_{H_0^m}+\Vert u_{\delta}\Vert^{\nu-1}_{W^{m,\nu}_0}\right)^{\nu'}\,dt}
\leq
\delta 2^{\nu'-1}\erww{\int_0^T \left(\Vert u_{\delta}\Vert_{H_0^m}+1\right)^{2}+\Vert u_{\delta}\Vert^{\nu}_{W^{m,\nu}_0}\,dt}\\
\leq& 2^{\nu'}\delta\erww{\int_0^T (\Vert u_{\delta}\Vert^2_{H_0^m}+\Vert u_{\delta}\Vert_{W^{m,\nu}_0}^{\nu})\,dt}+2^{\nu'}T
\leq
2^{\nu'}(K_{11}+T).
\end{aligned}
\end{align}
Since we also have $\nu'<p'$, thanks to \eqref{Estim1} it follows that
\begin{align}\label{221222_06}
\begin{aligned}
&\erww{\int_0^T\left((C+1)\Vert u_{\delta}\Vert_{W^{1,p}_0}^{p-1}+C\Vert u_{\delta}\Vert_{L^2}+\bar{h}(\omega,t)\right)^{v'}\,dt}\\
\leq& (C+1)^{\nu'}\erww{\int_0^T\left(\Vert u_{\delta}\Vert_{W^{1,p}_0}^{p-1}+\Vert u_{\delta}\Vert_{L^2}+\bar{h}(\omega,t)\right)^{\nu'}\,dt}\\
\leq& (C+1)^{\nu'}2^{\nu'-1}\erww{\int_0^T\left(\Vert u_{\delta}\Vert_{W^{1,p}_0}^{p-1}+\bar{h}(\omega,t)+1\right)^{\nu'}+\left(\Vert u_{\delta}\Vert_{L^2}+1\right)^{\nu'}\,dt}\\
\leq& (C+1)^{\nu'}2^{\nu'-1}\erww{\int_0^T\left(\Vert u_{\delta}\Vert_{W^{1,p}_0}^{p-1}+\bar{h}(\omega,t)+1\right)^{p'}+\left(\Vert u_{\delta}\Vert_{L^2}+1\right)^{2}\,dt}\\
\leq& K_{13}\erww{\int_0^T \Vert u_{\delta}\Vert^p_{W^{1,p}_0}+\Vert u_{\delta}\Vert^2_{L^2}\, dt}+K_{14}\leq K_{13}K_{11}(1+2T)+K_{14}
\end{aligned}
\end{align}
where $K_{13}:=(C+1)^{\nu'}2^{\nu'}(2^{p'-1}+1)$, $K_{14}:=(C+1)^{\nu'}2^{\nu'-1}(2^{p'-1}+2)(\Vert\bar{h}\Vert^{p'}_{L^{p'}(\Omega_T)}+T)$. Thus, for any $0<\delta\leq 1$, from \eqref{221222_04}, \eqref{221222_05} and \eqref{221222_06} it follows that
\[\erww{\int_0^T\left\Vert\partial_t \left(u_\delta - \int_0^{\cdot} G(\cdot,\max(\psi,u_\delta))\,dW\right)\right\Vert_{\mathcal{V}'}^{\nu'}\,dt}\]
is uniformly bounded by a constant not depending on $\delta>0$.

We recall that $G(\omega,t,u)=G(\omega,t,0) + \sigma(\omega,t,u)$ with the information that $\sigma$ is vanishing at $u=0$ and Lipschitz, uniformly in $(\omega,t)$. Then, for $r=\max(p,2)$ and any $\alpha<1/2$, a constant $C(r,\alpha)\geq 0$ and a constant $K_{15}\geq 0$ not depending on $\delta>0$ exist such that, 
\begin{align}\label{EstimWfrac}
\erww{\left\Vert \int_0^{\cdot} \sigma(\max(\psi,u_\delta))\,dW\right\Vert^r_{W^{\alpha,r}(0,T;H)} }\leq C(r,\alpha) \erww{\int_0^T \|\sigma(\max(\psi,u_\delta))(t)\|^r_{HS}\,dt}\leq K_{15}.
\end{align}
Indeed, this is a consequence of \cite[Lemma 2.1]{Flan-Gata} by noticing that $\max(\psi,u_\delta)$ is bounded in $L^r(\Omega_T;H)$ by
\eqref{Estim1}. The following lemma is proposed to gather the previous estimates. 
\begin{lemma}\label{lem_estimations}
There exists a constant $K_{16}\geq 0$, independent of $\delta>0$ such that:
\begin{enumerate}
\item $(u_\delta)_{\delta}$ is bounded by $K_{16}\geq 0$ in $L^p(\Omega_T;V)$, in particular in $L^p(\Omega_T;L^p(D))$, 
\item $(u_\delta)_{\delta}$ is bounded by $K_{16}\geq 0$ in $L^2(\Omega;C([0,T];L^2(D)))$,
\item $(\sqrt[\nu]{\delta} u_\delta)_{\delta}$ is bounded by $K_{16}\geq 0$ in $L^\nu(\Omega_T;W_0^{m,\nu}(D))$ and 
$(\sqrt{\delta} u_\delta)_{\delta}$ is bounded by $K_{16}\geq 0$ in $L^2(\Omega_T;H_0^{m}(D))$, 
\item $\left(u_\delta - \int_0^{\cdot} G(\max(\psi,u_\delta))\,dW\right)_{\delta}$ is bounded by $K_{16}\geq 0$ in $L^2(\Omega_T;L^2(D))$\\ and $\partial_t \left(u_\delta - \int_0^{\cdot} G(\max(\psi,u_\delta))\,dW\right)_{\delta}$ is bounded by $K_{16}\geq 0$ in $L^{\nu'}(\Omega_T;\mathcal{V}')$.
\end{enumerate}
\end{lemma}

\subsubsection{Tightness}\label{Tightness}
We fix an arbitrary $\eta>0$. Let us define the space
\begin{align}\label{Espace_u}
\mathbb{W}=\left\{v \in L^2(0,T;L^2(D)),\ \partial_tv \in L^{\nu'}(0,T;\mathcal{V}')\right\}.
\end{align}
For $R>0$, from Lemma \ref{lem_estimations}, $(4)$, it follows that $v_\delta:= u_\delta - \int_0^{\cdot} G(\max(\psi,u_\delta))\,dW$ is bounded in $L^{\nu'}(\Omega;\mathbb{W})$ by a constant $K_{17}\geq 0$ not depending on $\delta>0$ and therefore
\begin{align}\label{221027_01}
\begin{aligned}
&\mathds{P}\left[\left\Vert v_\delta\right\Vert_{\mathbb{W}}>R\right]
\leq
\frac{1}{R^{\nu'}}\int\limits_{\{\|v_\delta\|_{\mathbb{W}}>R\}}\left\Vert v_\delta\right\Vert^{\nu'}_{\mathbb{W}}\,d\mathds{P}
\leq \frac{K_{17}^{\nu'}}{R^{\nu'}}.
\end{aligned}
\end{align}

Concerning the martingale term, 
\[\int_0^{\cdot} G(\max(\psi,u_\delta))- G(0)\,dW=\int_0^{\cdot} \sigma(\max(\psi,u_\delta))\,dW\]
and, by \eqref{EstimWfrac}, the term on the right-hand side of the above equation is bounded independently of $\delta>0$ in $L^r(\Omega; W^{\alpha,r}(0,T;L^2(D)))$ for $\alpha<1/2$ and $r=\max(2,p)$. 
\\
Thus, with similar arguments as in \eqref{221027_01}
\begin{align}\label{220127_02}
&\mathds{P}\left[\left\Vert \int_0^{\cdot} \sigma(\max(\psi,u_\delta))\,dW\right\Vert_{W^{\alpha,r}(0,T;L^2(D))}>R\right] 
\leq  \frac{K_{15}}{R^{r}}.
\end{align}
Since $\int_0^{\cdot} G(0)\,dW$ is a fixed element of $W^{\alpha,2}(0,T;L^2(D))$, as 
\begin{align*}
u_\delta = u_\delta - \int_0^{\cdot} G(\max(\psi,u_\delta))\,dW + \int_0^{\cdot} G(0)\,dW+\int_0^{\cdot} \sigma(\max(\psi,u_\delta))\,dW,
\end{align*}
and since $\mathbb{W}$ is continuously embedded into $W^{\alpha,\nu'}(0,T;\mathcal{V}')$, with similar arguments as in \eqref{221027_01} one gets that
\begin{align*}
&\mathds{P}\left[\|u_\delta\|_{W^{\alpha,\nu'}(0,T;\mathcal{V}')}>R_1\right] 
\leq \frac{\eta}{2} 
\end{align*}
for $R_1>0$ well-chosen. Similarly, from Lemma \ref{lem_estimations}, $(1)$ and $(2)$ it follows that
\begin{align*}
\mathds{P}\left[\|u_\delta\|_{L^{p}(0,T;V)}>R_2\right] +\mathds{P}\left[\|u_\delta\|_{C([0,T];L^2(D))}>R_2\right] 
\leq & \frac{K_{16}}{R_2^{\min(p,2)}} \leq \frac{\eta}{2} 
\end{align*}
for $R_2>0$ well-chosen. Denote by 
\begin{align*}
\mathcal{K}_0=\overline{B}_{W^{\alpha,\nu'}(0,T;\mathcal{V}')}(0,R_1) \cap \overline{B}_{L^{p}(0,T;V)}(0,R_2) \cap \overline{B}_{C([0,T];L^2(D))}(0,R_2)
\end{align*}
and note that, by \cite[Cor. 7]{Simon}, $\mathcal{K}_0$ is relatively compact in $L^s(0,T;L^2(D))$ for any finite $s>1$, in particular for $s=4$. Moreover,
\begin{align}\label{221027_03}
\mathds{P}\left[u_\delta \in \mathcal{K}_0\right]=1-\mathds{P}\left[ u_\delta \not\in \mathcal{K}_0\right] \geq 1 - \eta.
\end{align}
Now, we define
\[\widetilde{\mathcal{K}}_0 = \overline{B}_{\mathbb{W}}(0,R_3)+\overline{B}_{W^{\alpha,r}(0,T;L^2(D))}(0,R_3) + \widetilde{\mathcal{K}}_0^G\] 
where, Prokhorov's theorem yields the existence of a compact subset $\widetilde{\mathcal{K}}_0^G$ of $C([0,T];\mathcal{V}')$ such that 
\[\mathds{P}\left[\int_0^{\cdot} G(0)\,dW\in \widetilde{\mathcal{K}}_0^G\right]\geq 1-\eta.\]
Since $\psi$, $G(0)$ and $\sigma$ are random variables with values in the separable Banach spaces $C([0,T];L^2(D))$, $L^2(0,T;HS(H))$ and $L^4(0,T;\Sigma)$ respectively, Prokhorov's theorem yields the existence of three compact sets $\mathcal{K}_i$ ($i=1,2,3$) in each one of those three spaces such that 
\begin{align*}
\mathds{P}\left[\psi\in \mathcal{K}_1\right] \geq 1-\eta,\quad 
\mathds{P}\left[G(0)\in \mathcal{K}_2 \right] \geq 1-\eta, \quad 
\mathds{P}\left[\sigma \in \mathcal{K}_3 \right]\geq 1-\eta.
\end{align*}
Thanks to the assumption on $C_{\sigma}$ in H$_{3,2}$, it follows that $\sigma \in L^{\infty}(\Omega_T, \operatorname{Lip}_0(H; HS(H)))$, thus also in $\sigma \in L^4(0,T;\Sigma)$ a.s. in $\Omega$.
Now we consider the mapping
\begin{align*}
&\Theta : L^4(0,T;L^2(D))\times C([0,T];L^2(D)) \times L^2(0,T;HS(H)) \times L^4(0,T;\Sigma) \to L^2(0,T;HS(H))\\ 
&(u,\psi,g,\sigma)\mapsto \Theta(u,\psi,g,\sigma)=g+\sigma(\max(u,\psi)).
\end{align*}

Then, for any $(u_1,\psi_1,g_1,\sigma_1)$, $(u_2,\psi_2,g_2,\sigma_2)$ in $L^4(0,T;L^2(D))\times C([0,T];L^2(D)) \times L^2(0,T;HS(H)) \times L^4(0,T;\Sigma)$, 
\begin{align*}
&\left\Vert\Theta(u_1,\psi_1,g_1,\sigma_1)-\Theta(u_2,\psi_2,g_2,\sigma_2)\right\Vert_{HS}
=
\|g_1-g_2 +\sigma_1(\max(u_1,\psi_1))-\sigma_2(\max(u_2,\psi_2))\|_{HS}
\\ \leq &
\|g_1-g_2\|_{HS} +\|\sigma_1(\max(u_1,\psi_1))-\sigma_2(\max(u_1,\psi_1))\|_{HS} + \|\sigma_2(\max(u_1,\psi_1))-\sigma_2(\max(u_2,\psi_2))\|_{HS}
\\ \leq &
\|g_1-g_2\|_{HS} +\|\sigma_1-\sigma_2\|_{Lip_0}\|\max(u_1,\psi_1))\|_{L^2} + \|\sigma_2\|_{Lip_0}\|\max(u_1,\psi_1)-\max(u_2,\psi_2)\|_{L^2}
\end{align*}
and therefore 
\begin{align*}
&\int_0^T\left\Vert\Theta(u_1,\psi_1,g_1,\sigma_1)-\Theta(u_2,\psi_2,g_2,\sigma_2)\right\Vert^2_{HS}\,dt\\
\leq&
C\left(\|g_1-g_2\|^2_{L^2(0,T;HS)} +\|\sigma_1-\sigma_2\|^2_{L^4(0,T;\Sigma)}\|\max(u_1,\psi_1))\|^2_{L^4(0,T;L^2)}\right) 
\\
&+C\|\sigma_2\|^2_{L^4(0,T;\Sigma)}\|\max(u_1,\psi_1))-\max(u_2,\psi_2)\|^2_{L^4(0,T;L^2)}.
\end{align*}
Note that since $\Theta(0)=0$, the above inequality justifies that $\Theta$ has values in $L^2(0,T;HS(H))$ and the fact that it is a continuous mapping. Therefore, 
\[\mathcal{K}=\Theta(\mathcal{K}_0\times \mathcal{K}_1\times \mathcal{K}_2\times \mathcal{K}_3)\] 
is a compact subset of $L^2(0,T;HS(H))$ and since $G(\cdot,\max(\psi,u_\delta))=\Theta(u_\delta, \psi,G(0),\sigma)$, $\mathds{P}$-a.s. in $\Omega$, 
\begin{align}\label{221027_05}
\mathds{P}\left[G(\cdot,\max(\psi,u_\delta)) \in \mathcal{K}\right] \geq \mathds{P}\left[(u_\delta, \psi,G(0),\sigma) \in\mathcal{K}_0\times \mathcal{K}_1\times \mathcal{K}_2\times \mathcal{K}_3 \right] \geq 1 - 4\eta.
\end{align}
Let
\[\mu_\delta=\mathcal{L}(u_\delta, G(\cdot,\max(\psi,u_\delta)), \psi,G(0),\sigma,W,f,u_0)\] be the law of the random variable $(u_\delta, G(\cdot,\max(\psi,u_\delta)), \psi, G(0), \sigma, W,f,u_0)$ on the space
\begin{align*}
\mathcal{X}=&L^4(0,T;L^2(D))\times L^2(0,T;HS(H)) \times \left(L^p(0,T;V)\cap C([0,T];L^2(D)) \right)
\\ &\quad 
\times L^2(0,T;HS(H)) \times  L^4(0,T;\Sigma) \times C([0,T];U) \times L^{p'}(0,T;V')\times L^2(D).
\end{align*}
Thanks to \eqref{221027_03} and \eqref{221027_05}, which hold true for any fixed $\eta>0$ independently of $\delta>0$, it follows that $(\mu_\delta)_{\delta}$ is tight on $\mathcal{X}$. Hence, passing to a not relabeled subsequence if necessary, by Prokhorov's theorem if follows that, for $\delta\rightarrow 0^+$, $\mu_{\delta}$ converges to a probability measure $\mu_{\infty}$ with respect to the narrow topology of $\mathcal{X}$. Then,\cite[Thm. 6.7]{Bil99} yields the existence of 
a probability space $(\overline{\Omega}, \overline{\mathcal{F}},  \overline{\mathds{P}})$ which may be chosen as $([0,1],\mathcal{B}([0,1]),\operatorname{Leb})$ with $\operatorname{Leb}$ being the Lebesgue measure on $[0,1]$ and a family of random vectors 
\[(\overline u_{\delta}, \overline{G}_{\delta}, \overline\psi_{\delta},\overline G_0^{\delta},\overline\sigma_{\delta},\overline W_{\delta},\overline f_{\delta},\overline{u_0}^{\delta})\] 
on $(\overline \Omega, \overline{\mathcal{F}}, \overline{\mathds{P}})$ with values in $\mathcal{X}$ having the same law as $(u_\delta, G(\cdot,\max(\psi,u_\delta)),  \psi,G(0),\sigma,W,f,u_0)$. Moreover, there exist
\begin{itemize}
\item[$i)$] a random variable $\overline u_\infty$ with values in $L^4(0,T;L^2(D))$ such that $(\overline u_\delta)_{\delta}$ converges $\overline{\mathds{P}}$-a.s. to $\overline u_\infty$ in $L^4(0,T;L^2(D))$ for $\delta\rightarrow 0^+$.
\item[$ii.)$] a random variable $\overline{G}_\infty$ with values in $L^2(0,T;HS(H))$ such that $(\overline G_\delta)_{\delta}$ converges $\overline{\mathds{P}}$-a.s. to $\overline{G}_\infty$ in $L^2(0,T;HS(H))$ for $\delta\rightarrow 0^+$.
\item[$iii.)$] a random variable $\overline{\psi}_{\infty}$ such that $\mathcal{L}(\overline{\psi}_{\infty})=\mathcal{L}(\psi)$ and $(\overline{\psi}_{\delta})_{\delta}$ converges $\overline{\mathds{P}}$-a.s. to $\overline{\psi}_{\infty}$ in $L^p(0,T;V)\cap C([0,T];L^2(D))$ for $\delta\rightarrow 0^+$.
\item[$iv.)$] a random variable $\overline{G}_0^{\infty}$ such that $\mathcal{L}(\overline{G}_0^{\infty})=\mathcal{L}(G(0))$ and $(\overline{G}_0^{\delta})_{\delta}$ converges $\overline{\mathds{P}}$-a.s. to $\overline{G}_0^{\infty}$ in $L^2(0,T;HS(H))$ for $\delta\rightarrow 0^+$.
\item[$v.)$] a random variable $\overline{\sigma}_{\infty}$ such that $\mathcal{L}(\overline{\sigma}_{\infty})=\mathcal{L}(\sigma)$ and $(\overline{\sigma}_{\delta})_{\delta}$ converges $\overline{\mathds{P}}$-a.s. to $\overline{\sigma}_{\infty}$ in $L^4(0,T;\Sigma)$ for $\delta\rightarrow 0^+$.
\item[$vi.)$] a random variable $\overline{W}_{\infty}$ such that $\mathcal{L}(\overline{W}_{\infty})=\mathcal{L}(W)$ and $(\overline{W}_{\delta})_{\delta}$ converges $\overline{\mathds{P}}$-a.s. to $\overline{W}_{\infty}$ in $C([0,T];U)$ for $\delta\rightarrow 0^+$.
 \item[$vii.)$] a random variable $\overline{f}_{\infty}$ such that $\mathcal{L}(\overline{f}_{\infty})=\mathcal{L}(f)$ and $(\overline{f}_{\delta})_{\delta}$ converges $\overline{\mathds{P}}$-a.s. to $\overline{f}_{\infty}$ in $L^{p'}(0,T;V')$ for $\delta\rightarrow 0^+$.
  \item[$viii.)$] a random variable $\overline{u}_0^{\infty}$ such that $\mathcal{L}(\overline{u}_0^{\infty})=\mathcal{L}(u_0)$ and $(\overline{u}_0^{\delta})_{\delta}$ converges $\overline{\mathds{P}}$-a.s. to $\overline{u}_0^{\infty}$ in $L^2(D)$ for $\delta\rightarrow 0^+$.
\end{itemize}
For the sake of clarity, the expectation with respect to $(\overline \Omega, \overline{\mathcal{F}}, \overline{\mathbb{P}})$ will be denoted by $\overline{\mathds{E}}$. 
Moreover, we will consider $0<\delta\leq 1$. In order to shorten the notation, we will often use the index $\delta=\infty$ to include the random variables $\overline u_\infty$, $\overline{G}_\infty$, $\overline{\psi}_{\infty}$, $\overline{G}_0^{\infty}$, $\overline{\sigma}_{\infty}$, $\overline{W}_{\infty}$, $\overline{f}_{\infty}$, $\overline{u}_0^{\infty}$ into our considerations. By equality of laws, the random vector $(\overline u_{\delta}, \overline{G}_{\delta}, \overline\psi_{\delta},\overline G_0^{\delta},\overline\sigma_{\delta},\overline W_{\delta},\overline f_{\delta},\overline{u_0}^{\delta})$ has the same regularity properties as $(u_\delta, G(\cdot,\max(\psi,u_\delta)), \psi,G(0),\sigma,W,f,u_0)$ for all $0<\delta\leq 1$. In particular, $\overline{u}_{\delta}\in L^2(\overline{\Omega};C([0,T];L^2(D)))$ for all $0<\delta\leq 1$. Since H$_3$ implies $\sigma\in L^2(\Omega_T;Lip_0(H;HS(H)))$, from equality of laws it also follows that $\overline{\sigma}_{\infty}, \overline{\sigma}_{\delta}\in L^2(\overline{\Omega}\times (0,T);Lip_0(H;HS(H))$. Since, for all $0<\delta\leq 1$,
\[G(\cdot,\max(\psi,u_\delta))=G(0)+\sigma(\max(\psi,u_{\delta}))=\Theta(u_\delta, \psi,G(0),\sigma),\] 
by equality of laws one has that 
\begin{align*}
0&=\erww{\Vert\Theta(u_\delta, \psi,G(0),\sigma)-G(\cdot,\max(\psi, u_{\delta}))\Vert_{L^2(0,T;HS(H))}}
=\erws{\Vert\Theta(\overline u_{\delta}, \overline{\psi}_{\delta},\overline G_0^{\delta}, \overline\sigma_{\delta})-\overline G_{\delta}\Vert_{L^2(0,T;HS(H))}},
\end{align*}
hence 
\begin{align}\label{250416_01}
\overline G_\delta=\Theta(\overline u_\delta, \overline \psi_{\delta}, \overline G_0^{\delta},\overline \sigma_{\delta})=\overline{G}_0^{\delta}+\overline{\sigma}_{\delta}(\max(\overline{\psi}_{\delta},\overline{u}_{\delta})).
\end{align}
Let us remark that since $\mathcal{L}(W)=\mathcal{L}(\overline W_{\delta})=\mathcal{L}(\overline W_{\infty})$ it follows that $\overline W_{\delta}(0)=\overline W_{\infty}(0)=0$. 
\\
In the following, for fixed $\delta \in (0,1]\cup\{+\infty\}$,
we will construct a filtration $(\overline{\mathcal{F}}^\delta_t)_{t\in [0,T]}$ on $(\overline \Omega, \overline{\mathcal{F}}, \overline{\mathds{P}})$ such that $\overline{W}_{\delta}=(\overline{W}_{\delta}(t))_{t\in [0,T]}$ 
is a $Q$-Wiener process with values in $U$ with respect to $(\overline{\mathcal{F}}^\delta_t)_{t\in [0,T]}$ and such that $\overline G_\delta=\overline{G}_0^{\delta}+\overline{\sigma}_{\delta}(\max(\overline{\psi}_{\delta},\overline{u}_{\delta}))$ is predictable with respect to $(\overline{\mathcal{F}}^\delta_t)_{t\in [0,T]}$.
To obtain this filtration, we will appropriately enlarge the natural filtration generated by $(\overline{W}_{\delta}(t))_{t\in [0,T]}$. 
\begin{definition}
For fixed $\delta \in (0,1]\cup\{+\infty\}$ and any $t\in [0,T]$, we define $\widetilde{\mathcal{F}}^\delta_t$ to be the smallest sub $\sigma$-field of 
$\overline{\mathcal{F}}$ generated by $\overline u_0^{\delta}$, $\overline{W}_{\delta}(s)$, $\overline{\psi}_{\delta}(s)$, $\int_0^s\overline u_{\delta}(r)\,dr$, $\int_0^s \overline{G_0}^{\delta}(r)\,dr$, $\int_0^s \overline{G}_{\delta}(r)\,dr$, $\int_0^s \overline{f}_{\delta}(r)\,dr$, $\int_0^s \overline{\sigma}_{\delta}(r)\,dr$ for all $0\leq s\leq t$. 
The right-continuous, $\overline{\mathds{P}}$-augmented filtration of $(\widetilde{\mathcal{F}}^\delta_t)_{t\in [0,T]}$, denoted by $(\overline{\mathcal{F}}^\delta_t)_{t\in [0,T]}$ is defined by
\[\overline{\mathcal{F}}^\delta_t: =\bigcap_{T\geq r>t}\sigma\left[\widetilde{\mathcal{F}}^\delta_r\cup \{\mathcal{N}\in\overline{\mathcal{F}} \, : \, \overline{\mathds{P}}(\mathcal{N})=0)\}\right].\]
\end{definition}
The proof of the next Lemma is similar to the proof of \cite[Lemma 4.13] {BNSZ22}.
\begin{lemma}\label{230105_lem1}
For any $0<\delta\leq 1$, and $\delta=\infty$, respectively, there exist $(\overline{\mathcal{F}}^\delta_t)_{t\in [0,T]}$-predictable $d\overline{\mathds{P}}\otimes dt$-representatives of $\overline{u}_{\delta}$, $\overline{G_0}^\delta$, $\overline{G}_\delta$, $\overline{f}_\delta$, and $\overline{\sigma}_\delta$.
\end{lemma}
\begin{lemma}\label{230105_lem2}
For any $\delta \in (0,1]\cup\{+\infty\}$, $(\overline{W}_\delta(t))_{t\in [0,T]}$ is a $U$-valued, square integrable $(\overline{\mathcal{F}}^\delta_t)_{t\in [0,T]}$-martingale with quadratic variation process $tQ$ for any $t\in [0,T]$.
\end{lemma}
\begin{proof}
Since $(\overline{\mathcal{F}}^\delta_t)_{t\in [0,T]}$ is a filtration containing the natural filtration of $(\overline{W}_\delta(t))_{t\in [0,T]}$ 
it follows that
$(\overline{W}_\delta(t))_{t\in [0,T]}$ is adapted to $(\overline{\mathcal{F}}^\delta_t)_{t\in [0,T]}$.
Using $\mathcal{L}(\overline{W}_\delta)=\mathcal{L}(W)$, by Burkholder-Davis-Gundy inequality it follows that there exists a constant $C>0$ such that
\begin{align*}
\erws{\sup_{t\in[0,T]}|\overline{W}_\delta(t)|_U^2}=\erww{\sup_{t\in[0,T]}|W(t)|_U^2}\leq C Tr(Q)T=C(Q,T)<\infty 
\end{align*}
for any fixed $\delta\in (0,1]\cup \{+\infty\}$. For $0<\delta\leq 1$
let us define the random variables $X^\delta$, $\overline X^\delta$ with values in
\[L^2(D)\times C([0,T];U)\times C([0,T];H)\times C([0,T];H)\times C([0,T];HS(H))  \times C([0,T]; V') \times C([0,T]; \Sigma)\]
by
\begin{align*}
X^\delta_t:=&\left(u_0, W(t), \int_0^t u_{\delta}(r)\,dr, \psi,\int_0^t G(0)(r)\,dr,\int_0^t f(r)\,dr,\int_0^t \sigma(r)\,dr\right),
\\
\overline X^\delta_t:=&\left(\overline u_0^\delta, \overline W_\delta(t), \int_0^t\overline u_{\delta}(r)\,dr, \overline \psi_\delta,\int_0^t \overline G_0^\delta(r)\,dr,\int_0^t \overline f_\delta(r)\,dr,\int_0^t \overline \sigma_\delta(r)\,dr\right),
\end{align*}
for any $t\in [0,T]$. For all $k\in\mathbb{N}$, $0\leq s\leq t\leq T$ and all bounded, continuous functions 
\[\Upsilon:L^2(D)\times C([0,s];U)\times C([0,s];H)\times C([0,s];H)\times C([0,s];HS(H))  \times C([0,s]; V') \times C([0,s]; \Sigma) \rightarrow\mathbb{R},\]
thanks to $\mathcal{L}(\overline{W}_\delta)=\mathcal{L}(W)$ we have
\begin{align}\label{230106_01}
\begin{aligned}
0&=\erww{(W(t)-W(s),e_k)_U\Upsilon\left(X^\delta|_{[0,s]}\right)}\\
&=\erws{(\overline{W}_\delta(t)-\overline{W}_\delta(s),e_k)_U\Upsilon\left(\overline{X}^\delta|_{[0,s]}\right)}.
\end{aligned}
\end{align}
The real-valued random variable 
\[\overline{\Omega}\ni\overline{\omega}\mapsto \Upsilon\left(\overline{X}^\delta(\omega)|_{[0,s]}\right)\]
is $\widetilde{\mathcal{F}}^\delta_s$-measurable.
Since \eqref{230106_01} applies to every bounded and continuous function $\Upsilon$, from the Lemma of Doob-Dynkin (see, \textit{e.g.}, \cite[Proposition 3]{RS06}) it follows that
\begin{align*}
0=\erws{\mathds{1}_A(\overline{W}_\delta(t)-\overline{W}_\delta(s),e_k)_U}
\end{align*}
for all $k\in\mathbb{N}$, for all $\widetilde{\mathcal{F}}^\delta_s$-measurable subsets $A\in \overline{\mathcal{F}}$ and for all $0\leq s\leq t\leq T$.
From the above equation it now follows that 
\[\erws{(\overline{W}_\delta(t)-\overline{W}_\delta(s),e_k)_U|\,\widetilde{\mathcal{F}}^\delta_s}=0\]
for all $0\leq s\leq t\leq T$ and all $k\in\mathbb{N}$.
Therefore, $(\overline{W}_\delta(t))_{t\in [0,T]}$ is a $(\widetilde{\mathcal{F}}^\delta_t)_{t\in [0,T]}$-martingale for any fixed $0<\delta\leq 1$. Using \cite[p.75]{Della-Meyer} we may conclude that, for any fixed $0<\delta\leq 1$, $(\overline{W}_{\delta}(t))_{t\in [0,T]}$ is also a martingale with respect to the augmented filtration $(\overline{\mathcal{F}}^{\delta}_t)_{t\in [0,T]}$.
We recall that for our $U$-valued $Q$-Wiener process $(W(t))_{t\geq 0}$ with respect to $(\mathcal{F}_t)_{t\geq 0}$, its quadratic variation process $\ll W \gg_t$ is equal to $tQ$ for all $t\in [0,T]$, or equivalently,
\[\erww{(W(t)-W(s),e_k)_U(W(t)-W(s),e_j)_U-((t-s)Q(e_k),e_j)_U|\mathcal{F}_s}=0\]
for all $0\leq s\leq t\leq T$ and all $k,j\in\mathbb{N}$ (see \cite[p.75]{DPZ14}). 
Using again that $\mathcal{L}(\overline{W}_\delta)=\mathcal{L}(W)$, and the notation introduced before \eqref{230106_01}, for any $0<\delta\leq 1$, $0\leq s\leq t\leq T$ and $k,j\in\mathbb{N}$, 
we get the following result:
\begin{align}\label{230106_03}
\begin{aligned}
0&=\erww{(W(t)-W(s),e_k)_U(W(t)-W(s),e_j)_U-((t-s)Q(e_k),e_j)_U\Upsilon\left(X^\delta|_{[0,s]}\right)}\\
&=\erws{(\overline{W}_\delta(t)-\overline{W}_\delta(s),e_k)_U(\overline{W}_\delta(t)-\overline{W}_\delta(s),e_j)_U-((t-s)Q(e_k),e_j)_U\Upsilon\left(\overline{X}^\delta|_{[0,s]}\right)}.
\end{aligned}
\end{align}
With similar arguments as for \eqref{230106_01} 
from \eqref{230106_03} it follows that, for any $0<\delta\leq 1$, $\ll \overline W_\delta \gg_t=tQ$ for all $t\in [0,T]$.\\
For any $t\in [0,T]$, we define
\begin{align*}
\overline X^\infty_t:=&\left(\overline u_0^\infty, \overline W_\infty, \int_0^t\overline u_{\infty}(r)\,dr, \overline \psi_\infty,\int_0^t \overline G_0^\infty(r)\,dr,\int_0^t \overline f_\infty(r)\,dr,\int_0^t \overline \sigma_\infty(r)\,dr\right).
\end{align*}
First, we observe that thanks to the convergence result $i.)$, $\int_0^{\cdot}\overline{u}_{\delta}(r)\,dr\rightarrow \int_0^{\cdot}\overline{u}_{\infty}(r)\,dr$ as $\delta\rightarrow 0^+$ in $C([0,s];H)$ for all $0\leq s\leq t\leq T$. Then, by the  convergence results $i)$ to $viii)$, for any continuous, bounded function $\Upsilon$ as introduced above, $\Upsilon\left(\overline{X}^\delta|_{[0,s]}\right)$ converges a.s. to $\Upsilon\left(\overline{X}^\infty|_{[0,s]}\right)$ for $\delta\rightarrow 0^+$.
\\
Then, for any $t,s \in [0,T]$ and $ k \in \N$, the convergence result $vi.)$ yields that $(\overline{W}^\delta(t)-\overline{W}^\delta(s),e_k)_U$ converges a.s to $(\overline{W}^\infty(t)-\overline{W}^\infty(s),e_k)_U$. Thanks to the equality in law, all moments of the sequence are bounded, independently of $0<\delta\leq 1$. Hence, Vitali's theorem implies convergence in $L^p(\overline{\Omega})$, for any $1<p<\infty$, more precisely 
\begin{align*}
&\erww{(\overline{W}_\infty(t)-\overline{W}_\infty(s),e_k)_U\Upsilon\left(X^\infty|_{[0,s]}\right)}
=\lim_{\delta\rightarrow 0^+}\erws{(\overline{W}_\delta(t)-\overline{W}_\delta(s),e_k)_U\Upsilon\left(\overline{X}^\delta|_{[0,s]}\right)}=0.
\end{align*}
Similarly, one has 
\begin{align*}
\erws{(\overline{W}_\infty(t)-\overline{W}_\infty(s),e_k)_U(\overline{W}_\infty(t)-\overline{W}_\infty(s),e_j)_U-((t-s)Q(e_k),e_j)_U\Upsilon\left(\overline{X}^\infty|_{[0,s]}\right)}=0,
\end{align*}
and the lemma is proved also for $\delta=\infty$.
\end{proof}
\begin{remark}
For all $0<\delta\leq 1$, and $\delta=\infty$, it follows from Lemma \ref{230105_lem2} that $(\overline{W}_\delta(t))_{t\in [0,T]}$ is a $U$-valued $Q$-Wiener process with respect to $(\overline{\mathcal{F}}^\delta_t)_{t\in [0,T]}$. From Lemma \ref{230105_lem1} it follows that $\overline{G}_\delta=\overline{G}_0^{\delta}+\overline{\sigma}_{\delta}(\max(\overline{\psi}_{\delta},\overline{u}_{\delta}))$ has a predictable $d\overline{\mathds{P}}\otimes dt$-representative with respect to that filtration and therefore the It\^{o} stochastic integral $\int_0^t \overline{G}_\delta\,d\overline{W}_\delta$ is well-defined. 
\end{remark}
In the following, we want to recover the stochastic integral and our equation on the new probability space.
\begin{lemma}
For any $t\in [0,T]$, $0<\delta\leq 1$,  let us define
\begin{align}\label{230306_04}
M_{\delta}(t):=\overline{u}_{\delta}(t)-\overline{u}_0^\delta+\int_0^t \delta\, \partial J(\overline{u}_{\delta})-\Div a(\max(\overline{\psi}_\delta, \overline{u}_{\delta}), \nabla \overline{u}_{\delta})-\frac{1}{\eps}(\overline{u}_{\delta}-\overline{\psi}_\delta)^{-}-\overline{f}_\delta\,ds,
\end{align}
where $\langle \partial J(u),v\rangle_{\mathcal{V}',\mathcal{V}}=b(u,v)$ for all $u,v\in \mathcal{V}$ (see \eqref{230104_01}).
The stochastic process $(M_{\delta}(t))_{t\in [0,T]}$ is a square-integrable, continuous $(\overline{\mathcal{F}}^\delta_t)_{t\in [0,T]}$ martingale with values in $L^2(D)$  such that 
\begin{align}\label{230306_01}
\ll M_{\delta}\gg_t&=\int_0^t (\overline{G}_\delta\circ Q^{1/2})\circ(\overline G_{\delta}\circ Q^{1/2})^{\ast}\,ds
\end{align}
\begin{align}\label{230306_02}
\ll \overline{W}_\delta, M_{\delta}\gg_t&=\int_0^t Q\circ \overline{G}_\delta^{\ast}\,ds
\end{align}
for all $0<\delta\leq 1$.
\end{lemma}
\begin{proof}
We recall that
\[\mathcal{L}\left(M_{\delta}(\cdot)\right)=\mathcal{L}\left(\int_0^{\cdot} G(\max(\psi,u_{\delta}))\,dW\right)\]
for all $0<\delta\leq 1$. Then, the assertion follows with similar arguments as in the proof of Lemma \ref{230105_lem2}.
\end{proof}
\begin{lemma}
For all $0<\delta\leq 1$ and all $t\in [0,T]$ we have
\[M_{\delta}(t)=\int_0^t \overline{G}_\delta\,d\overline{W}_\delta \]
in $L^2(\overline{\Omega};L^2(D))$. In particular,
\begin{align*}
&d\overline{u}_\delta + \delta \partial J(\overline{u}_\delta)\,dt - \Div a(\max(\overline{\psi}_\delta,\overline{u}_\delta),\nabla \overline{u}_\delta)\,dt -\frac1\varepsilon\penalisation{\overline{u}_\delta-\overline{\psi}_\delta}\,dt = \overline{f}_\delta\,dt + \overline{G}_\delta\,d\overline{W}_\delta
\end{align*}
or equivalently
\begin{align*}
\partial_t \left(\overline{u}_\delta - \int_0^\cdot\overline{G}_\delta\, d\overline{W}_\delta\right) + \delta \partial J(\overline{u}_\delta) - \Div a(\max(\overline{\psi}_\delta,\overline{u}_\delta),\nabla \overline{u}_\delta) -\frac1\varepsilon\penalisation{\overline{u}_\delta-\overline{\psi}_\delta} = \overline{f}_\delta.
\end{align*}
\end{lemma}
\begin{proof}
We adapt the ideas from \cite[Appendix A]{Hof13}. Let $(\mathfrak{e}_k)_{k\in\mathbb{N}}$ be an orthonormal basis of $L^2(D)$. Then, for all $0<\delta\leq 1$ and all $t\in [0,T]$ we have
\begin{align}\label{230306_03}
\erws{\left\Vert M_{\delta}(t)-\int_0^t \overline{G}_{\delta}\,d\overline{W}_\delta\right\Vert^2_{L^2(D)}}=\sum_{k=1}^{\infty}\erws{\left(M_{\delta}(t)-\int_0^t \overline{G}_{\delta}\,d\overline{W}_\delta,\mathfrak{e}_k\right)^2_{L^2(D)}}
\end{align}
where, for any $k\in\mathbb{N}$,
\begin{align*}
&\erws{\left(M_{\delta}(t)-\int_0^t \overline{G}_{\delta}\,d\overline{W}_\delta,\mathfrak{e}_k\right)^2_{L^2(D)}}
\\
=&\erws{\left(M_{\delta}(t),\mathfrak{e}_k\right)^2_{L^2(D)}}-2\erws{\left(M_{\delta}(t),\mathfrak{e}_k\right)_{L^2(D)}\left(\int_0^t \overline{G}_{\delta}\,d\overline{W}_\delta,\mathfrak{e}_k\right)_{\! \! L^2(D)}}+\erws{\left(\int_0^t \overline{G}_{\delta}\,d\overline{W}_\delta,\mathfrak{e}_k\right)^2_{L^2(D)}}.
\end{align*}
From the properties of the stochastic It\^{o} integral and from \eqref{230306_01} it follows that
\begin{align*}
&\sum_{k=1}^{\infty}\erws{\left(\int_0^t \overline{G}_{\delta}\,d\overline{W}_\delta,\mathfrak{e}_k\right)^2_{L^2(D)}}=\sum_{k=1}^{\infty}\erws{\left(M_{\delta}(t),\mathfrak{e}_k\right)^2_{L^2(D)}}\\
=&\erws{\int_0^t \operatorname{Tr}\,(\overline{G}_{\delta}\circ Q^{1/2})\circ(\overline G_{\delta}\circ Q^{1/2})^{\ast}\,ds}.
\end{align*}
From the properties of the stochastic It\^{o} integral and from \eqref{230306_02} using \cite[Th. 3.12, p.12]{Pardoux} further follows that
\begin{align*}
&\sum_{k=1}^{\infty}\erws{\left(M_{\delta}(t),\mathfrak{e}_k\right)_{L^2(D)}\left(\int_0^t \overline{G}_{\delta}\,d\overline{W}_\delta,\mathfrak{e}_k\right)_{L^2(D)}}=\erws{\operatorname{Tr}\,\ll M_{\delta}(\cdot),\int_0^{\cdot} \overline{G}_{\delta}\,d\overline{W}_\delta\gg_t}\\
=&\erws{\int_0^t\operatorname{Tr}\,(\overline{G}_{\delta}\circ Q^{1/2})\circ(\overline G_{\delta}\circ Q^{1/2})^{\ast}\,ds}
\end{align*}
and therefore \eqref{230306_03} is equal to zero. Now, the rest of the claim follows from \eqref{230306_04}.
\end{proof}
Therefore, applying the uniqueness result from Section \ref{higher order problem}, we may conclude that, for any $0<\delta\leq1$, $\overline{u}_\delta$ is the unique solution to Problem \ref{Pblme_delta} with initial datum $\overline{u}_0^\delta$ with respect to the stochastic basis $(\overline \Omega, \overline{\mathcal{F}}, (\overline{\mathcal{F}}^\delta_t)_{t\in [0,T]},\overline{\mathbb{P}},\overline W_\delta)$.
\begin{lemma}\label{Lemma 250416_01}
For $\delta\rightarrow 0^+$,
\[\int_0^{\cdot} \overline{G}_{\delta}\, d\overline{W}_{\delta}\rightarrow \int_0^{\cdot} \overline{G}_{\infty}\,d\overline{W}_{\infty}\]
in probability in $L^2(0,T;L^2(D))$.
\end{lemma}
\begin{proof}
From Lemma \ref{230105_lem1} it follows that $(\overline G_{\delta})_{\delta}$ is a collection of $\overline{\mathcal{F}}^\delta_t$-predictable processes for $0<\delta\leq 1$, and $\delta=\infty$, respectively. Thanks to the convergence results  $ii.)$ and $vi.)$, the assertion follows from \cite[Lemma 2.1, p. 1126]{DG-HT}
\end{proof}
\begin{lemma}\label{250416_02}
	For $\delta\rightarrow 0^+$, $\overline\sigma_{\delta}(\max(\overline{\psi}_{\delta},\overline{u}_{\delta}))\rightarrow\overline{\sigma}_{\infty}(\max(\overline{\psi}_{\infty},\overline{u}_{\infty}))$ $\overline{\mathbb{P}}$-a.s. in $L^2(0,T;HS(H))$. Moreover,
	$\overline{G}_{\infty}=\overline{G}_0^{\infty}+\overline{\sigma}_{\infty}(\max(\overline{\psi}_{\infty},\overline{u}_{\infty}))$.
	\end{lemma}
\begin{proof}
For any $0<\delta\leq 1$ we have 
	\begin{align}\label{250411_01}
		\begin{aligned}
	&	\int_0^T \Vert \overline{\sigma}_{\delta}(\max(\overline{\psi}_{\delta},\overline{u}_{\delta}))- \overline{\sigma}_{\infty}(\max(\overline{\psi}_{\infty},\overline{u}_{\infty}))\Vert^2_{HS}\,dt\\
	&\leq\int_0^T\left(\Vert \overline{\sigma}_{\delta}(\max(\overline{\psi}_{\delta},\overline{u}_{\delta}))- \overline{\sigma}_{\delta}(\max(\overline{\psi}_{\infty},\overline{u}_{\infty}))\Vert_{HS}+\Vert \overline{\sigma}_{\delta}(\max(\overline{\psi}_{\infty},\overline{u}_{\infty}))- \overline{\sigma}_{\infty}(\max(\overline{\psi}_{\infty},\overline{u}_{\infty}))\Vert_{HS}   \right)^2\,dt
	\end{aligned}
	\end{align}
	Now, considering the first term on the last line of the above equation, we have thanks to H$_{3,2}$ and equality of laws, 
	\begin{align*}
		\Vert \overline{\sigma}_{\delta}(\max(\overline{\psi}_{\delta},\overline{u}_{\delta}))- \overline{\sigma}_{\delta}(\max(\overline{\psi}_{\infty},\overline{u}_{\infty}))\Vert_{HS}\leq \sqrt{C_{\sigma}}\Vert \max(\overline{\psi}_{\delta},\overline{u}_{\delta})-\max(\overline{\psi}_{\infty},\overline{u}_{\infty})\Vert_{L^2}.
		\end{align*}
        Next, 
		\begin{align}\label{250411_02}
			\begin{aligned}
		&	\Vert \max(\overline{\psi}_{\delta},\overline{u}_{\delta})-\max(\overline{\psi}_{\infty},\overline{u}_{\infty})\Vert_{L^2}
		\leq 
        \Vert \overline{\psi}_{\delta}-\overline{\psi}_{\infty}\Vert_{L^2}+\Vert \overline{u}_{\delta}-\overline{u}_{\infty}\Vert_{L^2}.
		\end{aligned}
		\end{align}
		Plugging \eqref{250411_02} into \eqref{250411_01} 
        we find a constant $C\geq 0$ not depending on $0<\delta\leq 1$ such that
		\begin{align}\label{250411_03}
			\begin{aligned}
			&\int_0^T \Vert \overline{\sigma}_{\delta}(\max(\overline{\psi}_{\delta},\overline{u}_{\delta}))- \overline{\sigma}_{\infty}(\max(\overline{\psi}_{\infty},\overline{u}_{\infty}))\Vert^2_{HS}\,dt\\
			\leq& C\int_0^T\left(\Vert \overline{\psi}_{\delta}-\overline{\psi}_{\infty}\Vert^2_{L^2}+\Vert \overline{u}_{\delta}-\overline{u}_{\infty}\Vert^2_{L^2}
            \right)\,dt\\
			&+C\int_0^T\Vert (\overline{\sigma}_{\delta}- \overline{\sigma}_{\infty})(\max(\overline{\psi}_{\infty},\overline{u}_{\infty}))\Vert^2_{HS}  \,dt.
			\end{aligned}
		\end{align}
		Thanks to the convergence results $i.)$ and $iii.)$, the first line of the right-hand side of \eqref{250411_03} converges to $0$, for $\delta\rightarrow 0^+$, $\overline{\mathbb{P}}$-a.s. Moreover, concerning the second term on the right-hand side of \eqref{250411_03}, from H$_3$ it follows that
		\begin{align*}
			&\int_0^T\Vert (\overline{\sigma}_{\delta}- \overline{\sigma}_{\infty})(\max(\overline{\psi}_{\infty},\overline{u}_{\infty}))\Vert^2_{HS} \,dt=\int_0^T\Vert (\overline{\sigma}_{\delta}- \overline{\sigma}_{\infty})(t)[\max(\overline{\psi}_{\infty},\overline{u}_{\infty})]\Vert^2_{HS}\,dt \\
			&\leq\int_0^T \Vert (\overline{\sigma}_{\delta}- \overline{\sigma}_{\infty})(t)\Vert^2_{\operatorname{Lip}_0(H;HS(H))}\Vert \max(\overline{\psi}_{\infty},\overline{u}_{\infty})\Vert^2_{L^2}\,dt.
		\end{align*}
		Recalling that $\Sigma$ is a sub-Banach space of $\operatorname{Lip}_0(H;HS(H))$ with respect to the same norm (see H$_{3,3}$), using H\"older's inequality 
        as well as the inequality $(a+b)^p\leq 2^{p-1}(a^p+b^p)$ we may find constants $C_1,C_2\geq 0$ such that
		\begin{align*}
			&\int_0^T\Vert (\overline{\sigma}_{\delta}- \overline{\sigma}_{\infty})(\max(\overline{\psi}_{\infty},\overline{u}_{\infty}))\Vert^2_{HS} \,dt\leq \int_0^T\Vert (\overline{\sigma}_{\delta}- \overline{\sigma}_{\infty})(t)\Vert^2_{\Sigma}\Vert \max(\overline{\psi}_{\infty},\overline{u}_{\infty})\Vert^2_{L^2}\,dt\\
			&\leq \left(\int_0^T\Vert (\overline{\sigma}_{\delta}- \overline{\sigma}_{\infty})(t)\Vert^4_{\Sigma}\,dt\right)^{1/2}\left(\int_0^T\Vert \max(\overline{\psi}_{\infty},\overline{u}_{\infty})\Vert^4_{L^2}\,dt\right)^{1/2}\\
			&\leq \left(\int_0^T\Vert (\overline{\sigma}_{\delta}- \overline{\sigma}_{\infty})(t)\Vert^4_{\Sigma}\,dt\right)^{1/2}\left(C_1\int_0^T \Vert \overline{\psi}_{\infty}\Vert_{L^2}^4+C_2 \Vert \overline{u}_{\infty}\Vert_{L^2}^4\, dt.\right)^{1/2}
		\end{align*}
		Now, from the convergence results $v.)$ it follows that the second term on the right-hand side of \eqref{250411_03} converges to $0$, for $\delta\rightarrow 0^+$, $\overline{\mathbb{P}}$-a.s. Hence, the convergence  of $\overline{\sigma}_{\delta}(\max(\overline{\psi}_{\delta},\overline{u}_{\delta}))$ towards $\overline{\sigma}_{\infty}(\max(\overline{\psi}_{\infty},\overline{u}_{\infty}))$ $\overline{\mathbb{P}}$-a.s. in $L^2(0,T;HS(H))$ may be concluded from \eqref{250411_03}. Combining this result with the convergence $iv.)$, $ii.)$, and \eqref{250416_01}, the rest of the assertion follows.
	\end{proof}
\begin{defi}
	In the following, in view of Lemma \ref{250416_02} and \eqref{250416_01} we will write
    \[\overline{G}_{\delta}(\overline{u}_{\delta}):=\overline{G}_{\delta}=\overline{G}_0^{\delta}+\overline{\sigma}_{\delta}(\max(\overline{\psi}_{\delta},\overline{u}_{\delta})).\]
    for $\delta \in [0,1) \cup \{+\infty\}$.
	\end{defi}

\begin{lemma}\label{230406_02}
$\overline{u}_{\infty}$ is a $(\overline{\mathcal{F}}_t^{\infty})_{t\in [0,T]}$-adapted, square-integrable stochastic process with continuous paths in $L^2(D)$. Moreover, $\overline{u}_{\infty}\in L^p(\overline\Omega \times (0,T);V)$ and there exists $A_{\infty}\in L^{p'}(\overline\Omega; L^p(0,T;L^p(D)))$ such that
\begin{align}\label{230306_05}
d\overline{u}_\infty  - \operatorname{div} \, A_\infty \,dt -\frac1\varepsilon\penalisation{\overline{u}_\infty-\overline{\psi}_{\infty}}\,dt = \overline{f}_{\infty} \,dt + \overline{G}_{\infty}(\overline{u}_\infty)\,d\overline{W}_{\infty}
\end{align}
or, equivalently,
\begin{align*}
\partial_t \left( \overline{u}_\infty - \int_0^{\cdot}\overline{G}_{\infty}(\overline{u}_\infty) d\overline{W}_{\infty}\right) - \operatorname{div}\, A_\infty -\frac1\varepsilon\penalisation{\overline{u}_\infty-\overline{\psi}_{\infty}} = \overline{f}_{\infty}.
\end{align*}
In addition, It\^{o}'s energy formula holds.
\end{lemma}
\begin{proof} 
Note first that the estimates of Lemma \ref{lem_estimations} hold true for $\overline{u}_\delta$; then, since  $\overline u_\delta$ converges to $\overline{u}_\infty$ in $L^4(0,T;L^2(D))$ a.s. in $\overline\Omega$,  $\overline u_\delta$ converges to $\overline{u}_\infty$ in $L^s(\overline\Omega;L^4(0,T;L^2(D)))$ for any $s<2$ by Vitali's converging theorem. Then, up to a subsequence, $\overline{u}_\delta$ converges to $\overline{u}_\infty$ in the following modes: a.e. in $\overline\Omega \times (0,T) \times D$;  in $L^2(D)$ a.e. in $\overline\Omega \times (0,T)$; and weakly in $L^2(\overline\Omega;L^4(0,T;L^2(D)))$ and weakly in $L^p(\overline\Omega \times (0,T);V)$. Moreover, by Assumption H$_3$, $\overline{G}_\delta(\overline{u}_\delta)$ converges a.s. to $\overline{G}_\infty(\overline{u}_\infty)$ in $L^2(0,T;HS(H))$ and therefore weakly in $L^2(\overline\Omega \times (0,T);HS(H))$.
\\
By Lemma \ref{Lemma 250416_01}, Lemma \ref{lem_estimations} and Vitali's converging theorem, $\int_0^{\cdot} \overline{G}_{\delta}(\overline{u}_\delta)\, d\overline{W}_{\delta}\rightarrow \int_0^{\cdot} \overline{G}_{\infty}(\overline{u}_\infty)\,d\overline{W}_{\infty}$ 
in $L^s(\overline{\Omega},L^2(0,T;L^2(D)))$ for any $s<2$. Therefore, the convergence holds in $L^s(0,T;L^s(\overline{\Omega};L^2(D)))$, and, for a not relabeled subsequence, there exists $Z \subset(0,T)$ of full measure, such that for any $r \in Z$, $\int_0^{r} \overline{G}_{\delta}(\overline{u}_\delta)\, d\overline{W}_{\delta}\rightarrow \int_0^{r} \overline{G}_{\infty}(\overline{u}_\infty)\,d\overline{W}_{\infty}$ 
in $L^s(\overline{\Omega};L^2(D))$. Let $t \in [0,T]$ and $r \in Z$. W.l.o.g. we can assume $t \leq r$ (otherwise the limits of some upcoming integrals need to be exchanged and the expression $\sqrt{r-t}$ needs to be exchanged by $\sqrt{t-r}$). H\"older's inequality and It\^o's isometry yield
\begin{align*}
&\erws{\|\int_0^{r} \overline{G}_{\delta}(\overline{u}_\delta)\, d\overline{W}_{\delta}-\int_0^{t} \overline{G}_{\delta}(\overline{u}_\delta)\, d\overline{W}_{\delta}\|_{L^2(D)}^s} 
\leq    
\Big(\erws{\|\int_{t}^{r} \overline{G}_{\delta}(\overline{u}_\delta)\, d\overline{W}_{\delta}\|_{L^2(D)}^2}\Big)^{s/2}
\\
& \leq
\Big(\erws{\int_{t}^{r} \|\overline{G}_{\delta}(\overline{u}_\delta)\|^2_{L^2(D)}d\tau}\Big)^{s/2}
\leq
C\Big(\erws{\int_{t}^{r} \|\overline{G}_0^{\delta}\|^2_{L^2(D)}d\tau+ C_\sigma^2\int_{t}^{r} \|\overline{u}_\delta\|^2_{L^2(D)}d\tau}\Big)^{s/2}.
\end{align*}
One has that $\erws{\int_{t}^{r} \|\overline{G}_0^{\delta}\|^2_{L^2(D)}d\tau}=\erww{\int_{t}^{r} \|G_{0}\|^2_{L^2(D)}d\tau}$ and
\begin{align*}
\erws{\int_{t}^{r} \|\overline{u}_\delta\|^2_{L^2(D)}d\tau}
\leq 
\sqrt{r-t}\,\erws{\Big(\int_{t}^{r} \|\overline{u}_{\delta}\|^4_{L^2(D)}d\tau\Big)^{1/2}}
\leq 
\sqrt{r-t}\,\erws{\|\overline{u}_{\delta}\|^2_{L^4(0,T;L^2(D))}}
\leq C \sqrt{r-t}.
\end{align*}
Since the same inequalities hold for $\overline{G}_{\infty}$ and $\overline{u}_{\infty}$, one gets 
\begin{align*}
&\erws{\|\int_0^{t} \overline{G}_{\infty}(\overline{u}_\infty)\, d\overline{W}_{\infty}-\int_0^{t} \overline{G}_{\delta}(\overline{u}_\delta)\, d\overline{W}_{\delta}\|_{L^2(D)}^s}
\\ \leq & C\erws{\|\int_0^{r} \overline{G}_{\infty}(\overline{u}_\infty)\, d\overline{W}_{\infty}-\int_0^{t} \overline{G}_{\infty}(\overline{u}_\infty)\, d\overline{W}_{\infty}\|_{L^2(D)}^s}
+ C\erws{\|\int_0^{r} \overline{G}_{\infty}(\overline{u}_\infty)\, d\overline{W}_{\infty}-\int_0^{r} \overline{G}_{\delta}(\overline{u}_\delta)\, d\overline{W}_{\delta}\|_{L^2(D)}^s}
\\ &+
C\erws{\|\int_0^{r} \overline{G}_{\delta}(\overline{u}_\delta)\, d\overline{W}_{\delta}-\int_0^{t} \overline{G}_{\delta}(\overline{u}_\delta)\, d\overline{W}_{\delta}\|_{L^2(D)}^s}
\\ &\leq C\Big(\sqrt{r-t}+\erww{\int_{t}^{r} \|G_{0}\|^2_{L^2(D)}d\tau}\Big)^{s/2}+ C\erws{\|\int_0^{r} \overline{G}_{\infty}(\overline{u}_\infty)\, d\overline{W}_{\infty}-\int_0^{r} \overline{G}_{\delta}(\overline{u}_\delta)\, d\overline{W}_{\delta}\|_{L^2(D)}^s}.
\end{align*}
Assuming that $r \in Z$, one gets, for any $t \in [0,T]$,  
\begin{align*}
\limsup_{\delta}\erws{\|\int_0^{t} \overline{G}_{\infty}(\overline{u}_\infty)\, d\overline{W}_{\infty}-\int_0^{t} \overline{G}_{\delta}(\overline{u}_\delta)\, d\overline{W}_{\delta}\|_{L^2(D)}^s} 
\leq     
C\Big(\sqrt{r-t}+\erww{\int_{t}^{r} \|G_{0}\|^2_{L^2(D)}d\tau}\Big)^{s/2}.
\end{align*}
Then, by density of $Z$ in $[0,T]$, one concludes the convergence of $\int_0^{t} \overline{G}_{\delta}(\overline{u}_\delta)\, d\overline{W}_{\delta}$ to $\int_0^{t} \overline{G}_{\infty}(\overline{u}_\infty)\, d\overline{W}_{\infty}$ in $L^s(\overline{\Omega};L^2(D))$, for any $t \in [0,T]$. Since by It\^o's isometry
\begin{align*}
\erws{\|\int_0^{t} \overline{G}_{\delta}(\overline{u}_\delta)\, d\overline{W}_{\delta}\|_{L^2(D)}^2}
\leq 
\erws{\int_{0}^{T} \|\overline{G}_{\delta}(\overline{u}_\delta)\|_{L^2(D)}^2 \, dt }  \leq C,
\end{align*}
one may conclude the weak convergence of $\int_0^{t} \overline{G}_{\delta}(\overline{u}_\delta)\, d\overline{W}_{\delta}$ to  $\int_0^{t} \overline{G}_{\infty}(\overline{u}_\infty)\,d\overline{W}_{\infty}$ in $L^2(\overline{\Omega};L^2(D))$ for any $t\in[0,T]$.
\\
Recall that $\overline u_\delta - \int_0^{\cdot} \overline{G}_{\delta}(\overline{u}_\delta)\,d\overline W_{\delta}$ is bounded in $L^{\nu'}(\overline \Omega;\mathbb{W})$ and converges  to $\overline u_\infty - \int_0^\cdot \overline{G}_{\infty}(\overline{u}_\infty)\,d\overline W_\infty$ weakly in $L^{\nu'}(\overline \Omega;\mathbb{W})$ and in $L^{\nu'}(\overline \Omega;C([0,T];\mathcal{V}'))$ by continuous embedding  of $\mathbb{W}$ in $C([0,T];\mathcal{V}')$. 
\\
In particular, $\overline u_\delta$ converges weakly to $\overline u_\infty$ in $L^{\nu'}(\overline \Omega;C([0,T];\mathcal{V}'))$ and, for any $t$, $\overline u_\delta(t)$ converges weakly to $\overline u_\infty(t)$ in $L^{\nu'}(\overline \Omega;\mathcal{V}')$. 
Remember that $\overline u_\delta$ is bounded in $L^2(\overline \Omega;C([0,T];L^2(D)))$ to conclude that the weak convergence of $\overline u_\delta(t)$  to $\overline u_\infty(t)$ holds in $L^{2}(\overline \Omega;L^2(D))$.\\
Thanks to the growth assumption \eqref{croissance1} on $a$, up to a subsequence $a(\max(\overline{\psi},\overline{u}_\delta),\nabla \overline{u}_\delta)$ converges weakly to some $A_\infty$ in $L^{p'}(\overline\Omega; L^p(0,T;L^p(D)))$. Consequently, from 
\begin{align*}
&\overline u_\delta(t) + \delta \int_0^t \partial J(\overline u_\delta)\,ds - \int_0^t\Div a(\max(\overline \psi_\delta,\overline{u}_\delta),\nabla \overline{u}_\delta)\,ds -\frac1\varepsilon\int_0^t \penalisation{\overline{u}_\delta-\overline{\psi}_{\delta}}\,dx\,ds 
\\ 
=&
\overline u_0^{\delta} + \int_0^t \overline f_{\delta}\,ds + \int_0^t \overline{G}(\overline{u}_\delta)\, d\overline{W}_{\delta},
\end{align*}
and the above convergence results, we may deduce \eqref{230306_05}.
In particular, \cite[Theorem 4.2.5, p. 91]{Liu-Rock} yields $\overline{u}_\infty \in L^2(\overline\Omega; C([0,T]; L^2(D)))$, as well as It\^o's energy formula.\\
\end{proof}
\begin{lemma}
We have 
$
A_{\infty}=a(\max(\overline\psi_\infty,\overline{u}_\infty),\nabla \overline{u}_\infty)
$
 in $L^{p'}(\overline{\Omega}; L^{p'}(0,T;L^{p'}(D)))$. Especially, $A_{\infty}$ has a predictable representation and $A_{\infty} \in L^{p'}(\overline\Omega \times (0,T); L^{p'}(D))$.
\end{lemma}

\begin{proof}
By It\^o's energy formula with exponential weight, for any $\beta \in \R^+$ and any $t \in [0,T]$,
\begin{align*}
&\frac{e^{-\beta t}}2\erws{\|\overline u_\delta(t)\|^2_{L^2}} + \delta \erws{\int_0^t \!\!\!e^{-\beta s} b(\overline u_\delta,\overline u_\delta)\,ds} + \erws{\int_0^t \!\!\! e^{-\beta s} \!\!\!\!\int_{D}\!\!\! a(\max(\overline{\psi}_{\delta},\overline{u}_\delta),\nabla \overline{u}_\delta)\nabla \overline u_\delta -\frac1\varepsilon\penalisation{\overline{u}_\delta-\overline{\psi}_{\delta}}\overline u_\delta \,dx\,ds} \\ 
=&\erws{\int_0^te^{-\beta s}\langle \overline f_{\delta},\overline u_\delta\rangle\,ds} + \frac12\erws{\int_0^t e^{-\beta s}\|\overline{G}_{\delta}(\overline{u}_\delta)\|_{HS}^2\,ds} + \frac12\erws{\|\overline u_0^{\delta}\|^2_{L^2}}-\frac\beta2\erws{\int_0^t e^{-\beta s}\|\overline u_\delta\|_{L^2}^2\,ds}.
\end{align*}
Note that: 
\begin{align*}
&b(\overline u_\delta,\overline u_\delta) \geq 0; \qquad  
\erws{\int_0^te^{-\beta s}\langle \overline f_{\delta},\overline u_\delta\rangle\,ds}  \to  \erws{\int_0^te^{-\beta s}\langle \overline f_{\infty},\overline u_\infty\rangle\,ds};
\\ &
\erws{\|\overline u_\infty(t)\|^2_{L^2}}\leq \liminf_{\delta\rightarrow 0^+} \erws{\|\overline u_\delta(t)\|^2_{L^2}}.
\end{align*}
Moreover, we have
\begin{align*}
-\frac1\varepsilon\penalisation{\overline{u}_\delta-\overline{\psi}_{\delta}}\overline u_\delta =\frac1\varepsilon|(\overline u_\delta -\overline \psi_{\delta} )^-|^2 -\frac1\varepsilon\penalisation{\overline{u}_\delta-\overline{\psi}_{\delta}}\overline \psi_{\delta}.
\end{align*}
By equality of laws, the weak convergence $\overline u_{\delta} \rightharpoonup \overline u_{\infty}$ in $L^2(\overline \Omega; L^2(0,T;L^2(D)))$ and
a.s. convergences, Fatou's lemma yields 
\begin{align*}
\liminf_{\delta\rightarrow 0^+ }-\frac1\varepsilon\erws{\int_0^t \int_D e^{-\beta s}  \penalisation{\overline{u}_\delta-\overline{\psi}_{\delta}}\overline u_\delta \,dx\,ds} \geq -\frac1\varepsilon\erws{\int_0^t \int_D e^{-\beta s}  \penalisation{\overline{u}_\infty-\overline{\psi}_{\infty}}\overline u_\infty \,dx\,ds}
\end{align*}
and finally
\begin{align*}
&\frac12\erws{\int_0^t e^{-\beta s}\|\overline{G}_{\delta}(\overline{u}_\delta)\|_{HS}^2\,ds}-\frac\beta2\erws{\int_0^t e^{-\beta s}\|\overline u_\delta\|_{L^2}^2\,ds}
\\ =&
\frac12\erws{\int_0^t e^{-\beta s}\Big(\|\overline{G}_{\delta}(\overline{u}_\delta)-\overline{G}_{\infty}(\overline{u}_\infty)\|_{HS}^2-\beta \|\overline u_\delta-\overline u_\infty\|_{L^2}^2\Big)\,ds}
\\& +
\erws{\int_0^t e^{-\beta s}\Big(\langle\overline{G}_{\delta}(\overline{u}_\delta),\overline{G}_{\infty}(\overline{u}_\infty)\rangle-\beta (\overline u_\delta,\overline u_\infty)\Big)\,ds}
\\&
-\frac12\erws{\int_0^t e^{-\beta s}\Big(\|\overline{G}_{\infty}(\overline{u}_\infty)\|_{HS}^2-\beta \|\overline u_\infty\|_{L^2}^2\Big)\,ds}.
\end{align*}
Then, choosing $\beta$ greater than the Lipschitz constant of $\overline G$, one gets by passing to the superior limit, 
\begin{align*}
&\frac{e^{-\beta t}}2\erws{\|\overline u_\infty(t)\|^2_{L^2}}  + \limsup_{\delta\rightarrow 0^+ }\erws{\int_0^t \!\!\!\! e^{-\beta s}\!\!\! \int_{D}\!\!\! a(\max(\overline{\psi}_{\delta},\overline{u}_\delta),\nabla \overline{u}_\delta)\nabla \overline u_\delta \,dx\,ds} - \frac1\varepsilon \erws{\int_0^t\!\!\! \!e^{-\beta s}\!\!\! \int_{D} \!\!\penalisation{\overline{u}_\infty-\overline{\psi}_{\infty}}\overline u_\infty \,dx\,ds}
\\  &\leq
\erws{\int_0^t\!\!e^{-\beta s}\langle \overline f_{\infty},\overline u_\infty\rangle\,ds} + \frac12\erws{\int_0^t \!\!e^{-\beta s}\|\overline{G}_{\infty}(\overline{u}_\infty)\|_{HS}^2\,ds} + \frac12\erws{\|\overline u_0^{\infty}\|^2_{L^2}}-\frac\beta2\erws{\int_0^t \!\!e^{-\beta s}\|\overline u_\infty\|_{L^2}^2\,ds}.
\end{align*}
Since It\^o's energy formula with exponential weight holds for the limit process $\overline u_\infty$, one has also 
\begin{align*}
&\frac{e^{-\beta t}}2\erws{\|\overline u_\infty(t)\|^2_{L^2}}  +\erws{\int_0^t \!\! e^{-\beta s} \int_{D} A_\infty\nabla \overline u_\infty \,dx\,ds} - \frac1\varepsilon \erws{\int_0^t e^{-\beta s} \int_{D} \penalisation{\overline{u}_\infty-\overline{\psi}_{\infty}}\overline u_\infty \,dx\,ds}
\\  &=
\erws{\int_0^te^{-\beta s}\langle \overline f_{\infty},\overline u_\infty\rangle\,ds} + \frac12\erws{\int_0^t e^{-\beta s}\|\overline{G}_{\infty}(\overline{u}_\infty)\|_{HS}^2\,ds} + \frac12\erws{\|\overline u_0^{\infty}\|^2_{L^2}}-\frac\beta2\erws{\int_0^t e^{-\beta s}\|\overline u_\infty\|_{L^2}^2\,ds}
\end{align*}
and, by subtraction and setting $t=T$
\begin{align*}
\limsup_{\delta\rightarrow 0^+ } \erws{\int_0^Te^{-\beta t}\int_{D} a(\max(\overline{\psi}_{\delta},\overline{u}_\delta),\nabla \overline{u}_\delta)\nabla \overline u_\delta \,dx\,dt}
\leq  
\erws{\int_0^Te^{-\beta t}\int_{D} A_\infty \nabla \overline u_\infty \,dx\,dt}. 
\end{align*}
We recall that, a.e. in $\overline\Omega \times (0,T)$, $\overline u_\delta(\omega,t)$ converges to $\overline{u}_\infty(\omega,t)$ in $L^2(D)$ and that it is bounded in \linebreak $L^2(\overline\Omega; L^4(0,T;L^2(D)))$ and $L^p(\overline\Omega \times (0,T);V)$. Thus, following \cite[Lemma 8.8, p. 208]{Roubicek}, revisited in \cite{Vallet-Zimm1} in the stochastic framework, one is able to apply a stochastic version of Minty's trick to identify $A_\infty=a(\max(\overline{\psi}_{\infty},\overline{u}_\infty),\nabla \overline{u}_\infty)$; the idea being that under suitable assumptions, if an operator is pseudomonotone on $V$, it will be pseudomonotone on $L^p(\overline\Omega \times (0,T);V)$ as well.  
\end{proof}

\subsection{Uniqueness and existence of strong solutions}

Let $\eps >0$ and let us first prove the following  result of uniqueness.
\begin{prop}\label{uniqueness}
 Assume that $W$ is a $(\mathcal{F}_t)$-adapted $Q$-Wiener process with values in $U$ with respect to the stochastic basis $(\Omega,\mathcal{F},(\mathcal{F}_t),\mathds{P})$ and $(\psi, f, G)$ given. If $u_1^\eps, u_2^\eps$ are solutions to \eqref{penalization} with respect  to the initial values $u_{01}$ and $u_{02}$ in $L^2(D)$ respectively on $(\Omega,\mathcal{F},(\mathcal{F}_t),\mathds{P})$, then we have
 \begin{align}
 \erww{\int_D\vert u_1^\eps(t)-u_2^\eps(t)\vert \,dx} \leq \erww{\int_D\vert u_{01}-u_{02}\vert \,dx} 
 \end{align}
 for all $t \in [0,T]$.
\end{prop}
\begin{proof}
For simplicity, we will write $u_i$ instead $u_i^\eps ,\quad  i=1,2$. For $\delta >0$, define the  regular approximation of the absolute value $N_ \delta(r):=\int_0^r\eta_\delta(s)ds$, where $\eta_\delta$ is an odd, regular, nondecreasing Lipschitz approximation of the sign function, vanishing at $0$, with Lipschitz constant $\dfrac{2c}{e\delta}$ where $c$ is a positive constant (see the proof of \cite[Prop. 5]{Vallet-Zimm2}). Applying It\^o's formula with $F_{\delta}(u)=\int_D N_\delta(u)\,dx$ (see, \textit{e.g.}, \cite[Th. 5.3]{Pardoux}) to the process
\begin{align*}
u_1(t)-u_2(t)&-\dfrac{1}{\varepsilon}\int_0^t(\penalisation{u_1-\psi}- \penalisation{u_2-\psi})\,ds-\int_0^t(\Div{\tilde{{a}}(u_1,\nabla u_1)}-\Div\tilde{{a}}(u_2,\nabla u_2))\,ds\\
 &=(u_{01}-u_{02})+\int_0^t(\tilde{{G}}(u_1)-\tilde{{G}}(u_2))\,dW
 \end{align*}
 yields
 \begin{align*}
&\int_DN_\delta(u_1-u_2)(t)\,dx-\int_DN_\delta(u_{01}-u_{02})\,dx-\dfrac{1}{\varepsilon}\int_0^t\int_D\left(\penalisation{u_1-\psi}- \penalisation{u_2-\psi}\right)N_\delta^\prime(u_1-u_2)\,dx\,ds\\
&+\int_0^t \int_D\left(\tilde{{a}}(u_1,\nabla u_1)- \tilde{{a}}(u_2,\nabla u_2)\right)\cdot \nabla(u_1-u_2)N_\delta^{\prime\prime}(u_1-u_2)\,dx\,ds\\
=&\int_0^t \left(\tilde{{G}}(u_1)-\tilde{{G}}(u_2)\,dW,N_\delta^\prime(u_1-u_2)\right)_H\\
&
+\dfrac{1}{2}\operatorname{Tr}\int_0^t F''_{\delta}(u_1 - u_2)(\tilde{G}(u_1) - \tilde{G}(u_2))Q(\tilde{G}(u_1) - \tilde{G}(u_2))^{\ast}\,ds
\\
\Leftrightarrow& I_1+I_2+I_3=I_4+I_5,
\end{align*}
where, with a slight abuse of notation, we still denote the multiplication operator associated to $F''_{\delta}(u_1-u_2)$ by $F''_{\delta}(u_1-u_2)$.
Since, for any $t\in[0,T], N_\delta(u_1-u_2)(t)$ converges to $\vert (u_1-u_2)(t)\vert$ for $\delta \to 0^+$ a.e. in $\Omega\times D$ and $\vert N_\delta(u_1-u_2)(t)\vert \leq \vert (u_1-u_2)(t)\vert$ for all $\delta >0$ a.s. in $\Omega\times D$, it follows that
\begin{align*}
\lim_{\delta \to 0^+}\erww{I_1}=\erww{\int_D\vert (u_1-u_2)(t)\vert\,dx}-\erww{\int_D \vert u_{01}-u_{02}\vert \,dx}.
 \end{align*}
$I_2$ is non-negative by monotonicity of the penalization procedure and since $N_\delta$ is convex, and by using $H_{2,2}$, it follows that
\begin{align*}
&\left(\tilde{{a}}(u_1,\nabla u_1)- \tilde{{a}}(u_2,\nabla u_2)\right)\cdot \nabla(u_1-u_2)N_\delta^{\prime\prime}(u_1-u_2)\\
&\geq \left(\tilde{{a}}(u_1,\nabla u_2)- \tilde{{a}}(u_2,\nabla u_2)\right)\cdot\nabla(u_1-u_2)N_\delta^{\prime\prime}(u_1-u_2).
\end{align*}
 Note that, by using $H_{2,4}$, one has
 \begin{align*}
&\erww{\int_0^t\int_D\left|\left(\tilde{{a}}(u_1,\nabla u_2)- \tilde{{a}}(u_2,\nabla u_2)\right)\cdot \nabla(u_1-u_2)N_\delta^{\prime\prime}(u_1-u_2)\right| \,dx\,ds} \\
\leq& \frac{2c}{e\delta}\erww{\int_{\{ \vert u_1-u_2 \vert \leq \delta\}}\left( l+C_3^{a}\vert \nabla u_2\vert^{p-1} \right)\vert u_1-u_2\vert \vert \nabla (u_1-u_2)\vert \,dx\,ds}\\
\leq& \frac{2c}{e}\erww{\int_{\{ \vert u_1-u_2 \vert \leq \delta\}}\left( l+C_3^{a}\vert \nabla u_2\vert^{p-1} \right) \vert \nabla (u_1-u_2)\vert \,dx\,ds}.
 \end{align*}
 H\"older's inequality  ensures that
\begin{align*}
&\erww{\int_0^t\int_D\vert\left(\tilde{{a}}(u_1,\nabla u_2)- \tilde{{a}}(u_2,\nabla u_2)\right)\cdot \nabla(u_1-u_2)N_\delta^{\prime\prime}(u_1-u_2)\vert \,dx\,ds}\\
\leq& C\erww{\int_{\{ \vert u_1-u_2 \vert \leq \delta\}}\vert \nabla (u_1-u_2)\vert^p \,dx\,ds}^{1/p} \overset{\delta \to 0^+}{\longrightarrow} 0
 \end{align*}
 where $C\geq0$ is a constant not depending on $\delta>0$. Therefore
 \begin{align*}
 \liminf_{\delta \to 0^+}\erww{I_3} \geq 0.
 \end{align*}
Note that 
\begin{align*}
\erww{I_4}=\erww{\int_0^t\left(\tilde{{G}}(u_1)-\tilde{{G}}(u_2)\,dW,N_\delta^\prime(u_1-u_2) \right)_H}=0
\end{align*} 
for all $t\in [0,T]$. 

According to \cite[Proposition B.0.10.]{Liu-Rock} we have
\begin{align*}
I_5 = &\dfrac{1}{2}\operatorname{Tr}\int_0^t F''_{\delta}(u_1 - u_2)(\tilde{G}(u_1) - \tilde{G}(u_2))Q(\tilde{G}(u_1) - \tilde{G}(u_2))^{\ast}\,ds\\
= &\dfrac{1}{2}\operatorname{Tr}\int_0^t ((\tilde{G}(u_1) - \tilde{G}(u_2))Q^{1/2})^{\ast}F''_{\delta}(u_1 - u_2)(\tilde{G}(u_1) - \tilde{G}(u_2))Q^{1/2}\,ds \\
=&\dfrac{1}{2}\int_0^t \sum\limits_{k=1}^{\infty} \left(\tilde{G}(u_1) - \tilde{G}(u_2)Q^{1/2})^{\ast}F''_{\delta}(u_1 - u_2)(\tilde{G}(u_1) - \tilde{G}(u_2))Q^{1/2}(e_k), e_k \right)_U \, ds \\
= &\dfrac{1}{2}\int_0^t \sum\limits_{k=1}^{\infty} \left(F''_{\delta}(u_1 - u_2)(\tilde{G}(u_1) - \tilde{G}(u_2))Q^{1/2}(e_k), (\tilde{G}(u_1) - \tilde{G}(u_2))Q^{1/2}(e_k) \right)_H \, ds \\
=& \dfrac{1}{2}\int_0^t \int_D N_{\delta}''(u_1 - u_2) \sum\limits_{k=1}^{\infty} | (h_k(\max(\psi,u_1)) - h_k(\max(\psi,u_2))|^2 \, dx \, ds.
\end{align*}

By  using $ H_{3,2}$ it follows that
\begin{align*}
\vert\erww{I_5}\vert &=\frac{1}{2}\left|\erww{\int_0^t\int_DN_\delta^{\prime\prime}(u_1-u_2) \sum_{k=1}^\infty\vert h_k(\max(u_1,\psi))-h_k(\max(u_2,\psi))\vert^2\,dx\,ds}\right|\\
 &\leq C_{\sigma}\frac{c}{e\delta}\erww{\int_{\{ \vert u_1-u_2 \vert \leq \delta\}}\vert u_1-u_2\vert^2 \,dx\,ds} \leq C\delta \overset{\delta \to 0^+}{\longrightarrow} 0
 \end{align*}
and  the result holds.
\end{proof}
In particular, Proposition \ref{uniqueness} implies that whenever a strong solution to \eqref{penalization} exists, it is unique. Moreover, if $\mu^{1,2}$ is the joint law of $(u_1,u_2)$ on $L^4(0,T;L^2(D))^2$, then Proposition \ref{uniqueness} implies that
\begin{align*}
\mu^{1,2}\left\{  (\xi,\varsigma) \in L^4(0,T;L^2(D)^2 \ | \ \xi=\varsigma  \right\}=1.
\end{align*}
Let us recall the following lemma, which  contains a suitable necessary and sufficient condition to get existence of strong solutions, see \cite[Lemma 1.1.]{GK96}.
\begin{lemma}\label{strongsolutionlemma}
 Let $V$ be a Polish space equipped with the Borel $\sigma$-algebra. A sequence of $V$-valued random variables $(X_n)$  converges in probability if and only if for every pair of subsequences $X_n$ and $X_m$ there exists a joint subsequence 
 $(X_{n_k}, X_{m_k})$ which converges for $k \to \infty$ in law to a probability measure $\mu$ such that 
 \begin{align*}
 \mu(\{ (w,z)\in V\times V \vert \quad w=z\})=1.
 \end{align*}
\end{lemma}
As a consequence of Proposition \ref{uniqueness} we get the following result.
\begin{lemma}
 The sequence   $(u_\delta)_\delta$ converges in probability on $(\Omega,\mathcal{F},P)$ in $L^4(0,T;L^2(D))$ and there exists a unique strong solution to \eqref{penalization}.
\end{lemma}

\section{Existence of solution for the obstacle problem in the "regular case"}\label{sec-obstacle-pb}

Following Assumption $H_{5}$, $f - \partial_t \left(\psi - \int_0^{\cdot} G(\psi)\,dW\right)  +\Div a(\psi,\nabla \psi)=h=h^+-h^-$ belongs to the order dual $L^{p}(\Omega_T;V)^*$. In this subsection we impose an additional regularity on $h^-$, namely $h^- \in L^{\alpha}(\Omega_T; L^{\alpha}(D))$, where hereafter $\alpha:=\max(2,p')$. Our aim in this section is to prove the following theorem.

\begin{theorem}\label{bisexistence-regular}
Assume $H_1-H_6$ with $0 \leq h^- \in L^{\alpha}(\Omega_T; L^{\alpha}(D))$. There exists a unique pair of stochastic processes 
\[(u,\rho)\in L^p(\Omega;L^p(0,T;W^{1,p}_0(D))) \times L^\alpha(\Omega_T; L^\alpha( D))\]  
 such that
 \begin{enumerate}
  \item $u \in L^2(\Omega; C([0,T]; L^2(D)))$ is predictable,
  \item    $u(0,\cdot)=u_0$ a.s. in $L^2(D)$ and $u\geq \psi$,
  \item For all $t\in[0,T]$, a.s. in $\Omega$
  \begin{align}\label{230307_02}
  u(t)+ \int_0^t\rho \,ds - \int_0^t \operatorname{div} a(\cdot,u_,\nabla u) \, ds = u_0+\int_0^t G(u)\,dW(s)+\int_0^t f \, ds,
  \end{align}
  \item $-\rho \geq 0$
  and
  \begin{align*}
  \erww{\int_0^T\int_D \rho(u-\psi) \, dx\,dt} = 0.
  \end{align*}
 \end{enumerate}
\end{theorem}
The uniqueness part of Theorem \ref{bisexistence-regular} follows from
\begin{prop}\label{bisuniqueness1}
Assume that $(u_1, \rho_1), (u_2,\rho_2)$ are solutions to \eqref{I} in the sense of Definition \ref{def}, with respect  to the initial values $u_{01}$ and $u_{02}$ in $L^2(D)$. Then, we have
 \begin{align}
 \erww{\int_D\vert u_1(t)-u_2(t)\vert \, dx} \leq \erww{\int_D\vert u_{01}-u_{02}\vert \, dx} 
 \end{align}
 for all $t \in [0,T]$. In particular, for $u_{01}=u_{02}$ we have $(u_1,\rho_1)=(u_2,\rho_2)$.
\end{prop}
\begin{proof}
The proof is in every respect analogous to that of Proposition \ref{uniqueness}, where $\tilde a=a$ and $\tilde G=G$ since the constraint is satisfied by the solution. The only difference is that one needs to carefully consider the term
\begin{align}\label{230307_01}
\int_0^t\langle \rho_ 1- \rho_2, N_\delta^\prime(u_1-u_2)\rangle \, ds.
\end{align}
Note that 
$
\langle  \rho_1-\rho_2,N_\delta^\prime(u_1-u_2)\rangle = 
\langle \rho_1,N_\delta^\prime(u_1-u_2)\rangle
+
\langle \rho_2,N_\delta^\prime(u_2-u_1)\rangle
$ 
and since $-\rho_1$ is non-negative, one has for $\lambda>0$
\begin{align*}
\langle \rho_1,N_\delta^\prime(u_1-u_2)\rangle =& 
\langle \rho_1,N_\delta^\prime(u_1-u_2)^+\rangle - \langle \rho_1,N_\delta^\prime(u_1-u_2)^-\rangle
\geq
\langle \rho_1,N_\delta^\prime(u_1-u_2)^+\rangle
\\ =&
\lambda \langle \rho_1,u_1-[u_1-\frac1\lambda N_\delta^\prime(u_1-u_2)^+]\rangle.
\end{align*}
Set $w_{\delta}=u_1-\frac{1}{\lambda_{\delta}} N_\delta^\prime(u_1-u_2)^+$ and $\lambda_{\delta} = \|N_\delta^{\prime\prime}\|_\infty$. Then, $w_{\delta} \in L^p(\Omega_T,V)$. Since $N_{\delta}'(0)=\eta_{\delta}(0)=0$ and  $N_{\delta}'(x)^+ \leq |N_{\delta}'(x)| \leq |x|\lambda_{\delta}$ for any $x \in \R$, we have
\begin{align*}
w_{\delta}-\psi = 
\left\{
\begin{array}{ccc}
u_1-\psi  &  \text{ if } & u_1 \leq u_2  \\
u_1-u_2-\frac1\lambda N_\delta^\prime(u_1-u_2)^+ + u_2-\psi  &  \text{ if } & u_1 > u_2
\end{array}
\right\} \geq 0.
\end{align*}
Therefore, $w_{\delta} \in K$ and $\langle \rho_1,N_\delta^\prime(u_1-u_2)\rangle\geq 0$. Since similarly $\langle \rho_2,N_\delta^\prime(u_2-u_1)\rangle\geq 0$, \eqref{230307_01} is non-negative and it follows that $u_1=u_2$. Then $\rho_1=\rho_2$ follows from \eqref{230307_02}.
\end{proof}

\subsection{\textit{A priori} estimates with respect to $\eps$}
Let $\eps>0$, 
we recall the existence and uniqueness of $u_\eps$ solution to  the  penalized problem  \eqref{penalization} satisfying: $u_\eps: \Omega_T \to L^2(D)$ is a predictable process in $L^p(\Omega_T;V)$ belonging to $L^2(\Omega; C([0,T];L^2(D)))$, $u_\varepsilon(0,\cdot)=u_0$ in $L^2(D)$ and 
\begin{align}\label{220908_02}
u_\eps(t)-\dfrac{1}{\varepsilon}\int_0^t (u_\eps-\psi)^-\,ds-\int_0^t \operatorname{div}\,\tilde{a}(u_\eps,\nabla u_\eps) \,ds=u_0+\int_0^t \tilde{G}(u_\eps)\,dW(s)+\int_0^t f\,ds
 \end{align}
holds in $V'$ for all $t\in[0,T]$, $\mathds{P}$-a.s. in $\Omega$.

\begin{lemma}\label{lem1}
Under assumptions of Theorem \ref{bisexistence-regular}, we
 have the following uniform bounds with respect to $\eps$: 
 \begin{itemize}
  \item[$i.)$] $(u_\varepsilon)_{\varepsilon>0}$ is bounded in  $ L^p(\Omega_T;V)$ and in $L^2(\Omega; C([0,T];H))$.
  \item[$ii.)$] $(\tilde{a}(\cdot, u_\varepsilon,\nabla u_\varepsilon))_{\varepsilon>0}$ is bounded in $L^{p^\prime}(\Omega_T\times D)^d$. 
  \item[$iii.)$] $(- \operatorname{div} \tilde{a}\,(\cdot, u_\varepsilon,\nabla u_\varepsilon))_{\varepsilon>0}$ is bounded in $L^{p^\prime}(\Omega_T;V')$.
\end{itemize}
\end{lemma}
\begin{proof}
 Since $\partial_t\left( \psi-\int_0^\cdot \widetilde{G}(\psi) \, dW\right) \in  L^{p^\prime}(\Omega_T;\Vprime)$ is predictable, using \eqref{PDEforpsi} one has that
 \begin{align*}
 &u_\varepsilon(t)-\psi(t)-\int_0^t(\operatorname{div}\,\tilde{a}(u_\eps,\nabla u_\eps)+\dfrac{1}{\varepsilon}(u_\varepsilon-\psi)^-) \, ds
 \\ =&
 u_0-\psi(0) +\int_0^t [f- \partial_t\left( \psi-\int_0^\cdot \widetilde{G}(\psi) \, dW\right)] \, ds+\int_0^t (\widetilde{G}(u_\varepsilon)-\widetilde{G}(\psi)) \, dW. 
 \end{align*}
 Then It\^o's stochastic energy yields
 \begin{align*}
 &\|u_\varepsilon-\psi \|_{H}^2(t)+2\int_0^t \langle -\operatorname{div}\, \tilde{a}(u_\eps,\nabla u_\eps),u_\varepsilon-\psi \rangle \, ds 
 -2\int_0^t \int_D\dfrac{1}{\varepsilon}(u_\varepsilon-\psi)^-(u_\varepsilon-\psi ) \, dx \, ds 
 \\ =&
 \|u_0-\psi (0)\|^2_{H} + 2\int_0^t \langle  f- \partial_t\left( \psi -\int_0^\cdot \widetilde{G}(\psi ) \, dW\right)  , u_\varepsilon-\psi  \rangle \, ds
 \\&
 +2\int_0^t  \left(u_\varepsilon-\psi ,\widetilde{G}(u_\varepsilon)-\widetilde{G}(\psi ) \, dW\right)_H + \int_0^t \|\widetilde{G}(u_\varepsilon)-\widetilde{G}(\psi )\|_{HS(H)}^2 \, ds. 
 \end{align*}
 Note that $-\dint_0^t \dint_D\dfrac{1}{\varepsilon}(u_\varepsilon-\psi)^-(u_\varepsilon-\psi ) \, dx \, ds =\dfrac{1}{\varepsilon} \int_0^t \Vert (u_\varepsilon-\psi)^-\Vert_H^2 \, ds$ and applying $H_{2,3}$ and Young's inequality yield the existence of a constant $C_{\bar\alpha} >0$ such that 
 \begin{align*}
 \langle -\operatorname{div}\,\tilde{a}(\cdot,u_\eps,\nabla u_\eps),u_\varepsilon-\psi \rangle &=\int_D \tilde{a}(\cdot,u_\eps,\nabla u_\eps)\cdot \nabla(u_\eps-\psi ) \, dx
 \\&\geq \dfrac{\bar\alpha}{2} \Vert u_\varepsilon\Vert_V^p - C_{\bar\alpha}(  \Vert \psi\Vert_{L^p(D)}^p+1) -
 \int_D \tilde{a}(\cdot,u_\eps,\nabla u_\eps)\cdot \nabla \psi \, dx.
 \end{align*}
 By using $H_{2,3}$, Poincar\'e's inequality and Young inequality one gets, for any $\delta >0$, a constant $C_{\delta}>0$ such that
 \begin{align*}
 \vert \int_D \tilde{a}(\cdot,u_\eps,\nabla u_\eps)\cdot \nabla \psi dx\vert  \leq \delta\Vert u_\eps\Vert_V^p +C(\Vert \bar k\Vert_{L^{p^\prime}(D)}^{p^\prime}+1)+C_\delta \Vert \psi \Vert_V^p.
 \end{align*}
 Therefore, choosing $\delta = \dfrac{\bar\alpha}{4}$ yields the existence of a function $\tilde{g}=\tilde{g}(\psi) \in L^1(\Omega_T)$ only depending on $\psi$ and fixed constants such that
 \begin{align*}
 \langle -\operatorname{div}\, \tilde{a}(\cdot,u_\eps,\nabla u_\eps),u_\varepsilon-\psi \rangle
 \geq \frac{\bar \alpha}{4} \Vert u_\varepsilon\Vert_V^p   - \tilde{g}.
 \end{align*}
 
 Thus, for any positive $\gamma$, Young's inequality yields the existence of a positive constant $C_\gamma$ such that 
 \begin{align*}
 &\|u_\varepsilon-\psi \|_{H}^2(t)+\int_0^t \frac{\bar\alpha}{2} \Vert u_\varepsilon\Vert_V^p \, ds+\dfrac{2}{\varepsilon}\int_0^t \!\!\!\Vert (u_\varepsilon-\psi)^-\Vert_H^2 \, ds
 \\ \leq & \|u_0-\psi (0)\|^2_{H}+
 \|\tilde{g}\|_{L^1(0,T)} + C_\gamma\int_0^t \left\Vert f- \partial_s\left( \psi -\int_0^\cdot \widetilde{G}(\psi) \, dW\right)\right \Vert_{V'}^{p'} \,ds
 \\&+ \gamma \int_0^t\|u_\varepsilon-\psi  \|^p_V \, ds+2\int_0^t  \left(u_\varepsilon-\psi ,\widetilde{G}(u_\varepsilon)-\widetilde{G}(\psi ) \, dW\right)_H+\int_0^t \|\widetilde{G}(u_\varepsilon)-\widetilde{G}(\psi )\|_{HS(H)}^2 \, ds. 
 \end{align*}
By choosing $\gamma= \frac{\bar\alpha}{2^{p+1}}$ and applying $H_{3,2}$ we may conclude that
 \begin{align*}
&\|u_\varepsilon-\psi \|_{H}^2(t)+\int_0^t \frac{\bar\alpha}{2} \Vert u_\varepsilon\Vert_V^p \, ds+\dfrac{2}{\varepsilon}\int_0^t \!\!\!\Vert (u_\varepsilon-\psi)^-\Vert_H^2 \, ds
\\ \leq & 
 C_{\sigma} \int_0^t \Vert u_\varepsilon - \psi \Vert_H^2 \, ds+ \int_0^t \frac{\bar\alpha}{4} \|u_\varepsilon\|^p_V \, ds+ C(\psi, u_0, \tilde{g},f)+2\int_0^t  \left(u_\varepsilon-\psi ,\widetilde{G}(u_\varepsilon)-\widetilde{G}(\psi ) \, dW\right)_H
 \end{align*}
 where $C(\psi, u_0, \tilde{g},f)\in L^1(\Omega)$ is a non-negative function depending on $u_0, \psi$, $\tilde{g}$ and $f$. Then, the first part of $i.)$ is proved by taking the expectation and applying Gronwall's lemma, the second part of $i.)$ follows from classical arguments based on BDG inequality as developed in Section \ref{estimates_delta}. The items $ii.)$ and $iii.)$ follow by adding the information H$_{2,3}$. 
\end{proof}
\begin{lemma}\label{monotone_sequence}
Still under assumptions of  Theorem \ref{bisexistence-regular}, the sequence $(u_\varepsilon)_{\varepsilon>0}$ is a non-decreasing sequence when $\varepsilon$ goes to $0$, 
more precisely, for all $0<\delta<\varepsilon$ and all $t \in [0,T]$ we have $\ueps(t) \leq u_{\delta}(t)$ $d\mathds{P}\otimes dx$-a.e. in $\Omega \times D$ and thus $\ueps \leq u_{\delta}$ $d\mathds{P}\otimes dt\otimes dx$-a.e. in $\Omega\times (0,T)\times D$.
\end{lemma}
\begin{proof}
Set $0 < \delta< \varepsilon$. Denote by $(H_\theta)_{\theta>0}$ the $C^2$ approximation of the positive-part function such that $H^{\prime\prime}_\theta$ is the mollifier with support $[0,2\theta]$, $H_\theta(0)=H^{\prime}_\theta(0)=0$ and $H_{\theta}(r)\rightarrow r^+$ for $\theta\rightarrow 0^+$. Thus, by It\^o's formula with $\Phi_{\theta}(u):= \int_D H_{\theta}(u) \, dx$, for any $t\in [0,T]$ we have
\begin{align}\label{221107_01}
\begin{aligned}
&\erww{\int_D H_\theta(u_\varepsilon-u_\delta)(t) \, dx} + \erww{\int_{Q_t} H^{\prime\prime}_{\theta}(u_\varepsilon-u_\delta)[\tilde a(u_\varepsilon,\nabla u_\varepsilon) - \tilde a(u_\delta,\nabla u_\delta)]\nabla(u_\varepsilon-u_\delta) \, dx \, ds} 
\\ \leq& \frac1\varepsilon \erww{\int_{Q_t}[ (u_\varepsilon-\psi)^--(u_\delta-\psi)^-]H^\prime_\theta(u_\varepsilon-u_\delta) \, dx \, ds} + (\frac1\varepsilon-\frac1\delta)\erww{\int_{Q_t} (u_\delta-\psi)^- H^\prime_\theta(u_\varepsilon-u_\delta) \, dx \, ds} 
\\ &+ C_{\sigma} \erww{\int_{Q_t} H_{\theta}''(u_{\varepsilon} - u_{\delta}) |u_{\varepsilon} - u_{\delta}|^2 \, dx \, ds}.
\end{aligned}
\end{align}
Note that the middle-line in \eqref{221107_01} is non-positive and that $H^{\prime\prime}_{\theta}(u_\varepsilon-u_\delta)1_{\{u_\varepsilon\neq u_\delta\}}$ goes a.e. to $0$. 
Moreover, to be able to apply Lebesgue's theorem, note that there exists a constant $C>0$ such that $2\theta\Vert H_{\theta}''\Vert_{\infty} \leq C$ and therefore we get
\begin{align*}
0 \leq H^{\prime\prime}_{\theta}(u_\varepsilon-u_\delta)|u_\varepsilon-u_\delta|^2 \leq C |u_\varepsilon-u_\delta|\in L^1(\Omega\times Q_T),
\end{align*}
\begin{align*}
0 &\leq  H^{\prime\prime}_{\theta}(u_\varepsilon-u_\delta)\left(\tilde a(u_\varepsilon,\nabla u_\varepsilon) - \tilde a(u_\varepsilon,\nabla u_\delta)\right)\cdot\nabla(u_\varepsilon-u_\delta),
\end{align*}
and, by using H$_{2,4}$,
\begin{align*}
&\left|H^{\prime\prime}_{\theta}(u_\varepsilon-u_\delta)\left(\tilde{a}(u_   
\varepsilon,\nabla u_\delta) - \tilde a(u_\delta,\nabla u_\delta)\right)
\cdot\nabla(u_\varepsilon-u_\delta)\right|\\  
\leq& H^{\prime\prime}_{\theta}(u_\varepsilon-u_\delta) \left(C_3^{a}|\nabla u_\delta|^{p-1} + l\right) |u_\varepsilon- u_\delta| \, |\nabla(u_\varepsilon-u_\delta)|
\\ 
\leq& C \left(C_3^{a}|\nabla u_\delta|^{p-1} + l\right)  \, |\nabla(u_\varepsilon-u_\delta)| \in L^1(\Omega\times Q_T).
\end{align*}
Thus, by Fatou's lemma, passing to the limit with $\theta\rightarrow 0^{+}$ for we have 
\begin{align}\label{230308_01}
\erww{\int_{D}(u_\varepsilon-u_\delta)^+(t)\,dx}=0
\end{align} 
for any $t \in [0,T]$, hence $\ueps(t)\leq u_{\delta}(t)$ $d\mathds{P}\otimes dx$-a.e. in $\Omega \times D$. The rest of the assertion follows by integrating \eqref{230308_01} over $(0,T)$.   
\end{proof}

\begin{definition}\label{a.e.-convergence}
From Lemma \ref{monotone_sequence} it follows that we may define a $d\mathds{P}\otimes dt\otimes dx$-measurable function $u:\Omega\times (0,T)\times D\rightarrow\overline{\mathbb{R}}$ by setting
\begin{align*}
u(\omega,t,x):=\sup_{\eps>0}\ueps(\omega,t,x)
\end{align*}
and, for any $t$, a $d\mathds{P}\otimes  dx$-measurable function $u_t:\Omega\times D\rightarrow\overline{\mathbb{R}}$ by setting
\begin{align*}
u_t(\omega,x):=\sup_{\eps>0}\ueps(t)(\omega,x).
\end{align*}
Note that $u(\omega,t,x)=u_t(\omega,x)$ for almost every $(\omega,t,x) \in \Omega \times (0,T) \times D$ and from Lemma \ref{lem1} and Fatou's lemma, it follows that $u$ and $u_t$ are a.e. finite on $\Omega \times (0,T) \times D$ and on $\Omega \times D$ for any $t \in [0,T]$, respectively.
\end{definition}
\begin{lemma}\label{230308_lem01}
Under assumptions of Theorem \ref{bisexistence-regular}, we have the following convergence results for $\eps\rightarrow 0^+$:
\begin{itemize}
\item[$i.)$] $\ueps\rightharpoonup u$ in $L^p(\Omega_T;V)$, in $L^2(\Omega;L^{\beta}(0,T;L^2(D))$ for any $1\leq \beta<\infty$ and $*$-weakly in \linebreak $L^2(\Omega;L^{\infty}_w(0,T;L^2(D)))$ where $w$ denotes the weak-$\ast$ measurability in $L^{\infty}(0,T;L^2(D))$. In particular, $u$ is predictable with values in $V$ and in $H$.
\item[$ii.)$] $\ueps(t)\rightharpoonup u_t$ in $L^2(\Omega;L^2(D))$ for all $t\in [0,T]$. 
\item[$iii.)$] For any $t\in [0,T]$, $\ueps^+(t)\rightarrow u^+_t$, $\ueps^-(t)\rightarrow u^-_t$ in $L^2(\Omega;L^2(D))$.
\item[$iv.)$] $\ueps^+\rightarrow u^+$, $\ueps^-\rightarrow u^-$ and  $\ueps\rightarrow u$ in $L^2(\Omega;L^{\beta}(0,T;L^2(D))$ for any $1\leq \beta<\infty$.
\end{itemize}
\end{lemma}
\begin{proof}
The convergence results of $i.)$ are a direct consequence of Lemma \ref{lem1}, $i.)$. For any $t\in [0,T]$ $\ueps(t)\rightarrow u_t$ for $\eps\rightarrow 0^+$ $d\mathds{P}\otimes dx$-a.s. in $\Omega\times D$, hence together with the bound in $L^2(\Omega;C([0,T];L^2(D)))$ from Lemma \ref{lem1} we obtain $\ueps(t)\rightharpoonup u_t$ in $L^2(\Omega;L^2(D))$ for all $t\in [0,T]$. From the monotone convergence of $\ueps$ it follows that $u_t^+=\sup_{\eps>0}\ueps^+(t)$ $d\mathds{P}\otimes dx$-a.s. in $\Omega\times D$ for all $t\in [0,T]$ and $u^+=\sup_{\eps>0} u_{\eps}^+$ $d\mathds{P}\otimes dt\otimes dx$-a.s. in $\Omega\times (0,T)\times D$. Consequently,
\begin{align*}
|\ueps^+(t)|=\ueps^+(t)\leq u_t^{+}
\qquad \text{and} \qquad 
|\ueps^-(t)|= \ueps^-(t) \leq u_1^-(t)
\end{align*}
for all $0<\eps<1$ $d\mathds{P}\otimes dx$-a.s. in $\Omega\times D$ for all $t\in [0,T]$ and
\begin{align*}
|\ueps^+|=\ueps^+\leq u^{+}
\qquad \text{and} \qquad 
|\ueps^-|= \ueps^- \leq u_1^-
\end{align*}
$d\mathds{P}\otimes dt\otimes dx$-a.s. in $\Omega\times (0,T)\times D$.
For any $t\in [0,T]$, by using Fatou's lemma, one has $u_t^+\in L^2(\Omega\times D)$  and since $u_1(t) \in L^2(\Omega\times D)$, $iii.)$ follows from Lebesgue's dominated convergence theorem. On the other hand, again Fatou's lemma and Lemma \ref{lem1} $i.)$ ensure $u^+\in L^2(\Omega;L^{\beta}(0,T;L^2(D))$ for any $1\leq \beta<\infty$. Since the sequences are bounded in $L^2(\Omega;L^{\infty}(0,T;L^2(D))$, Lebesgue's dominated convergence theorem ensures first that $\ueps^+\rightarrow u^+$, $\ueps^-\rightarrow u^-$ in $L^2(\Omega;L^{2}(0,T;L^2(D))$, then in $L^2(\Omega;L^{\beta}(0,T;L^2(D))$ for any $1\leq \beta<\infty$ by some interpolation inequalities.
\end{proof}

\begin{lemma}\label{lem2} 
Under assumptions of Theorem \ref{bisexistence-regular},
there exists a constant $K_1\geq 0$ depending on the norm of $h^-$ in $L^2(\Omega\times Q_T)$ and not depending on $\varepsilon>0$ such that
\begin{align}\label{220908_07}
&\erww{\int_0^T\left\Vert \frac{1}{\varepsilon}(\ueps-\psi)^{-}\right\Vert_{L^2(D)}^2\,dt}\leq K_1,
\end{align}
\begin{align}\label{220908_08}
\sup_{t \in [0,T]} \erww{\Vert(\ueps-\psi)^-(t)\Vert_{L^2(D)}^2}\leq K_1\varepsilon,
\end{align}
\begin{align}\label{220908_09}
\erww{ \ \int\limits_{Q_T} \left|(\tilde  a(u_\varepsilon,\nabla u_\varepsilon)-\tilde{a}(\psi,\nabla \psi)) \cdot \nabla (u_\varepsilon-\psi)^-\right|\,dx\,ds} \leq K_1\frac{\varepsilon}{2}. 
\end{align}

\end{lemma}
\begin{proof}
 Let $\delta>0$ and consider the following approximation  from \cite[p. 152]{Pardoux}:
\begin{align*}
 F_\delta(r)= \left\{ \begin{array}{l}
 r^2-\dfrac{\delta^2}{6} \quad\quad\quad  \text{if} \hspace*{0.2cm} r\leq -\delta, \\
 -\dfrac{r^4}{2\delta^2}-\dfrac{4r^3}{3\delta}\quad \text{if} \hspace*{0.2cm}  -\delta\leq r \leq 0, \\[2ex]
 0 \quad\quad\quad\quad \quad\quad \text{if} \hspace*{0.2cm}  r\geq 0.
 \end{array}
 \right.
 \end{align*} 
 Note that $ (-\dfrac{1}{2}F_\delta^\prime)_\delta $  approximates the negative part. Moreover, $F_\delta(\cdot)\in C^2(\mathbb{R})$ satisfies:
 \begin{align*}
  \left\{ \begin{array}{l}
   \vert F_\delta(r)\vert \leq r^2 \ \text{for all} \ r\in\mathbb{R},\\
   \vert F_\delta^\prime(r)\vert \leq 2r \ \text{and} \ F_\delta^\prime(r) \leq 0 \ \text{for all} \ r \in \mathbb{R}, \\
   \vert F_\delta^{\prime\prime}(r)\vert \leq \frac{8}{3} \ \text{and} \ F_\delta^{\prime\prime}(r) \geq 0 \ \text{for all} \ r \in \mathbb{R}.
  \end{array}
  \right.
 \end{align*} 
Taking the difference of \eqref{PDEforpsi} and \eqref{220908_02} we obtain 
\begin{align}\label{220908_03}
\begin{aligned}
&\ueps(t)-\psi(t)-\frac{1}{\eps}\int_0^t(\ueps-\psi)^{-}\,ds-\int_0^t\Div(\tilde{a}(\ueps,\nabla \ueps)-\tilde{a}(\psi,\nabla \psi))\,ds\\
=&\int_0^t(\tilde{G}(\ueps)-\tilde{G}(\psi))\,dW(s)+u_0-\psi(0)+\int_0^t h\,ds
\end{aligned}
\end{align} 
for all $t\in [0,T]$, $\mathds{P}$-a.s in $\Omega$. Applying  It\^{o}'s formula (see \cite[Thm. 4.2 p. 65]{Pardoux} and also \cite[Lemma 4]{Rascanu}) with $\Psi_{\delta}(v)=\int_DF_\delta(v(x)) \, dx$ for $v\in L^2(D)$ in \eqref{220908_03} we obtain
\begin{align}\label{250701}
\begin{aligned}
&\int_D F_{\delta}(\ueps-\psi)(t)\,dx -\int_D F_{\delta}(u_0-\psi(0))\,dx -\frac{1}{\eps} \int_0^t\int_D(\ueps-\psi)^{-}F'_{\delta}(\ueps-\psi)\,dx\,ds\\
&-\int_0^t \left\langle\Div(\tilde{a}(\ueps,\nabla \ueps)-\tilde{a}(\psi,\nabla \psi)),F'_{\delta}(\ueps-\psi)\right\rangle\,ds\\
=&\int_0^t\langle h,F'_{\delta}(\ueps-\psi)\rangle\,ds + \int_0^t\left(F'_{\delta}(\ueps-\psi),\tilde{G}(\ueps)-\tilde{G}(\psi)\,dW(s)\right) \\
&+\frac{1}{2}\operatorname{Tr}\int_0^t \Psi_{\delta}''(\ueps-\psi)(\tilde{G}(\ueps)-\tilde{G}(\psi))Q(\tilde{G}(\ueps)-\tilde{G}(\psi))^{\ast}\,ds
\end{aligned}
\end{align}
for all $t\in [0,T]$, $\mathds{P}$-a.s in $\Omega$. Since $u_0\geq \psi(0)$, $F'_{\delta}(u_0-\psi(0))=0$. Recalling that $h=h^+-h^-$ with $h^+$ non-negative by H$_5$, using the assumption that $h^-$ is in $L^2(\Omega\times Q_T)$ and that $F'_{\delta}(\ueps-\psi)\leq 0$ we get
\begin{align*}
&\int_0^t\langle h,F'_{\delta}(\ueps-\psi)\rangle\,ds=\int_0^t \langle h^+,F'_{\delta}(\ueps-\psi)\rangle\,ds-\int_0^t\int_D h^- F'_{\delta}(\ueps-\psi)\,dx\,ds\\
\leq& -\int_0^t\int_D h^- F'_{\delta}(\ueps-\psi)\,dx\,ds.
\end{align*}
Since $F'_{\delta}(\ueps-\psi)=0$ on the set $\{\ueps-\psi\geq 0\}$ and $G(\max(\ueps,\psi))-G(\psi)=0$ in $\{\ueps-\psi\leq 0\}$, the last and the second last term in \eqref{250701} are equal to $0$, namely
\begin{align*}
&\int_0^t \left(F'_{\delta}(\ueps-\psi),\tilde{G}(\ueps)-\tilde{G}(\psi)\,dW(s)\right)=0 \\
\end{align*}
and
\begin{align*}
&\operatorname{Tr}\int_0^t \Psi_{\delta}''(\ueps-\psi)(\tilde{G}(\ueps)-\tilde{G}(\psi))Q(\tilde{G}(\ueps)-\tilde{G}(\psi))^{\ast}\,ds = 0
\end{align*}
for all $t\in [0,T]$, $\mathds{P}$-a.s in $\Omega$.
Thus,
\begin{align}\label{220908_04}
\begin{aligned}
&\int_D F_{\delta}(\ueps-\psi)(t)\,dx -\frac{1}{\eps} \int_0^t\int_D(\ueps-\psi)^{-}F'_{\delta}(\ueps-\psi)\,dx\,ds
\\&
-\int_0^t \left\langle\Div(\tilde{a}(\ueps,\nabla \ueps)-\tilde{a}(\psi,\nabla \psi)),F'_{\delta}(\ueps-\psi)\right\rangle\,ds\\
\leq& -\int_0^t\int_D h^- F'_{\delta}(\ueps-\psi)\,dx\,ds
\end{aligned}
\end{align}
for all $t\in [0,T]$, $\mathds{P}$-a.s in $\Omega$. Repeating the arguments of \cite[Sec. 3.1.2]{YV}, for all $t\in [0,T]$ and $\mathds{P}$-a.s in $\Omega$,  $F^\prime_\delta(u_\varepsilon-\psi)$ converges to $-2(u_\varepsilon-\psi)^-$ in $V$. Hence, we are now in position to conclude that, for $\delta\rightarrow 0^+$, 
\begin{itemize}[label=$\bullet$]
\item[$\bullet$] $\Psi_\delta (u_\varepsilon(t)-\psi(t)) \rightarrow \Vert(u_\varepsilon-\psi)^-(t)\Vert_{L^2(D)}^2$ for all $t\in [0,T]$, $\mathds{P}$-a.s in $\Omega$,  
\item[$\bullet$] $\left(-(u_\varepsilon-\psi)^-,F_\delta^\prime(u_\varepsilon-\psi)\right)_2\rightarrow 2\Vert (u_\varepsilon-\psi)^-\Vert_{L^{2}(D)}^{2}$ for all $t\in [0,T]$, $\mathds{P}$-a.s in $\Omega$, 
\item[$\bullet$] $\left(-h^-,F_\delta^\prime(u_\varepsilon-\psi)\right)_2 \rightarrow 2\left( h^-, (u_\varepsilon-\psi)^-\right)_2$ $d\mathds{P}\otimes dt$-a.e. in $\Omega_T$,
\item[$\bullet$]$d\mathds{P}\otimes dt$-a.e. in $\Omega_T$, 
\[\left\langle -\Div (\tilde{a}(u_\eps,\nabla u_\eps)-\tilde{a}(\psi,\nabla \psi)),F_\delta^\prime(u_\eps-\psi)\right\rangle \to \left\langle -\Div(\tilde{a}(u_\eps,\nabla u_\eps)-\tilde{a}(\psi,\nabla \psi)),-2(u_\eps-\psi)^-\right\rangle.\]
 \end{itemize}
Taking expectation in \eqref{220908_04} and using Lebesgue's dominated convergence theorem we may pass to the limit with $\delta\rightarrow 0^+$ and we find
\begin{align}\label{Pen}
\begin{aligned}
&\erww{\Vert(u_\eps-\psi)^-(t)\Vert_{L^2(D)}^2}+\dfrac{2}{\eps} \erww{\int_0^t\Vert (u_\varepsilon-\psi)^-\Vert_{L^{2}(D)}^{2}\,ds}\\
&-2\erww{\int_0^t\int_D (\tilde{a}(u_\eps,\nabla u_\eps)-\tilde{a}(\psi,\nabla \psi))\nabla(u_\eps-\psi)^-\,dx\,ds}\\
\leq& 2\erww{\int_0^t\int_D h^- (u_\varepsilon-\psi)^-\,dx\,ds}
\end{aligned}
\end{align}
for all $t\in [0,T]$. We have
\begin{align*}
&-\left(\tilde{a}(u_\eps,\nabla u_\eps)-\tilde{a}(\psi,\nabla \psi)\right)\nabla(u_\eps-\psi)^- 
=\left(\tilde{a}(u_\eps,\nabla u_\eps)-\tilde{a}(\psi,\nabla \psi)\right)\nabla (u_\eps-\psi) \mathds{1}_{\{u_\eps < \psi\}}\\
=&
\left(a(\psi,\nabla u_\eps)-a(\psi,\nabla \psi)\right)\nabla (u_\eps-\psi) \mathds{1}_{\{u_\eps < \psi\}} \geq 0.
\end{align*}
Multiplying \eqref{Pen} by $\frac{1}{2\eps}$, one gets
\begin{align}\label{220908_05}
\begin{aligned}
&\frac{1}{2\eps}\erww{\Vert(u_\eps-\psi)^-(t)\Vert_{L^2(D)}^2}+\frac{1}{\eps^2}\erww{\int_0^t \Vert (u_\eps-\psi)^-\Vert_{L^{2}(D)}^{2}\,ds}\\
&+\frac{1}{\eps}\erww{\int_0^t\int_D \left|\left(a(\psi,\nabla u_\eps)-a(\psi,\nabla \psi)\right)\nabla(u_\eps-\psi)^-\right|\,dx\,ds}\\
\leq& \erww{\int_0^t\int_D h^-\frac{1}{\eps}(u_\eps-\psi)^- \,dx\,ds}
\end{aligned}
\end{align}
for all $t\in [0,T]$. Applying Young's inequality on the right-hand side of \eqref{220908_05} we obtain
\begin{align}\label{220908_06}
\begin{aligned}
&\frac{1}{2\eps}\erww{\Vert(u_\eps-\psi)^-(t)\Vert_{L^2(D)}^2}+\frac{1}{2\eps^2}\erww{\int_0^t \Vert (u_\eps-\psi)^-\Vert_{L^{2}(D)}^{2}\,ds}\\
&+\frac{1}{\eps}\erww{\int_0^t\int_D \left|\left(a(\psi,\nabla u_\eps)-a(\psi,\nabla \psi)\right)\nabla(u_\eps-\psi)^-\right|\,dx\,ds}\\
\leq&\frac{1}{2}\erww{\int_0^t\Vert h^-\Vert^2_{L^2(D)}\,ds}
\end{aligned}
\end{align}
for all $t\in [0,T]$. From \eqref{220908_06} now it follows that \eqref{220908_07}, \eqref{220908_08} and \eqref{220908_09} hold with 
\[K_1=\erww{\int_0^T\Vert h^-\Vert^2_{L^2(D)}\,ds}.\] 
\end{proof}
From the previous Lemma one gets that $(\dfrac{1}{\eps}(u_\eps - \psi)^-)_{\eps >0}$ is bounded in $L^{2}(\Omega_T; L^{2}(D))$, we need now to add the boundedness in $L^{p'}(\Omega_T; L^{p'}(D))$, especially if $p<2$.
\begin{lemma}\label{lem2_2}
With the assumptions 
of Theorem \ref{bisexistence-regular},
$(\dfrac{1}{\eps}(u_\eps - \psi)^-)_{\eps >0}$ is bounded in $L^{\alpha}(\Omega_T; L^{\alpha}(D))$. 
\end{lemma}
\begin{proof}
 For $p\geq2$ there is nothing to prove. Now let us assume $p<2$. Since $\tilde{a}(u_\eps, \nabla u_\eps)= \tilde{a}(\psi, \nabla u_\eps)$ and $G(\max(\ueps,\psi))=G(\psi)$ on the set $\{u_\eps  \leq \psi \}$, applying It\^{o}'s formula to the process $u_\eps - \psi$ yields
 \begin{align*}
 -\dfrac{1}{\eps}\int_0^t\int_D [(u_\eps - \psi)^-] F_M^\prime(u_\eps-\psi) \, dx \, ds
 \leq -\int_0^t\int_D h^- F_M^\prime(u_\eps-\psi) \, dx \, ds,
 \end{align*}
 where $F_M\in C^2(\mathbb{R})$ is the non-negative even convex function such that $F_M(x)=F_M^{\prime}(x)=0$ for any $x\geq 0$, equal to $F_M(x)=|x|^{p^\prime}$ on $[-M, 0]$ and whose second-order derivative on $(-\infty,-M]$ is given by $F_M^{\prime \prime}(x)=p^\prime(p^\prime-1) M^{p^\prime-2}$. After a few routine steps, one gets that $\forall r \in \mathbb{R}, \quad 0 \leq r^2 F_M^{\prime \prime}(r) \leq p^\prime(p^\prime-1) F_M(r)$,  and $0 \leq r F_M^{\prime}(r) \leq p^\prime F_M(r)$. Then, for any $\lambda >0$ there exists $C_{\lambda} >0$ such that
 \begin{align*}
 -\dfrac{1}{\eps^{p^\prime}}\int_0^t\int_D [(u_\eps - \psi)^-] F_M^\prime(u_\eps-\psi) \, dx \, ds
 \leq& 
 -\int_0^t\int_D \dfrac{1}{\eps^{p^\prime-1}}h^- F_M^\prime(u_\eps-\psi) \, dx \, ds
 \\ \leq & 
 \lambda \dfrac{1}{\eps^{p^\prime}}\int_0^t\int_D |F_M^\prime(u_\eps-\psi)|^p \, dx \, ds
 + C_\lambda \int_0^t\int_D |h^-|^{p^\prime} \, dx \, ds,
 \end{align*}
 \textit{i.e.} 
 \begin{align*}
 C_\lambda \int_0^t\int_D |h^-|^{p^\prime} \, dx \, ds \geq &-\dfrac{1}{\eps^{p^\prime}}\int_0^t\int_{D \cap \{-M<u_\eps-\psi\leq 0\}} [(u_\eps - \psi)^-] F_M^\prime(u_\eps-\psi) + \lambda |F_M^\prime(u_\eps-\psi)|^p \, dx \, ds
 \\&
 -\dfrac{1}{\eps^{p^\prime}}\int_0^t\int_{D \cap \{u_\eps-\psi\leq -M\}} [(u_\eps - \psi)^-] F_M^\prime(u_\eps-\psi) + \lambda |F_M^\prime(u_\eps-\psi)|^p \, dx \, ds
 \\ =&
 \dfrac{1}{\eps^{p^\prime}}\int_0^t\int_{D \cap \{-M<u_\eps-\psi\leq 0\}} (p^\prime -  (p^\prime)^p\lambda)|(u_\eps - \psi)^-|^{p^\prime} \, dx \, ds
 \\&
 +\dfrac{1}{\eps^{p^\prime}}\int_0^t\int_{D \cap \{u_\eps-\psi\leq -M\}} [(u_\eps - \psi)^-] |F_M^\prime(u_\eps-\psi)| - \lambda |F_M^\prime(u_\eps-\psi)|^p \, dx \, ds.
 \end{align*}
 Set $\lambda>0$ such that  $p^\prime -  (p^\prime)^p\lambda=1$. Since $F'_M(x)=p^\prime M^{p^\prime-2}[(p^\prime-1)x+M(p^\prime-2)]$ on $(-\infty,-M]$ and $p<2$, for $M$ big enough, $\Big([(u_\eps - \psi)^-] |F_M^\prime(u_\eps-\psi)| - \lambda |F_M^\prime(u_\eps-\psi)|^p\Big)1_{\{u_\eps-\psi\leq -M\}}\geq 0$ and 
 \begin{align*}
 \int_0^t\int_{D \cap \{-M<u_\eps-\psi\leq 0\}} |\dfrac{1}{\eps}(u_\eps - \psi)^-|^{p^\prime} \, dx \, ds \leq C \int_0^t\int_D |h^-|^{p^\prime} \, dx \, ds.
 \end{align*}
 Passing to the limit over $M$ yields the result.
\end{proof}
Note finally that since $\partial_t \left( u_\epsilon - \int_0^\cdot \tilde G(u_\eps) dW\right) = \dfrac{1}{\varepsilon} (u_\eps-\psi)^-+ \operatorname{div}\,\tilde{a}(u_\eps,\nabla u_\eps) + f$, Lemmas \ref{lem1}, \ref{lem2} and \ref{lem2_2} yield the following lemma. 
\begin{lemma}\label{lem2_3}
With assumptions 
of Theorem \ref{bisexistence-regular},
$\left(\partial_t \left( u_\eps - \int_0^\cdot \tilde G(u_\eps) dW\right)\right)_{\eps >0}$ is bounded in $L^{p'}(\Omega_T; V')$. 
\end{lemma}

\subsection{Passage to the limit with $\eps \to 0^+$}
From Lemmas \ref{lem2} and \ref{lem2_2} it follows that there exists a predictable $\rho \in L^\alpha(\Omega_T; L^\alpha(D))$ where $\alpha=\max(2,p')$ such that $-\rho \geq0$ and, passing to a not relabeled subsequence if necessary, for $\eps\rightarrow 0^+$,
\begin{align}\label{220909_03}
-\frac{1}{\eps}(\ueps-\psi)^-\rightharpoonup \rho \ \text{in} \ L^\alpha(\Omega_T; L^\alpha(D)).
\end{align}

\begin{lemma}\label{Lem penalization}
We have $u\geq \psi$, \textit{i.e.} $u \in L^p(\Omega_T; V)$ and $u \geq \psi$ a.e. in $\Omega_T \times D$. Moreover, we have
\begin{align*}
\erww{\int_0^T \int_D \rho(u-\psi)\,dx\,dt}=0.
\end{align*}
\end{lemma}
\begin{proof}
From Lemma \ref{230308_lem01} it follows that $u\in L^p(\Omega_T; V)$ and according to Definition \ref{a.e.-convergence} we have $(\ueps - \psi)^- \to (u-\psi)^-$ a.e. in $\Omega \times (0,T)\times D$. On the other hand, from Lemma \ref{lem2} it follows that $(\ueps - \psi)^- \to 0$ in $L^2(\Omega_T;L^2(D))$ for $\varepsilon \to 0^+$ and therefore $u \in K$. 
Now, using \eqref{220908_07}, \eqref{220909_03} and the convergence of $\ueps$ in $L^2(\Omega_T;H)$,  we get
\begin{align*}
&\erww{\int_0^T \int_D \rho(u-\psi)\,dx\,dt}= \lim\limits_{\varepsilon \to 0^+} \erww{\int_0^T \int_D -\frac{1}{\varepsilon}(\ueps-\psi)^- (\ueps - \psi)\,dx\,dt} \\
= &\lim\limits_{\varepsilon \to 0^+} \erww{\int_0^T \varepsilon \Vert \frac{1}{\varepsilon}(\ueps-\psi)^- \Vert_{L^2(D)}^2\,dx\,dt} 
\leq  \lim\limits_{\varepsilon \to 0^+} \varepsilon K_1 =0.
\end{align*}
\end{proof}

Note that $u_\eps = u_\eps - \int_0^\cdot \tilde G(u_\eps) dW + \int_0^\cdot \tilde G(u_\eps) dW$, where $\int_0^\cdot \tilde G(u_\eps) dW$ converges strongly to $\int_0^\cdot G(u) dW $ in $L^2(\Omega;C([0,T];H))$ by similar reasoning as in the proof of Lemma \ref{250416_02}. Since $\ueps$ converges strongly to $u$ in $L^2(\Omega_T;H)$, Lemma \ref{lem2_3} yields that $u_\eps - \int_0^\cdot \tilde G(u_\eps) dW$ converges weakly to $u - \int_0^\cdot G(u) dW$ in $L^{\min(2,p')}(\Omega;\mathcal{W})$ for a subsequence, where
\begin{align*}
\mathcal{W}:= \{ v \in L^2(0,T;L^2(D)) \ | \ \partial_t v \in L^{p'}(0,T; V')\}.
\end{align*}
Since $\mathcal{W}$ is compactly imbedded in $C([0,T]; V')$ by \cite{Simon}, we may conclude that $\ueps$ converges weakly  to $u$ in $L^{\min(2,p')}(\Omega;C([0,T];V'))$. Especially, we have $u \in L^{\min(2,p')}(\Omega;C([0,T];V'))$ and $\ueps(t) \rightharpoonup u(t)$ in $L^{\min(2,p')}(\Omega;V')$ for any $t \in [0,T]$. As a by product, $u_t$ defined in Definition \ref{a.e.-convergence} is equal to $u(t)$ for any $t \in [0,T]$ and it is possible to summarize the convergence information in the following remark.
\begin{rem}\label{rem1}
Up to a subsequence, Lemma \ref{lem1}, Lemma \ref{230308_lem01}, Lemma \ref{Lem penalization}, H$_{3,2}$ and BDG inequality yield the existence of $A \in L^{p'}(\Omega_T; L^{p'}(D))^d$ predictable and the validity of the following convergences:
 \begin{itemize}[label=$\bullet$]
  \item $\ueps \to u$ in $L^2(\Omega_T; L^2(D))$ and $L^2(\Omega;L^\beta(0,T; L^2(D)))$ for any finite $\beta$,
  \item $\ueps \rightharpoonup u$ in $L^p(\Omega_T; V)$ and in $L^{\min(2,p')}(\Omega;C([0,T];V'))$,
  \item $\ueps(t) \to u(t)$ in $L^2(\Omega; L^2(D))$ for all $t \in [0,T]$,
  \item $-\frac{1}{\eps} (\ueps - \psi)^- \rightharpoonup \rho$ in $L^\alpha(\Omega_T;L^\alpha(D))$, $\alpha=\max(2,p')$,
  \item $\tilde{a}(\cdot,\ueps, \nabla \ueps) \rightharpoonup A $ in $L^{p'}(\Omega_T;L^{p'} (D))^d$,
  \item $\operatorname{div} \tilde{a}(\cdot,\ueps, \nabla \ueps) \rightharpoonup \operatorname{div} A $ in $L^{p'}(\Omega_T; V')$,
  \item $\tilde{G}(\ueps) \to G(u)$ in $L^2(\Omega_T; HS(H))$.
 \end{itemize}
\end{rem}
\begin{lemma}\label{230406_01}
\begin{align}\label{221109_01}
u(t) + \int_0^t \rho \, ds - \int_0^t \operatorname{div} A \, ds = u_0 + \int_0^t G(u) \, dW + \int_0^t f \, ds
\end{align}
holds in $V'$ a.s. in $\Omega$ for all $t \in [0,T]$. Especially, we have $u \in L^2(\Omega; C([0,T]; L^2(D)))$.
\end{lemma}
\begin{proof}
According to \eqref{220908_02} we have
\begin{align}\label{221109_02}
\partial_t\left(\ueps - \int_0^{\cdot} \tilde{G}(\ueps) \, dW\right) = \operatorname{div} \tilde{a}(\cdot,\ueps, \nabla \ueps) + \frac{1}{\varepsilon} (\ueps - \psi)^- + f
\end{align}
in $L^{p'}(\Omega_T; V')$. We fix $\phi \in W_0^{1,p}(D)$, $\xi \in C_0^{\infty}(0,T)$ and $F \in \mathcal{F}$. Now, similar as in \cite{Vallet-Zimm1}, testing \eqref{221109_02} with $\phi \xi \chi_{F}$ and passing to the limit with $\varepsilon \to 0^+$ yields
\begin{align*}
\partial_t \left(u - \int_0^{\cdot} G(u) \, dW\right) = \operatorname{div} A - \rho + f
\end{align*}
in $L^{p'}(0,T; V')$, a.s. in $\Omega$.
Therefore, the function $u$ satisfies equality \eqref{221109_01}. Moreover, \cite[Theorem 4.2.5, p. 91]{Liu-Rock} yields that $u \in L^2(\Omega; C([0,T]; L^2(D)))$.
\end{proof}
The following lemma finalizes the proof of Theorem \ref{bisexistence-regular}:
\begin{lemma}
$A= a(\cdot, u , \nabla u)$ in $L^{p'}(\Omega_T \times D)^d$.
\end{lemma}
\begin{proof}
Applying  It\^{o}'s formula to equalities \eqref{220908_02} and \eqref{221109_01} with $\Vert \cdot \Vert_2^2$ yields for all $t \in [0,T]$, a.s. in $\Omega$:
\begin{align}\label{221109_03}
\frac{1}{2} \Vert \ueps(t) \Vert_2^2 &- \int_0^t \int_D  \frac{1}{\varepsilon}(\ueps - \psi)^- \ueps \, dx \, ds + \int_0^t \int_D \tilde{a}(\cdot, \ueps, \nabla \ueps) \cdot \nabla \ueps \, dx \, ds \notag \\
&= \frac{1}{2} \Vert u_0 \Vert_2^2 + \int_0^t \langle f, \ueps \rangle_{V',V} \, ds + \int_0^t (\tilde{G}(\ueps) \, dW, \ueps)_2 + \frac{1}{2} \int_0^t \Vert \tilde{G}(\ueps) \Vert_{HS(H)}^2 \, ds
\end{align}
and
\begin{align}\label{221109_04}
\frac{1}{2} \Vert u(t) \Vert_2^2 &+ \int_0^t \int_D \rho u \, dx \, ds + \int_0^t \int_D A \cdot \nabla u \, dx \, ds \notag \\
&= \frac{1}{2} \Vert u_0 \Vert_2^2 + \int_0^t \langle f,u \rangle_{V',V} \, ds + \int_0^t (G(u) \, dW, u)_2 + \frac{1}{2} \int_0^t \Vert G(u) \Vert_{HS(H)}^2 \, ds.
\end{align}
We set $t=T$, take the expectation in \eqref{221109_03} and \eqref{221109_04} and pass to the limit with $\varepsilon \to 0^+$ in \eqref{221109_03}. Then, according to Remark \ref{rem1}, we have
\begin{align*}
\lim\limits_{\varepsilon \to 0^+} \erww{\int_0^T \int_D \tilde{a}(\cdot, \ueps, \nabla \ueps) \cdot \nabla \ueps \, dx \, dt} = \erww{\int_0^T \int_D A \cdot \nabla u \, dx \, dt}.
\end{align*}
Now, as already mentioned at the end of Section \ref{section-2.2}, a stochastic version of Minty's trick yields $A = a(\cdot, u , \nabla u)$ in $L^{p'}(\Omega_T \times D)^d$.
\end{proof}

\section{Lewy-Stampacchia's inequalities in the "regular case"}\label{sec-LS-regular}

\begin{prop}\label{LSinequlity}
Assume H$_1$-H$_6$ such that $h^- \in L^\alpha(\Omega_T; L^\alpha(D))\cap L^p(\Omega_T;V)$ where $\alpha=\max(2,p')$ with $\partial_t h^- \in L^2(\Omega_T; H)$. Then, the unique solution $(u,\rho)$ to \eqref{230307_02} in the sense of Theorem \ref{bisexistence-regular} satisfies
\begin{align}\label{LS}
\begin{aligned}
0 &\leq -\rho=\partial_t \left( u-\displaystyle\int_0^\cdot G(u)\,dW\right) - \operatorname{div}\,a(u,\nabla u)-f \\
&\leq h^-=\left(f - \partial_t \left(\psi - \int_0^\cdot G(\psi)\,dW\right) + \operatorname{div}\, a(\psi,\nabla \psi)\right)^-.
\end{aligned}
\end{align}
\end{prop}
\begin{proof}
Since the sign of $\rho$ is known and the equation for $h^-$ follows from \eqref{230308_02}, one has to prove $-\rho\leq h^-$. Since $\rho$ is the weak limit in $L^\alpha(\Omega_T; L^\alpha(D))$ of $-\frac1\varepsilon (u_\varepsilon-\psi)^-$, where $\ueps$ is given by \eqref{220908_02}, we will consider the following inequality:
\begin{align*}
-\frac1\varepsilon (u_\varepsilon-\psi)^- = -\frac1\varepsilon (u_\varepsilon-\psi)^-+h^--h^- \geq -\frac1\varepsilon (u_\varepsilon-\psi +\varepsilon h^-)^- -h^-.
\end{align*}
Lewy-Stampacchia's inequalities will be proved if one proves that $\frac1\varepsilon (u_\varepsilon-\psi +\varepsilon h^-)^-$ converges to $0$ for $\varepsilon\rightarrow 0^+$ in $L^2(\Omega\times Q_T)$.
We recall that 
\begin{align*}
  &d u_\varepsilon - \operatorname{div}\,\tilde{a}(u_\varepsilon,\nabla   u_\varepsilon)\, dt-\frac1\varepsilon( u_\varepsilon-\psi)^- \,dt =  \tilde{G}(u_\varepsilon)\,dW+ f\,dt,
  \\&
  d \psi - \operatorname{div}\,\tilde{a}(\psi,\nabla \psi)\, dt +h^+-h^- =  \tilde{G}(\psi)\,dW+ f\,dt.
\end{align*}
Then, since $\partial_t h^- \in L^2(\Omega_T; H)$, from the above equations it follows that
\begin{align}\label{230308_03}
\begin{aligned}
d(u_\varepsilon-\psi +\varepsilon h^-) &-\operatorname{div}\,\left(\tilde{a}( u_\varepsilon,\nabla u_\varepsilon)-\tilde{a}(\psi,\nabla  \psi) \right)\,dt-\frac1\varepsilon(u_\varepsilon-\psi)^- \,dt
\\
&=(\tilde{G}(u_\varepsilon)-\tilde{G}(\psi))\,dW+ (h^+-h^-+\varepsilon \partial_th^-)\,dt.
\end{aligned}
\end{align}
 Let $\delta >0$ and denote by $F_{\delta}$ an approximation of the square of the negative part as in the proof of Lemma \ref{lem2} and
 consider $\varphi_\delta(u)=\int_D F_\delta(u(x))\,dx$. 
 Applying It\^o's formula with $\varphi_\delta$, in \eqref{230308_03}, for any $t\in [0,T]$,
 \begin{align*}
  \varphi_\delta( u_\varepsilon-\psi+\varepsilon h^-)(t)&-\int_0^t\left\langle \operatorname{div}\,\left(\tilde{a}( u_\varepsilon,\nabla u_\varepsilon)-\tilde{a}(\psi,\nabla  \psi) \right), F_\delta^\prime( u_\varepsilon-\psi+\varepsilon h^-)\right\rangle \, ds \\
  &+\int_0^t \int_D\left(-\frac1\varepsilon( u_\varepsilon-\psi)^-+h^-\right)F_\delta^\prime( u_\varepsilon-\psi+\varepsilon h^-) \, dx \, ds \nonumber\\
  =&\varphi_\delta( u_\varepsilon-\psi+\varepsilon h^-)(0)+\int_0^t\langle h^++\varepsilon \partial_th^-,F_\delta^\prime( u_\varepsilon-\psi+\varepsilon h^-)\rangle \, ds \\
  &+\int_0^t(\tilde G( u_\varepsilon)-\tilde G( \psi)\,dW,F_\delta^\prime( u_\varepsilon-\psi+\varepsilon h^-))_H\\
  &+\dfrac{1}{2}\int_0^t\int_D F_\delta^{\prime\prime}(u_\varepsilon-\psi+\varepsilon h^-)\sum\limits_{k=1}^{\infty} | h_k(\max(\psi,u_\varepsilon)) - h_k(\psi)|^2 \, dx \, ds\\
  &\Longleftrightarrow I_1^\delta+I_2^\delta+I_3^\delta=I_4^\delta+I_5^\delta+I_6^\delta+I_7^\delta .
 \end{align*}
 Denote by $S_{\varepsilon}$ the set $\{  u_\varepsilon+\varepsilon h^-\leq \psi\}$ and note that $\tilde G( u_\varepsilon)=G(\psi)$  on $S_{\varepsilon}$. Therefore $I_6^\delta+I_7^\delta=0$.\\
 Passing to the limit with respect to the regularization parameter and multiplying with $\frac{1}{2\varepsilon}$, one gets
 \begin{align*}
 \dfrac{1}{2\varepsilon}\Vert( u_\varepsilon-\psi+\varepsilon h^-)^-\Vert_{H}^2(T)&+\frac1\varepsilon\int_0^T\langle \operatorname{div}\,\left(\tilde{a}( u_\varepsilon,\nabla u_\varepsilon)-\tilde{a}(\psi,\nabla  \psi) \right), ( u_\varepsilon-\psi+\varepsilon h^-)^-\rangle  \, dt \\
 &-\frac1\varepsilon\int_0^T \int_D\left(-\frac1\varepsilon( u_\varepsilon-\psi)^-+h^-\right)( u_\varepsilon-\psi+\varepsilon h^-)^- \, dx \, dt\\
 =&\dfrac{1}{2\varepsilon }\Vert( u_\varepsilon-\psi+\varepsilon h^-)^-\Vert_{H}^2(0)+\frac1\varepsilon\int_0^T\langle h^++\varepsilon \partial_th^-,-( u_\varepsilon-\psi+\varepsilon h^-)^-\rangle \, dt\\
  &\Longleftrightarrow I_1+I_2+I_3=I_4+I_5.
 \end{align*}
 Note that $I_4=0$ since $ u_\varepsilon(0)=u_0 \geq \psi(0)$. We have, using Young's inequality,
 \begin{align*}
 I_5:&=\frac1\varepsilon\int_0^T\langle h^++\varepsilon \partial_th^-,-( u_\varepsilon-\psi+\varepsilon h^-)^-\rangle \, dt\\
 &\leq\int_0^T\int_D\vert \partial_t h^-\vert \vert ( u_\varepsilon-\psi+\varepsilon h^-)^-\vert \, dx \, dt\\
 &\leq \dfrac{\varepsilon^2}{2}\int_0^T\int_D\vert \partial_t h^-\vert^2 \, dx \, dt+\dfrac{1}{2\varepsilon^2}\int_0^T\int_D \vert ( u_\varepsilon-\psi+\varepsilon h^-)^-\vert^2 \, dx \, dt.
  \end{align*}
  Now, let us estimate $I_3$,
  \begin{align*}
   I_3:&=-\frac1\varepsilon\int_0^T \int_D\left(-\frac1\varepsilon( u_\varepsilon-\psi)^-+h^-\right)( u_\varepsilon-\psi+\varepsilon h^-)^- \, dx \, dt\\
   &=\frac1{\varepsilon^2}\int_0^T \int_D \vert ( u_\varepsilon-\psi+\varepsilon h^-)^-\vert^2 \, dx \, dt.
  \end{align*}
  Therefore, after discarding the non-negative terms,
   \begin{align}
  \frac1\varepsilon\int_0^T\langle \operatorname{div}\,\left(\tilde{a}( u_\varepsilon,\nabla u_\varepsilon)-\tilde{a}(\psi,\nabla  \psi) \right), ( u_\varepsilon-\psi+\varepsilon h^-)^-\rangle  \, dt \nonumber&+\dfrac{1}{2\varepsilon^2}\int_0^T \int_D \vert ( u_\varepsilon-\psi+\varepsilon h^-)^-\vert^2 \, dx \, dt \nonumber\\
  &\leq  \dfrac{\varepsilon^2}{2}\int_0^T\int_D\vert \partial_t h^-\vert^2 \, dx \, dt. \label{LSapproche}
    \end{align}
    Let us now study the term
\begin{align*}
I_2= \frac1\varepsilon\int_0^T\langle \operatorname{div}\,\left(\tilde{a}( u_\varepsilon,\nabla u_\varepsilon)-\tilde{a}(\psi,\nabla  \psi) \right), ( u_\varepsilon-\psi+\varepsilon h^-)^-\rangle \, dt.
\end{align*}
We have
    \begin{align*}
     I_2&=\frac1\varepsilon\int_0^T\int_D \left(\tilde{a}(   u_\varepsilon,\nabla   u_\varepsilon)- \tilde{a}(   \psi,\nabla  \psi)\right)\cdot \nabla ( u_\varepsilon-\psi+\varepsilon h^-)\mathds{1}_{\{ u_\varepsilon+\varepsilon h^- <\psi \}}  \, dx \, dt\\
     &=\frac1\varepsilon\int_0^T\int_D \left(\tilde{a}(   u_\varepsilon,\nabla   u_\varepsilon)- \tilde{a}(   \psi,\nabla  \psi)\right)\cdot \nabla ( u_\varepsilon-\psi)\mathds{1}_{\{ u_\varepsilon+\varepsilon h^- <\psi \}} \, dx \, dt\\
     &\hspace{0.5cm}+\int_0^T\int_D \left(\tilde{a}(   u_\varepsilon,\nabla   u_\varepsilon)- \tilde{a}(   \psi,\nabla  \psi)\right)\cdot \nabla h^-\mathds{1}_{\{ u_\varepsilon+\varepsilon h^- <\psi \}} \, dx \, dt:=I_{2,1}+I_{2,2}.
    \end{align*}
  Since  $ u_\varepsilon \leq \psi $ on the set $S_{\varepsilon}$ it follows that $\tilde{a}(   u_\varepsilon,\nabla   u_\varepsilon)=\tilde{a}(  \psi,\nabla   u_\varepsilon)$. Therefore, thanks to $H_{2,2}$,
  \begin{align*}
   I_{2,1}&= \frac1\varepsilon\int_0^T\int_D \left(\tilde{a}( u_\varepsilon,\nabla   u_\varepsilon)- \tilde{a}(   \psi,\nabla  \psi)\right)\cdot \nabla ( u_\varepsilon-\psi)\mathds{1}_{\{ u_\varepsilon+\varepsilon h^- <\psi \}} \, dx \, dt \\
   &=\frac1\varepsilon\int_0^T\int_D \left(\tilde{a}(  \psi,\nabla   u_\varepsilon)- \tilde{a}(   \psi,\nabla  \psi)\right)\cdot \nabla ( u_\varepsilon-\psi)\mathds{1}_{\{ u_\varepsilon+\varepsilon h^- <\psi \}} \, dx \, dt \geq 0.
  \end{align*}
  Note that $\{ u_\varepsilon+\varepsilon h^- <\psi \} \subset \{  u_\varepsilon <  \psi\}$ and $\mathds{1}_{\{ u_\varepsilon+\varepsilon h^- <\psi \}} \leq \mathds{1}_{\{  u_\varepsilon <  \psi\}}$, hence
  \begin{align*}
   I_{2,2}&:=\int_0^T\int_D \left(\tilde{a}(   u_\varepsilon,\nabla   u_\varepsilon)- \tilde{a}(   \psi,\nabla  \psi)\right)\cdot \nabla h^-\mathds{1}_{\{ u_\varepsilon+\varepsilon h^- <\psi \}} \, dx \, dt\\
   & \geq - \int_0^T\int_D \vert \tilde{a}(   u_\varepsilon,\nabla   u_\varepsilon)- \tilde{a}(   \psi,\nabla  \psi)\vert \vert \nabla h^-\vert \mathds{1}_{\{ u_\varepsilon <\psi \}} \, dx \, dt\\
   & \geq - \int_0^T\int_D \vert \tilde{a}(  \psi,\nabla   u_\varepsilon)- \tilde{a}(   \psi,\nabla  \psi)\vert \vert \nabla h^-\vert \mathds{1}_{\{ u_\varepsilon <\psi \}} \, dx \, dt.
  \end{align*}
From \eqref{220908_09} in Lemma \ref{lem2} we have
\begin{align*}
 \erww{\int_{Q_T} \vert\left(\tilde{a}( \psi,\nabla  u_\varepsilon)-\tilde{a}( \psi,\nabla \psi)\right)\cdot\nabla ( u_\varepsilon-\psi)^-\vert \, dx \, dt} \to 0 \text{ for } \eps \to 0^+
\end{align*}
and up to a subsequence, 
\begin{align*}
|\tilde{a}( \psi,\nabla  u_\varepsilon)-\tilde{a}( \psi,\nabla \psi)|\mathds{1}_{\{u_\varepsilon<\psi\}}|\nabla ( u_\varepsilon-\psi)| \to 0
\end{align*}
for $\eps\rightarrow 0^+$ a.e. in $\Omega_T \times D$. Let us fix $(\omega,t,x)$ such that the above convergence holds. For $\eps>0$ we set 
\begin{align*}
x_{\varepsilon}:=|\tilde{a}( \psi,\nabla  u_\varepsilon)-\tilde{a}( \psi,\nabla \psi)|\mathds{1}_{\{u_\varepsilon<\psi\}} (\omega,t,x)
\end{align*}
and 
\begin{align*}
y_{\varepsilon} := |\nabla ( u_\varepsilon-\psi)|.
\end{align*}
Then, the sequences $(x_{\varepsilon})_{\eps>0}$ and $(y_{\varepsilon})_{\eps>0}$ have the following properties:
\begin{itemize}
\item [$(i)$]$x_{\varepsilon} y_{\varepsilon} \to 0$ for $\varepsilon \to 0^+$,
\item[$(ii)$]For any subsequence $\varepsilon'$ of $\varepsilon$ we have: $y_{\varepsilon'} \to 0$ for $\varepsilon' \to 0^+ \Longrightarrow x_{\varepsilon'} \to 0$ for $\varepsilon' \to 0^+$.
\end{itemize}
We want to show that $x_{\varepsilon} \to 0$ for $\varepsilon \to 0^+$. If we assume that this is not the case, then there exist $\delta>0$ and a subsequence $\varepsilon''$ of $\varepsilon$ such that $|x_{\varepsilon''}| \geq \delta$ for any $\varepsilon''$.
Then we have
\begin{align*}
\delta |y_{\varepsilon''}| \leq |x_{\varepsilon''} y_{\varepsilon''}| \to 0~ \text{as} ~ \varepsilon'' \to 0^+.
\end{align*}
Hence $y_{\varepsilon''} \to 0$ as $\varepsilon'' \to 0^+$ and property $(ii)$ yields $x_{\varepsilon''} \to 0$ as $\varepsilon'' \to 0^+$, a contradiction to the assertion.
Now, 
\[|\tilde{a}( \psi,\nabla  u_\varepsilon)-\tilde{a}( \psi,\nabla \psi)|\mathds{1}_{\{u_\varepsilon<\psi\}} \to 0\] 
for $\eps\rightarrow 0^+ $ a.e. in $\Omega_T \times D$ and since this sequence is bounded in $L^{p'}(\Omega_T; L^{p'}(D))$, it converges weakly to $0$ in $L^{p'}(\Omega_T; L^{p'}(D))$ and as $|\nabla h^-|\in L^p(\Omega_T;L^p( D))$, one has that 
\[\liminf\limits_{\varepsilon \to 0^+} \erww{ I_{2,2}} \geq 0.\]
Taking expectation and passing to the limit with $\varepsilon\rightarrow 0^+$ in \eqref{LSapproche}, it follows that 
  \begin{align*}
   \frac1\varepsilon( u_\varepsilon-\psi+\varepsilon h^-)^- \to 0 \text{  in } L^2(\Omega_T \times D),
  \end{align*}
and the result holds.
  \end{proof}

\section{Well-posedness and LS inequality (general case)}\label{GC}

Let us start this section by recalling that, even in the general case, the solution to Problem \eqref{I} in the sense of Definition \ref{def} is unique thanks to Proposition \ref{bisuniqueness1}. 
\begin{lemma}\label{230310_lem01}
The positive cone of $L^p(\Omega_T;V)\cap L^{\alpha}(\Omega_T;L^{\alpha}(D))$ with time-derivative in $L^2(\Omega_T;L^2(D))$  is dense in the positive cone of $L^{p'}(\Omega_T;V')$. 
\end{lemma}
\begin{proof} Note first that $L^{\alpha}(\Omega_T;L^{\alpha}(D))$ embeds in $L^{p'}(\Omega_T;V')$. Following the arguments developed  in \cite[Lemma 4.1]{GMYV} mainly based on monotonicity techniques and truncation, one has just to add the predictability assumption to the spaces of type $L^r(\Omega_T; X )$ to obtain that the positive cone of $L^p(\Omega_T;V)\cap L^\alpha(\Omega_T;L^\alpha(D))$ is dense in the positive cone of $L^{p'}(\Omega_T;V')$.

Thus, for any non-negative element $\overline{h}$  of $L^{p'}(\Omega_T;V')$, we find a non-negative approximation sequence $(\overline{h}_n)\subset L^p(\Omega_T;V)\cap L^\alpha(\Omega_T;L^\alpha(D))$ such that $\overline{h}_n\rightarrow\overline{h}$ for $n\rightarrow\infty$ in $L^{p^\prime}(\Omega_T;\Vprime)$. For $n\in\mathbb{N}$ by Landes regularisation, we obtain $h_n$ by solving the random equation
\begin{align*}
\frac1{n+1}\frac{d h_n}{dt}+h_n=\overline{h}_n \text{ with } h_n(0)=0,
\end{align*}
more precisely
\begin{align*}
h_n(t)=(n+1)\int_0^t e^{(n+1)(s-t)}\overline{h}_n(s)\,ds
\end{align*}
for all $t\in [0,T]$, $\mathds{P}$-a.s. in $\Omega$. Then, the sequence $(h_n)$ has the desired properties.
\end{proof}
In this section, we will finalize the proof of Theorem \ref{MTh}. According to Lemma \ref{230310_lem01}, let $(h_n) \subset L^p(\Omega_T;V)$ be a sequence of non-negative mappings with $\partial_t h_n \in L^{2}(\Omega_T;L^2(D))$ satisfying $h_n \to h^-$ in $L^{p^\prime}(\Omega_T;\Vprime)$ for $n\rightarrow \infty$. Then, associated with $h_n$, for $n\in\mathbb{N}$ we define
\begin{align*}
f_n=\partial_t\left(\psi-\int_0^{\cdot}G(\psi)\,dW\right)-\operatorname{div}\,a(\cdot,\psi,\nabla \psi)+h^+-h_n=f+h^--h_n.
\end{align*}
 Note that $f_n$ converges to $f$ in $L^{p^\prime}(\Omega_T;\Vprime)$ and since $h^+-h_n \in L^{p}(\Omega_T;V)^*$, Assumption H$_5$ is satisfied by $f_n$. From Theorem \ref{bisexistence-regular} it follows that Problem \eqref{I} with data $(u_0,\psi,f_n)$ and $h^-$ replaced by $h_n$  has a unique solution  $(u_n,\rho_n)$ in the sense of Definition \ref{def} with $\rho_n\in L^{\alpha}(\Omega_T;L^{\alpha}(D))$ for any $n\in\mathbb{N}$, \textit{i.e.},
\begin{align}\label{biseqn-un}
u_n(t)+\int_0^t\rho_n \,ds-\int_0^t \operatorname{div}\,a(\cdot,u_n,\nabla u_n)\,ds=u_0+\int_0^t G(u_n)\,dW+\int_0^t f_n\,ds,
\end{align}
for all $t\in[0,T]$, a.s. in $\Omega$. From Proposition \ref{LSinequlity} it follows that $(u_n,\rho_n)$ satisfies Lewy-Stampacchia's inequalities \eqref{LSfinal} in the sense of Definition \ref{230310_def1}, \textit{i.e.},
\begin{align}\label{230310_03} 
0\leq -\rho_n \leq h_n
\end{align}
in $L^{\alpha}(D)$, $d\mathds{P}\otimes dt$-a.s. in $\Omega_T$ for all $n\in\mathbb{N}$. 
\begin{lemma}\label{lem2Bis}
There exists $\rho\in L^{p'}(\Omega_T;V')$ with $-\rho\in L^{p'}(\Omega_T;V')^+$  such that, passing to a subsequence if necessary, $\rho_n\rightharpoonup \rho$ in $L^{p'}(\Omega_T;V')$ for $n\rightarrow \infty$. In particular
\begin{align}\label{230310_02}
-\rho\leq h^{-}
\end{align}
in the sense that $h^{-}+\rho$ is in $L^{p'}(\Omega_T;V')^+$.
\end{lemma}
\begin{proof}
For any $\varphi \in L^p(\Omega_T;V)$, from \eqref{230310_03} it follows that
\begin{align*}
\left|\erww{\int_0^T\langle -\rho_n, \varphi \rangle_{V',V} \,ds} \right|
&\leq \erww{\int_0^T \langle -\rho_n, \varphi^+ \rangle_{V',V}  \,ds} + \erww{\int_0^T  \langle - \rho_n, \varphi^- \rangle_{V',V}  \,ds}
\\
 &\leq \erww{\int_0^T \langle h_n, \varphi^+ \rangle_{V',V} \,ds}+ \erww{\int_0^T  \langle h_n, \varphi^- \rangle_{V',V} \,ds}\\
 & \leq 2 \Vert h_n\Vert_{L^{p^\prime}(\Omega_T;V^\prime)}\Vert \varphi\Vert_{L^{p}(\Omega_T;V)}.
 \end{align*}
Since $(h_n)_n$ converges to $h^-$ in $L^{p^\prime}(\Omega_T;V^\prime)$, there exists a constant $C\geq 0$ not depending on $n\in\mathbb{N}$ such that $\Vert h_n\Vert_{L^{p^\prime}(\Omega_T;V^\prime)}\leq C$ for all $n\in\mathbb{N}$.
Therefore $(\rho_n)_n$ is bounded in $L^{p^\prime}(\Omega_T;V^\prime)$  independently of $n\in\mathbb{N}$ and, up to a subsequence, $\rho$ exists such that $\rho_n\rightharpoonup \rho$ in $L^{p'}(\Omega_T;V')$ for $n\rightarrow \infty$. Now, \eqref{230310_02} follows by passing to the limit in \eqref{230310_03}.
\end{proof}

In the following,  
we will show that $(u,\rho)$, a limit of $(u_n,\rho_n)$ with $\rho$ satisfying Lemma \ref{lem2Bis}, will be a solution to Problem \eqref{I} in the sense of Definition \ref{def}.  Thanks to \eqref{230310_02}, Lewy-Stampacchia's inequalities will be satisfied and it is left to show that we can pass to the limit in \eqref{biseqn-un}. 
\begin{lemma}\label{230402_01}
We have the following uniform bounds with respect to $n\in\mathbb{N}$:
\begin{itemize}
\item[$i.)$] $(u_n)_n$ is bounded in $L^p(\Omega_T; V)$ and in $L^2(\Omega;C([0,T];H))$.
\item[$ii.)$]  $(a(u_n,\nabla u_n))_n$ is bounded in $L^{p'}(\Omega_T\times D)^d$,  $(-\Div a(u_n,\nabla u_n))_n$ is bounded in $L^{p'}(\Omega_T;V')$.
\item[$iii.)$]  $\left(u_n-\int_0^{\cdot} G(u_n)\, dW \right)_n$ is bounded in $L^2(\Omega;C([0,T];H))$ and $\left(\partial_t\left(u_n-\int_0^{\cdot} G(u_n) \, dW\right)\right)_n$ is bounded in $L^{p'}(\Omega;L^{p'}(0,T;V'))$,  hence $\left(u_n-\int_0^{\cdot} G(u_n)\, dW \right)_n$is bounded in $L^{\min(2,p')}(\Omega;\mathcal{W})$, where $\mathcal{W}$ is defined as in the proof of Lemma \ref{230406_01}.
\end{itemize}
\end{lemma}
\begin{proof}
Since $\langle \rho_n,u_n-\psi\rangle =0$ and $u_n\geq \psi$ for all $n\in\mathbb{N}$,  we may repeat the arguments in the proof of Lemma \ref{lem1} to obtain $i.)$ and $ii.)$.  Now,   Burkholder-Davis-Gundy inequality,  $ii.)$,  and 
\[\partial_t\left(u_n-\int_0^{\cdot} G(u_n)\, dW \right)=\Div a(u_n,\nabla u_n)-\rho_n+f_n\]
yields $iii.)$.
\end{proof}
\begin{lemma}
Let $\mu_n$ be the law of the random variable 
\[Z:=(u_n,G(u_n),\rho_n,f_n,\psi, G(0),\sigma,W,u_0)\]
on the space
\begin{align*}
\mathcal{X}:=& L^4(0,T;H)\times L^2(0,T;HS(H))\times \Pi\times L^{p'}(0,T;V')\times L^p(0,T;V)\cap C([0,T];H)\\
&\times L^2(0,T;HS(H))\times L^4(0,T;\Sigma)\times C([0,T];U)\times H
\end{align*}
with $\Pi:=\left(W^{1,p}(Q_T)\cap L^p(0,T;W^{2,p}_0(D))\right)^{\ast}$. Then, passing to a not relabeled subsequence if necessary, $(\mu_n)_n$ converges to a probability measure $\mu_{\infty}$ with respect to the narrow topology on $\mathcal{X}$.
\end{lemma}
\begin{proof}
We show that every component of $Z$ is a tight sequence and apply Prokhorov's theorem. Since $(\rho_n)_n$ is bounded in $L^p(0,T;V')$, the sequence of laws of $(\rho_n)_n$ on $\Pi$ is tight. The tightness results with respect to the other components follow from Lemma \ref{230402_01} repeating the arguments of the stochastic compactness argument in Section \ref{S1}.
\end{proof}
Now, similar as in Section \ref{S1}, \cite[Thm. 6.7]{Bil99}  yields the existence of a probability space $(\overline{\Omega}, \overline{\mathcal{F}},  \overline{\mathds{P}})$ which may be chosen as $([0,1],\mathcal{B}([0,1]),\operatorname{Leb})$ with $\operatorname{Leb}$ being the Lebesgue measure on $[0,1]$ and a family of random vectors
\begin{align*}
(\overline u_n, \overline G_n,\overline \rho_n, \overline f_n, \overline \psi_n,\overline G_0^n,\overline \sigma_n,\overline W_n,\overline u_0^n)
\end{align*}
on $(\overline \Omega, \overline{\mathcal{F}}, \overline{\mathds{P}})$ with values in $\mathcal{X}$ having the same law as $Z$. Moreover, there exist
\begin{itemize}
\item[$i.)$] a random variable $\overline u$ with values in $L^4(0,T;L^2(D))$ such that $\overline u_n$ converges $\overline{\mathds{P}}$-a.s. to $\overline u$ in $L^4(0,T;L^2(D))$ for $n\rightarrow \infty$,
\item[$ii.)$] a random variable $\widetilde{\overline{G}}$ with values in $L^2(0,T;HS(H))$ such that $\overline G_n$ converges $\overline{\mathds{P}}$-a.s. to $\widetilde{\overline{G}}$ in $L^2(0,T;HS(H))$ for $n\rightarrow \infty$,
\item[$iii.)$] a random variable $\overline{\rho}$ with values in $\Pi$ such that $\overline{\rho}_n$ converges $\overline{\mathds{P}}$-a.s. to $\overline{\rho}$ in $\Pi$ for $n\rightarrow \infty$,
\item[$iv.)$] a random variable $\overline f$ with values in $L^{p'}(0,T;V')$ such that $\overline f_n$ converges $\overline{\mathds{P}}$-a.s. to $\overline{f}$ in $L^{p'}(0,T;V')$ for $n\rightarrow \infty$,
\item[$v.)$] a random variable $\overline \psi$ with values in $L^p(0,T;V)\cap C([0,T]; H)$ such that $\mathcal{L}(\overline \psi) = \mathcal{L}(\psi)$ and $\overline \psi_n$ converges $\overline{\mathds{P}}$-a.s. to $\overline \psi$ in $L^p(0,T;V)\cap C([0,T]; H)$ for $n\rightarrow \infty$,
\item[$vi.)$] a random variable $\overline G_0$ with values in $L^2(0,T;HS(H))$ such that $\mathcal{L}(\overline G_0) = \mathcal{L}(G(0))$ and $\overline G_0^n$ converges $\overline{\mathds{P}}$-a.s. to $\overline G_0$ in $L^2(0,T;HS(H))$ for $n\rightarrow \infty$,
\item[$vii.)$] a random variable $\overline \sigma$ with values in $L^4(0,T; \Sigma)$ such that $\mathcal{L}(\overline \sigma) = \mathcal{L}(\sigma)$ and $\overline \sigma_n$ converges $\overline{\mathds{P}}$-a.s. to $\overline \sigma$ in $L^4(0,T; \Sigma)$ for $n\rightarrow \infty$,
\item[$viii.)$] a random variable $\overline W$ with values in $C([0,T];U)$ such that $\mathcal{L}(\overline W) = \mathcal{L}(W)$ and $\overline W_n$ converges $\overline{\mathds{P}}$-a.s. to $\overline W$ in $C([0,T];U)$ for $n\rightarrow \infty$,
\item[$ix.)$] a random variable $\overline u_0$ with values in $H$ such that $\mathcal{L}(\overline u_0) = \mathcal{L}(u_0)$ and $\overline u_0^n$ converges $\overline{\mathds{P}}$-a.s. to $\overline u_0$ in $H$ for $n\rightarrow \infty$.
\end{itemize}
Again, we will denote the expectation with respect to $(\overline \Omega, \overline{\mathcal{F}}, \overline{\mathbb{P}})$ by $\overline{\mathds{E}}$. Moreover, as in Section \ref{S1}, the random variables $(\overline u_n, \overline G_n,\overline \rho_n, \overline f_n, \overline \psi_n,\overline G_0^n,\overline \sigma_n,\overline W_n,\overline u_0^n)$ have the same properties as $Z$, especially we have 
\begin{itemize}
\item $\overline u_n \in L^2(\overline{\Omega};C([0,T];L^2(D)))$ for all $n \in \mathbb{N}$,
\item $\overline \sigma_n\in L^2(\overline{\Omega}; L^2(0,T;Lip_0(H;HS(H))))$,
\item $\overline G_n(\overline u_n) := \overline G_n = \overline G_0^n + \overline \sigma_n(\overline u_n)$,
\item $\overline f \in L^{p'}(\overline{\Omega}; L^{p'}(0,T;V'))$ and $\overline{f}_n \to \overline f$ in $L^{p'}(\overline{\Omega}; L^{p'}(0,T;V'))$ by Brezis-Lieb lemma (see  \cite{BrezisLieb}),
\item $\overline{\rho} \in L^{p'}(\overline{\Omega}; L^{p'}(0,T;V'))$ and $\overline{\rho}_n \rightharpoonup \overline{\rho}$ in $L^{p'}(\overline{\Omega}; L^{p'}(0,T;V'))$,
\item $\overline W_n(0)=0$.
\end{itemize}
\begin{definition}
For $t\in [0,T]$ and $n \in \mathbb{N}$ we define $\overline{\mathcal{F}}^{n'}_t$ to be the smallest sub $\sigma$-field of 
$\overline{\mathcal{F}}$ generated by $\overline u_0^n$, $\overline W_n(s)$, $\overline \psi_n(s)$, $\overline u_n(s)$, $\int_0^s \overline G_0^n(r)\,dr$, $\int_0^s \overline G_n(r)\,dr$, $\int_0^s \overline f_n(r)\,dr$, $\int_0^s \overline \rho_n \, dr$, $\int_0^s \overline \sigma_n(r)\,dr$ for $0\leq s\leq t$. The right-continuous, $\overline{\mathds{P}}$-augmented filtration of $(\overline{\mathcal{F}}^{n'}_t)_{t\in [0,T]}$, denoted by $(\overline{\mathcal{F}}^{n}_t)_{t\in [0,T]}$ is defined by
\[\overline{\mathcal{F}}^{n}_t:=\bigcap_{T\geq s>t}\sigma\left[\overline{\mathcal{F}}^{n '}_s\cup \{\mathcal{N}\in\overline{\mathcal{F}} \, : \, \overline{\mathds{P}}(\mathcal{N})=0)\}\right].\]
Similarly, we define $\overline{\mathcal{F}}^{'}_t$ to be the smallest sub $\sigma$-field of 
$\overline{\mathcal{F}}$ generated by $\overline u_0$, $\overline W(s)$, $\overline \psi(s)$, $\overline u(s)$, $\int_0^s \overline G_0(r)\,dr$, $\int_0^s \overline G(r)\,dr$, $\int_0^s \overline f(r)\,dr$, $\int_0^s \overline \rho \, dr$, $\int_0^s \overline \sigma(r)\,dr$ for $0\leq s\leq t$. The right-continuous, $\overline{\mathds{P}}$-augmented filtration of $(\overline{\mathcal{F}}^{'}_t)_{t\in [0,T]}$, denoted by $(\overline{\mathcal{F}}_t)_{t\in [0,T]}$ is defined by
\[\overline{\mathcal{F}}_t:=\bigcap_{T\geq s>t}\sigma\left[\overline{\mathcal{F}}^{'}_s\cup \{\mathcal{N}\in\overline{\mathcal{F}} \, : \, \overline{\mathds{P}}(\mathcal{N})=0)\}\right].\]
\end{definition}
We continue by giving several lemmas similarly to Section \ref{S1} to obtain a limit equation. Since the proofs of the upcoming lemmas are the same or very similar to the corresponding lemmas in Section \ref{S1} we only state minor differences to the proofs in Section \ref{S1} or some collections of properties which are needed to finalize the proof of Theorem \ref{MTh}.
\begin{lemma}\label{230405_lem1}
There exist $(\overline{\mathcal{F}}^{n}_t)_{t\in [0,T]}$-predictable $d\overline{\mathds{P}}\otimes dt$-representatives of $\overline G_0^n$, $\overline f_n$, $\overline \rho_n$ and $\overline \sigma_n$. Moreover, we have $-\overline \rho_n \in L^{p'}(\overline{\Omega}_T; V')^+$.
\end{lemma}
\begin{lemma}\label{230405_lem2}
For any $n \in \N$, $\overline{W}_n=(\overline{W}_n(t))_{t\in [0,T]}$ is a $U$-valued, square integrable $(\overline{\mathcal{F}}^{n}_t)_{t\in [0,T]}$-martingale with quadratic variation process $tQ$ for any $t\in [0,T]$ and $\overline{W}=(\overline{W}(t))_{t\in [0,T]}$ is a $U$-valued, square integrable $(\overline{\mathcal{F}}_t)_{t\in [0,T]}$-martingale with quadratic variation process $tQ$ for any $t\in [0,T]$. 
\end{lemma}
\begin{proof}
In the proof of Lemma \ref{230105_lem2} we replace $X^{\delta}$, $\overline{X}^{\delta}$ and $\overline{X}^{\infty}$ by $Y^n$, $\overline{Y}^n$ and $\overline{Y}$, where
\begin{align*}
Y^n_t:=&\left(u_0, W(t), u_n(t), \psi,\int_0^{t} G(0)(r)\,dr,\int_0^{t} f_n(r)\,dr, \int_0^{t} \rho_n \, dr, \int_0^{t} \sigma(r)\,dr\right),
\\
\overline Y^n_t:=&\left(\overline u_0^n, \overline W_n(t), \overline u_n(t), \overline \psi_n,\int_0^{t} \overline G_0^n(r)\,dr,\int_0^{t} \overline f_n(r)\,dr, \int_0^{t} \overline \rho_n \, dr, \int_0^{t} \overline \sigma_n(r)\,dr\right),
\\
\overline Y_t:=&\left(\overline u_0, \overline W(t), \overline u(t), \overline \psi,\int_0^{t} \overline G_0(r)\,dr,\int_0^{t} \overline f(r)\,dr, \int_0^{t} \overline \rho \, dr, \int_0^{t} \overline \sigma(r)\,dr\right).
\end{align*}
Moreover, we replace $\Upsilon$ by $\tilde{\Upsilon}$, where, for $0 \leq s \leq t \leq T$, $\tilde{\Upsilon}$ goes through all bounded and continuous real-valued functions defined on the space
\begin{align*}
L^2(D)\times C([0,s];U)\times C([0,s];H)^2 \times C([0,s];HS(H)) \times C([0,s]; V')^2 \times C([0,s]; \Sigma).
\end{align*}
\end{proof}
\begin{remark}
From Lemma \ref{230405_lem2} it follows that for each $n \in \N$, $\overline{W}_n=(\overline{W}_n(t))_{t\in [0,T]}$ is a $(\overline{\mathcal{F}}^{n}_t)_{t\in [0,T]}$-adapted $Q$-Wiener process with values in the separable Hilbert space $U$ and from Lemma \ref{230405_lem1} it follows that $\overline{G}_n(\overline u_n)$ has a predictable $d\overline{\mathds{P}}\otimes dt$-representative and therefore the It\^{o} stochastic integral $\int_0^t \overline{G}_n(\overline{u}_n)\,d\overline{W}_n$ is well-defined. Analogously, $\overline{W}=(\overline{W}(t))_{t\in [0,T]}$ is a $(\overline{\mathcal{F}}_t)_{t\in [0,T]}$-adapted $Q$-Wiener process with values in the separable Hilbert space $U$ and from Lemma \ref{230405_lem1} it follows that $\overline{G}(\overline u)$ has a predictable $d\overline{\mathds{P}}\otimes dt$-representative and therefore the It\^{o} stochastic integral $\int_0^t \overline{G}(\overline{u})\,d\overline{W}$ is well-defined.
\end{remark}
\begin{lemma}
For any $n \in \mathbb{N}$ and $t\in [0,T]$, let us define
\begin{align}
M_{n}(t):=\overline{u}_n(t)-\overline{u}_0^n+\int_0^t (-\Div a(\overline{u}_n, \nabla \overline{u}_n)+\overline{\rho}_n -\overline{f}_n)\,ds.
\end{align}
The stochastic process $(M_n(t))_{t\in [0,T]}$ is a square-integrable, continuous $(\overline{\mathcal{F}}^n_t)_{t\in [0,T]}$ martingale with values in $L^2(D)$ such that 
\begin{align}
\ll M_n\gg_t&=\int_0^t (\overline{G}_n(\overline{u}_n)\circ Q^{1/2})\circ(\overline G_n(\overline{u}_n)\circ Q^{1/2})^{\ast}\,ds
\end{align}
\begin{align}
\ll \overline{W}_n, M_n\gg_t&=\int_0^t Q\circ \overline{G}_n(\overline{u}_n)\,ds.
\end{align}
\end{lemma}
\begin{lemma}
For all $n \in \mathbb{N}$ and all $t\in [0,T]$ we have
\[M_n(t)=\int_0^t \overline{G}_n(\overline{u}_n)\,d\overline{W}_n\]
in $L^2(\overline{\Omega};L^2(D))$. In particular,
\begin{align*}
&d\overline{u}_n - \Div a(\overline{u}_n,\nabla \overline{u}_n)\,dt + \overline{\rho}_n \,dt = \overline{f}_n \,dt + \overline{G}_n(\overline{u}_n) \,d\overline{W}_n,
\end{align*}
or equivalently
\begin{align*}
\partial_t \left(\overline{u}_n - \int_0^{\cdot}\overline{G}_n(\overline{u}_n)\, d\overline{W}_n\right) - \Div a(\overline{u}_n,\nabla \overline{u}_n) + \overline{\rho}_n = \overline{f}_n.
\end{align*}
\end{lemma}
Therefore, applying the uniqueness result from Proposition \ref{bisuniqueness1}, we may conclude that, for any $n \in \mathbb{N}$, $\overline{u}_n$ is the unique solution to Problem \eqref{biseqn-un} with initial datum $\overline{u}_0^n$ with respect to new stochastic basis $(\overline \Omega, \overline{\mathcal{F}}, (\overline{\mathcal{F}}^n_t)_{t\in [0,T]},\overline{\mathbb{P}})$, associated with $\overline W_n$.
\begin{lemma}
For $n \to \infty$,
\[\int_0^{\cdot} \overline{G}_n\, d\overline{W}_n\rightarrow \int_0^{\cdot} \overline{G}\,d\overline{W}\]
in probability in $L^2(0,T;L^2(D))$.
\end{lemma}
\begin{lemma}
For $n \to \infty$, $\overline\sigma_n(\overline{u}_n)\rightarrow\overline{\sigma}(\overline{u})$ $\overline{\mathbb{P}}$-a.s. in $L^2(0,T;HS(H))$. Moreover,
$\overline{G} (\overline{u}):= \overline{G}=\overline{G}_0+\overline{\sigma}(\overline{u})$.
\end{lemma}
\begin{lemma}
$\overline{u}$ is a $(\overline{\mathcal{F}}_t)_{t\in [0,T]}$-adapted, square-integrable stochastic process with continuous paths in $L^2(D)$. Moreover, $\overline{u} \in L^p(\overline\Omega \times (0,T);V)$ and there exists $A_{\infty}\in L^{p'}(\overline\Omega \times (0,T)\times D)$ such that
\begin{align}\label{230405_01}
d\overline{u}  - \operatorname{div} \, A_\infty \,dt + \overline{\rho} \,dt = \overline{f} \,dt + \overline{G}(\overline{u})\,d\overline{W}
\end{align}
or, equivalently,
\begin{align*}
\partial_t \left( \overline{u} - \int_0^{\cdot}\overline{G}(\overline{u})d\overline{W}\right) - \operatorname{div}\, A_\infty + \overline{\rho} = \overline{f}.
\end{align*}
\end{lemma}
\begin{proof}
The estimates of this section and similar arguments as in Lemma \ref{230406_02} give us the following results:
\begin{itemize}
\item $\overline{u}_n \to \overline u$ in $L^4(0,T;L^2(D))$ a.s. in $\overline\Omega$; in $L^s(\overline\Omega;L^4(0,T;L^2(D)))$ for any $s<2$; in $L^2(D)$ a.e. in $\overline{\Omega} \times(0,T)$ and a.e. in $\overline{\Omega} \times(0,T)\times D$, 
\item $\overline{u}_n \rightharpoonup \overline u$ in $L^2(\overline\Omega;L^4(0,T;L^2(D)))$; in $L^p(\overline\Omega \times (0,T);V)$ and in $L^{\min(2,p')}(\overline \Omega;C([0,T];V'))$,
\item $\overline{u}_n(t) \rightharpoonup \overline{u}(t)$ in $L^2(\overline \Omega \times D)$ for any $t \in [0,T]$,
\item $\int_0^t \overline{G}_n(\overline{u}_n)d\overline W_n \rightharpoonup \int_0^t \overline{G}(\overline{u})d\overline W$ weakly in $L^2(\overline\Omega;L^2(D)))$ for all $t \in [0,T]$,
\item $a(\overline{u}_n, \nabla \overline{u}_n) \rightharpoonup A_{\infty}$ in $L^p(\overline{\Omega}; L^p(0,T;V))$ for some $A_{\infty} \in L^p(\overline{\Omega}; L^p(0,T;V))$,
\item $\overline{\rho}_n \rightharpoonup \overline{\rho}$ in $L^{p'}(\overline\Omega \times (0,T); V')$,
\item $\overline{f}_n \to \overline{f}$ in $L^{p'}(\overline\Omega \times (0,T); V')$.
\end{itemize}
These results may conclude equality \eqref{230405_01} and $\overline{u} \in L^2(\overline\Omega; C([0,T]; L^2(D)))$.
\end{proof}
\begin{lemma}
We have
 $
A_{\infty}= a(\overline{u}, \nabla \overline{u})
$ 
in $L^{p'}(\overline{\Omega}; L^{p'}(0,T;L^{p'}(D)))$. Especially, $A_{\infty}$ has a predictable representation and $A_{\infty} \in L^{p'}(\overline\Omega \times (0,T); L^{p'}(D))$. Moreover, we have $ \langle \overline{\rho}, \overline{u}-\overline{\psi} \rangle = 0$ a.e. in $\overline\Omega \times (0,T)$.
\end{lemma}
\begin{proof}
By It\^o's energy formula with exponential weight, for any $\beta \in \R^+$ and any $t \in [0,T]$,
\begin{align*}
&\frac{e^{-\beta  t}}2\erws{\|\overline u_n(t)\|^2_{L^2}} + \erws{\int_0^t  e^{-\beta  s} \int_{D} a(\overline{u}_n,\nabla \overline{u}_n)\nabla \overline u_n \,dx\,ds} + \erws{\int_0^t e^{-\beta  s} \langle \overline{\rho}_n,\overline u_n \rangle \,ds} \\ 
=&\erws{\int_0^te^{-\beta  s}\langle \overline f_n,\overline u_n\rangle\,ds} + \frac12\erws{\int_0^t e^{-\beta  s}\|\overline G_n( \overline u_n)\|_{HS}^2\,ds} + \frac12\erws{\|\overline u_0^n\|^2_{L^2}}-\frac\beta 2\erws{\int_0^t e^{-\beta s}\|\overline u_n\|_{L^2}^2\,ds}
\end{align*}
and
\begin{align*}
&\frac{e^{-\beta t}}2\erws{\|\overline u(t)\|^2_{L^2}}+ \erws{\int_0^t e^{-\beta s} \langle \overline{\rho},\overline u \rangle \,ds} + \erws{\int_0^t  e^{-\beta s} \int_{D} A_\infty\nabla \overline u \,dx\,ds} 
\\  =&
\erws{\int_0^te^{-\beta s}\langle \overline f,\overline u\rangle\,ds} + \frac12\erws{\int_0^t e^{-\beta s}\|\overline G(\overline u)\|_{HS}^2\,ds} + \frac12\erws{\|\overline u_0\|^2_{L^2}}-\frac\beta2\erws{\int_0^t e^{-\beta s}\|\overline u\|_{L^2}^2\,ds}.
\end{align*}
Note that $\mathcal{L}(\overline \psi_n) = \mathcal{L}(\psi)=\mathcal{L}(\overline \psi)$ for each $n \in \N$, $\| \overline \psi_n \|_{L^p(\overline \Omega; L^p(0,T;V))} = \| \psi \|_{L^p(\Omega; L^p(0,T;V))} = \| \overline \psi \|_{L^p(\overline \Omega; L^p(0,T;V))} $ and therefore by Brezis-Lieb lemma (see  \cite{BrezisLieb}) $\overline \psi_n \to \overline \psi$ in $L^p(\overline \Omega; L^p(0,T;V))$. Hence
\begin{align*}
\erws{\int_0^t e^{-\beta s} \langle \overline{\rho}_n,\overline u_n \rangle \,ds}=\erws{\int_0^t e^{-\beta s} \langle \overline{\rho}_n,\overline \psi_n \rangle \,ds} \to \erws{\int_0^t e^{-\beta s} \langle \overline{\rho},\overline \psi \rangle \,ds} .
\end{align*}
Now one follows verbatim the arguments in Lemma 2.8 and we obtain
\begin{align*}
\erws{\int_0^t e^{-\beta s} \langle \overline{\rho},\overline \psi - \overline u \rangle \,ds} + \limsup_{n\rightarrow \infty}\erws{\int_0^t \!\! e^{-\beta s} \int_{D} a(\overline{u}_n,\nabla \overline{u}_n)\nabla \overline u_n \,dx\,ds} 
\leq 
\erws{\int_0^t  e^{-\beta s} \int_{D} A_\infty\nabla \overline u \,dx\,ds}. \end{align*}
Note that $\erws{\int_0^t e^{-\beta s} \langle \overline{\rho},\overline \psi - \overline u \rangle \,ds} \geq 0$ for any $t \in [0,T]$
so that $$\limsup_{n\rightarrow \infty}\erws{\int_0^T \!\! e^{-\beta t} \int_{D} a(\overline{u}_n,\nabla \overline{u}_n)\nabla \overline u_n \,dx\,dt} 
\leq 
\erws{\int_0^T  e^{-\beta t} \int_{D} A_\infty\nabla \overline u \,dx\,dt}.$$
Then, since $A_\infty$ is the weak limit of $a(\overline{u}_n,\nabla \overline{u}_n)$, it is routine, see \textit{e.g.} \cite[Subsec. 2.5]{Vallet-Zimm1} to prove that $$\lim_{n\rightarrow \infty}\erws{\int_0^T \!\! e^{-\beta t} \int_{D} a(\overline{u}_n,\nabla \overline{u}_n)\nabla \overline u_n \,dx\,dt} 
=
\erws{\int_0^T  e^{-\beta t} \int_{D} A_\infty\nabla \overline u \,dx\,dt}$$ and we obtain $A_\infty=a(\overline{u},\nabla \overline{u})$ and $$\erws{\int_0^T e^{-\beta s} \langle \overline{\rho},\overline \psi - \overline u \rangle \,dt} = 0.$$
\end{proof}
Thus, $\overline u$ is solution of equation \eqref{230307_02} with respect to the stochastic basis $(\overline \Omega, \overline{\mathcal{F}}, (\overline{\mathcal{F}}_t)_{t\in [0,T]},\overline{\mathbb{P}})$, associated with $\overline W$. Since the path-wise uniqueness holds, Lemma \ref{strongsolutionlemma} yields that $u_n$ converges in probability in $L^4(0,T;L^2(D))$ and hence a.s. in the original probability space $(\Omega, \mathcal{F},\mathbb{P})$ for a subsequence. Now, repeating the previous arguments in this section with respect to the probability space $(\Omega, \mathcal{F},\mathbb{P})$ completes the proof of Theorem \ref{MTh}.

\section*{Acknowledgment}
This work has been supported by the German Research Foundation project ZI 1542/3-1 and the Procope Plus programme for Project-Related Personal Exchange (France-Germany). The work of Y. Tahraoui was partially  funded by national funds through the FCT - Funda\c c\~ao para a Ci\^encia e a Tecnologia, I.P., under the scope of the projects UIDB/00297/2020 and UIDP/00297/2020 (Center for Mathematics and Applications). G. Vallet acknowledges the ANR project SOS2ID: Stochastic Optimization Schemes - Infinite Dimensional and Inertial Dynamics, France (ANR-24-CE40-3786).


\begin{thebibliography}{1}

\bibitem{BARBU}
V. Barbu.
\newblock \textit{Nonlinear differential equations of monotone types in Banach spaces.}
\newblock Springer, New York, 2010.

\bibitem{Bar1} V. Barbu.
\newblock{The stochastic reflection problem with multiplicative noise.}  
\newblock \textit{Nonlinear Analysis: Theory, Methods $\&$ Applications}, 75(10), 3964-3972, 2012.

\bibitem{Bil99}
P.~Billingsley.
\newblock \textit{ Convergence of {P}robability {M}easures}.
\newblock Wiley Series in Probability and Statistics: Probability and
Statistics. John Wiley \& Sons, Inc., New York, second edition, 1999.


\bibitem{BNSZ22} 
C.~Bauzet,  F.~Nabet,  K.~Schmitz,  A.~Zimmermann.
\newblock Convergence of a finite-volume scheme for a heat equation with a multiplicative Lipschitz noise. 
\newblock \textit{ESAIM: M2AN}, 57(2), 745-783, 2023. 


\bibitem{Bauz}
C. Bauzet, E. Bonetti, G. Bonfanti, F. Lebon and G. Vallet.
\newblock {A global existence and uniqueness result for a stochastic Allen-Cahn equation with constraint.}
\newblock \textit{Mathematical Method in the Applied Sciences}, 40(14), 5241-5261, 2017.

\bibitem{Ben-Ras} A.  Bensoussan and A. Rascanu.
\newblock\textit{Stochastic variational inequalities in infinite dimensional spaces.}  
\newblock Numerical Functional Analysis and Optimization, 18(1-2),  19-54, 1997.


\bibitem{IYG}
I. H. Biswas, Y. Tahraoui and G. Vallet.
\newblock \textit{ Obstacle problem for stochastic conservation law and Lewy  Stampacchia inequality.}
\newblock \textit{J. Math. Anal. Appl.} 527, 127356, 2023.

\bibitem{Breit}
D. Breit, E. Feireisl and M. Hofmanova
\newblock \textit{ Stochastically Forced Compressible Fluid Flows }
\newblock De Gruyter, 2018. 
\bibitem{BrezisLieb}
H. Brezis and E. Lieb
\newblock \textit{ A relation between pointwise convergence functions and convergences of functionals.}
\newblock \textit{ Proc. Amer. Math. Soc.} 88, 486-490, 1983.



\bibitem{Castaing}
Ch. Castaing, P. Raynaud de Fitte and M. Valadier
\newblock \textit {Young Measures on Topological Spaces. With Applications in Control Theory and Probability Theory}
\newblock KLUWER 2004.

\bibitem{DPZ14}
G.~Da~Prato and J.~Zabczyk.
\newblock { Stochastic {E}quations in {I}nfinite {D}imensions}, volume 152
of { Encyclopedia of Mathematics and its Applications}.
\newblock Cambridge University Press, Cambridge, second edition, 2014.

\bibitem{DG-HT}A.~Debussche, N.~Glatt-Holtz and R.~Temam.
\newblock Local martingale and pathwise solutions for an abstract fluids model. 
\newblock \textit{Phys. D}, 14-15, 1123-1144, 2011.


\bibitem{Della-Meyer}
C. Dellacherie, P.A Meyer.
\newblock{ Probabilit\'es et potentiel. Chapitres V-VIII. Th\'eorie des martingales.}
\newblock \textit{Revised edition. Actualit\'es Scientifiques et Industrielles,} 1385. Hermann, Paris 1980.

\bibitem{DenisMatoussiZhang}
L. Denis, A. Matoussi and J. Zhang.
\newblock\textit{ The obstacle problem for quasilinear stochastic PDEs: Analytical approach.}
\newblock The Annals of Probability, 42(3),  865-905, 2014.


\bibitem{DenisMatoussiZhangNonHom}
L. Denis, A. Matoussi and J. Zhang.
\newblock\textit{ The obstacle problem for quasilinear stochastic PDES with non-homogeneous operator. }
\newblock Discrete Contin. Dyn. Syst., 35(11),  5185 - 5202, 2015. 



\bibitem{FLAVIODONATI}
F.~Donati. 
\newblock{A penalty method approach to strong solutions of some nonlinear parabolic unilateral problems.}
\newblock  \textit{Nonlinear Analysis, Th. Meth $\&$ App.}, 6(6), 585-597, 1982.


\bibitem{DonatiMartinPardoux}
C. Donati-Martin and E. Pardoux.
\newblock {White noise driven SPDEs with reflection.}
\newblock \textit{Probab. Theory Related Fields}, 95(1), 1-24, 1993. 

\bibitem{J.L.L}
C. Duvaut and J. L. Lions.
\newblock\textit{Les in\'equations en m\'ecanique et en physique. Travaux et recherches math\'ematiques. }
\newblock vol. 21. Paris: Dunod. XX,   387 p., 1972.






\bibitem{Flan-Gata}
\newblock{F. Flandoli and D. Gatarek. Martingale and stationary solutions for stochastic Navier-Stokes equations.}
\newblock  \textit{ Prob. Theory Relat. Fields}, 102(3), 367-391, 1995.




\bibitem{GigliMosconi}
N.~Gigli  and S.~Mosconi.
\newblock \textit{ The abstract {L}ewy-{S}tampacchia inequality and applications.}
\newblock  J. Math. Pures Appl. 104(2),  258-275, 2015.

\bibitem{GMYV}
 O.~Guib\'e, A.~Mokrane, Y.~Tahraoui and G.~Vallet.
\newblock{{L}ewy-{S}tampacchia's inequality for a pseudomonotone parabolic problem.}
\newblock  \textit{ Adv. Nonlinear Anal.}, 9, 591-612, 2020.



\bibitem{GK96}
I.~Gy{\"o}ngy and N.~Krylov.
\newblock Existence of strong solutions for {I}t\^o's stochastic equations via approximations.
\newblock \textit{ Probab. Theory Relat. Fields}, 105(2), 143-158, 1996.


\bibitem{Han-Joly}
B. Hanouzet and J. L. Joly.
  \newblock{M{\'e}thodes d'ordre dans l'interpr{\'e}tation de certaines in{\'e}quations variationnelles et applications},
  \newblock\textit{Journal of Functional Analysis},
  34(2), 217-249, 1979.



\bibitem{Hof13}
M.~Hofmanov\'{a}.
\newblock{Degenerate parabolic stochastic partial differential equations.}
\newblock\textit{Stochastic Processes and their Applications}, 123(12), 4294-4336, 2013.



\bibitem{Hau-Par} U. G. Haussmann and  E. Pardoux.
\newblock\textit{Stochastic variational inequalities of parabolic type.}
\newblock  E. Appl Math Optim., 20, 163-192,  1989.

\bibitem{Kara}
I. Karatzas and S.E. Shreve. 
\newblock\textit{ Methods of mathematical Finance,}
\newblock  Springer-Verlag, New York. 1998.


\bibitem{L-S} 
H. Lewy and G. Stampacchia.
\newblock{On the smoothness of superharmonics which solve a minimum problem.}  
\newblock \textit{J. Analyse Math.}, 23, 227-236, 1970.

\bibitem{Liu-Rock}
W. Liu  and  M. R\"{o}ckner.
\newblock \textit {Stochastic Partial Differential Equations: An Introduction.}
\newblock Springer International Publishing Switzerland, 2015.


\bibitem{MatoussiLDP}
A. Matoussi,  W. Sabbagh and  T. Zhang. 
\newblock \textit {Large deviation principles of obstacle problems for quasilinear stochastic PDEs.} 
\newblock Applied Mathematics and Optimization, 83(2), 849-879, 2021.

\bibitem{Mok-Mur}A. Mokrane and  F. Murat.
\newblock\textit{ A Proof of the {L}ewy-{S}tampacchia's Inequality by a Penalization Method,}  \newblock Potential Analysis 9,  105-142, 1998. 


\bibitem{Mok-Val} A. Mokrane, Y. Tahraoui  and  G. Vallet.
\newblock\textit{On {L}ewy-{S}tampacchia Inequalities for a Pseudomonotone Elliptic Bilateral Problem in Variable Exponent {S}obolev Spaces.}  
\newblock Mediterr. J. Math., 16(3) paper 64, 2019.

\bibitem{Ondrejat}
M.~Ondrej{\'a}t and J.~Seidler.
\newblock\textit{  A counterexample to the strong {S}korokhod representation theorem.}
\newblock Stochastics and Partial Differential Equations: Analysis and
  Computations, 2025.


\bibitem{Pardoux}
E. Pardoux.
\newblock \textit{Equations aux D\'eriv\'ees Partielles Stochastiques Non Lin\'eaires Monotones.}
\newblock Th\`ese, Paris Sud-Orsay, 1975.


\bibitem{Ras2} A. Rascanu.
\newblock\textit{Parabolic stochastic obstacle problem.}
\newblock {Nonlinear Analysis: Theory, Methods $\&$ Applications}, 11(1), 71-83, 1987.


\bibitem{Rascanu}
A. Rascanu and E. Rotenstein.
\newblock{Obstacle problems for parabolic SDEs with  {H}\"older continuous diffusion: From weak to strong solutions.}
\newblock  \textit{J. Math. Anal. Appl.}, 450, 647-669, 2017.


\bibitem{RS06}
M.M. Rao and R.J. Swift.
\newblock { Probability {T}heory with {A}pplications}.
\newblock Springer New York, NY, 2006.


\bibitem{J.F} J.~F.~Rodrigues.
\newblock\textit{Obstacle problems in mathematical physics, volume 134 of North-Holland Mathematics Studies,}.  \newblock Notas de Matematica [Mathematical Notes], 114, Amsterdam, 1987.


\bibitem{Rodrigues} J.~F.~Rodrigues.
\newblock\textit{On the hyperbolic obstacle problem of first order},
\newblock Chinese Ann. Math. Ser. B, 23(2), 253 -- 266,  2002.


 \bibitem{Roubicek}
T. Roubi{\v{c}}ek.
\newblock \textit {Nonlinear Partial Differential Equations with Applications.}
\newblock Birkh{\"a}user Verlag Basel-Boston-Berlin, 2005.


\bibitem{Simon}
J. Simon.
\newblock \textit {Compact sets in the space $L^p(0,T; B)$.}
\newblock Ann. Mat. Pur. App, 146, 65-96, 1987.

\bibitem{STVZ25}
N. Sapountzoglou, Y. Tahraoui,  G. Vallet and A. Zimmermann. A \newblock \textit {nonlinear stochastic diffusion-convection equation with reflection.} \newblock \url{arXiv:2412.16413}, 2024


\bibitem{Yas}
Y. Tahraoui. \newblock\textit{Parabolic problems with constraints, deterministic and stochastic.} \newblock  PhD thesis. 2020. \url{https://theses.hal.science/tel-03126849/document}.
  
  \bibitem{YV}
  Y. Tahraoui and G. Vallet.
  \newblock\textit{Lewy-Stampacchia's inequality for a stochastic T-monotone obstacle problem.}
  \newblock Stoch PDE: Anal Comp 10, 90–125, 2022.
  %

\bibitem{YassLDP}
 Y.  Tahraoui.
 \newblock\textit{Large deviations for an obstacle problem with T-monotone operator and multiplicative noise.}
\newblock Potential Anal, 1-29, 2025.


\bibitem{YassInvarinat}
Y. Tahraoui.
 \newblock\textit{Invariant measures for stochastic T-monotone parabolic obstacle problems.}
 \newblock \url{arXiv:2311.02637}, 2023.


















\bibitem{Vakh}
N. N. Vakhania, V. I. Tarieladze and S. A. Chobanyan.
\newblock \textit{Probability Distributions on Banach Spaces.}
\newblock D. Reidel Publishing Company, 1987.


\bibitem{Vallet-Zimm2}
G. Vallet and  A. Zimmermann.
 \newblock Well-posedness for a pseudomonotone evolution problem with multiplicative noise.
 \newblock \textit{J. Evol. Equ.}, 19, 153–202, 2019. 

\bibitem{Vallet-Zimm1}
G. Vallet and   A. Zimmermann. 
\newblock Well-posedness for nonlinear SPDEs with strongly continuous perturbation.
\newblock \textit{Proceedings of the Royal Society of Edinburgh: Section A Mathematics}, 151(1), 265-295, 2021.  















\end{thebibliography}
\end{document}